\documentclass[11pt]{amsart}
\usepackage{amsmath,amssymb}%,pb-diagram}
\usepackage{graphicx}
\usepackage{amsfonts,amscd,latexsym,epsfig, oldgerm,psfrag,mathrsfs}
\usepackage[all]{xy}

\newcommand{\Fr}{{\rm Fr}}
\newcommand{\cl}{{cl}}
\newcommand{\ol}{{\ov{\mathcal{OL}}}}

\newcommand{\conv}{{\rm conv}}
\newcommand{\ve}{{\und{\eps}}}
\newcommand{\vd}{{\und{\de}}}
\newcommand{\er}{{\Diamond}}

\newcommand{\uY}{{\ul{Y}}}

\newcommand{\Sss}{{\mathscr S}}

\newcommand{\st}{{\rm st}}

\newcommand{\s}{{\mathfrak s}}

\newcommand{\TOo}{{\Ti \Oo}}

\newcommand{\ul}{\underline}

\newcommand{\Obj}{{\rm Obj}}

\newcommand{\Ti}{\widetilde}

\newcommand{\pr}{{\rm pr}}
\newcommand{\bB}{{\bf B}}

\newcommand{\bG}{{\bf G}}
\newcommand{\bM}{{\bf M}}

\newcommand{\bX}{{\bf X}}
\newcommand{\bW}{{\bf W}}
\newcommand{\bZ}{{\bf Z}}

      \newcommand{\upsi}{{\underline{\psi}}}

\newcommand{\rT}{{\rm T}}
\newcommand{\rd}{{\rm d}}

\newcommand{\oo}{{\mathfrak o}}
\newcommand{\Mor}{{\rm Mor}}

\newcommand{\intt}{{\rm int\,}}

\newcommand{\und}{\underline}
\renewcommand{\Hat}{\widehat}

\newcommand{\Cc}{{\mathcal C}}
\newcommand{\Dd}{{\mathcal D}}
\newcommand{\Kk}{{\mathcal K}}
\newcommand{\Ff}{{\mathcal F}}

\newcommand{\im}{{\rm im\,}}
\newcommand{\less}{{\smallsetminus}}

\newcommand{\supp}{{\rm supp\,}}
\newcommand{\TU}{{\Tilde U}}
\newcommand{\p}{{\partial}}
\newcommand{\al}{{\alpha}}

\newcommand{\be}{{\beta}}

\newcommand{\Om}{{\Omega}}
\newcommand{\om}{{\omega}}
\newcommand{\eps}{{\varepsilon}}
\newcommand{\de}{{\delta}}
\newcommand{\De}{{\Delta}}
\newcommand{\ga}{{\gamma}}
\newcommand{\Ga}{{\Gamma}}
\newcommand{\io}{{\iota}}
\newcommand{\ka}{{\kappa}}
\newcommand{\la}{{\lambda}}
\newcommand{\La}{{\Lambda}}
\newcommand{\si}{{\sigma}}

\newcommand{\Gg}{{\mathcal G}}

\newcommand{\Uu}{{\mathcal U}}

\newcommand{\Ww}{{\mathcal W}}

\newcommand{\Oo}{{\mathcal O}}

\newcommand{\Ss}{{\mathcal S}}

\newcommand{\ov}{\overline}

\newcommand{\id}{{\rm id}}

\renewcommand{\Tilde}{\widetilde}
\newcommand{\TM}{{\Tilde M}}
\newcommand{\TV}{{\Tilde V}}
\newcommand{\TW}{{\Tilde W}}

\newcommand{\Ee}{{\mathcal E}}
\newcommand{\Ii}{{\mathcal I}}

\newcommand{\Vv}{{\mathcal V}}

\newcommand{\Q}{{\mathbb Q}}
\newcommand{\R}{{\mathbb R}}
\newcommand{\C}{{\mathbb C}}

\newcommand{\Z}{{\mathbb Z}}

\newcommand{\Hom}{{\rm Hom}}

\newcommand{\Nn}{{\mathcal N}}
\newcommand{\Pp}{{\mathcal P}}
\newcommand{\Hh}{{\mathcal H}}

\newcommand{\SSS}{{\smallskip}}

\newcommand{\bE}{{\bf E}}
\newcommand{\bK}{{\bf K}}

\newtheorem{theorem}{Theorem}[subsection]
\newtheorem{thm}[theorem]{Theorem}

\newtheorem{cor}[theorem]{Corollary}
\newtheorem{lemma}[theorem]{Lemma}

\newtheorem{prop}[theorem]{Proposition}
\newtheorem{defn}[theorem]{Definition}
\newtheorem{example}[theorem]{Example}

\newtheorem{rmk}[theorem]{Remark}

\numberwithin{figure}{section}
\numberwithin{equation}{subsection}
\newcommand{\MS}{{\medskip}}

\newcommand{\NI}{{\noindent}}

   \newcounter{qcounter}

\newenvironment{itemlist}
   { \begin{list} {$\bullet$}
         {  \setlength{\itemsep}{.5ex} \setlength{\leftmargin}{2.5ex} } }
   { \end{list} }

\newcommand\quotient[2]{
        \mathchoice
            {% \displaystyle
                \text{\raise1ex\hbox{$#1$}\Big/\lower1ex\hbox{$#2$}}%
            }
            {% \textstyle
                #1\,/\,#2
            }
            {% \scriptstyle
                #1\,/\,#2
            }
            {% \scriptscriptstyle
                #1\,/\,#2
            }
    }
\newcommand\quot[2]{
                \text{\raise1ex\hbox{$#1$}/\lower1ex\hbox{$\scriptstyle#2$}}
  }

\newcommand\qu[2]{
                \text{\raise.8ex\hbox{$\scriptstyle#1\!$}/\lower.8ex\hbox{$\!\scriptstyle#2$}}
  }

\newcommand\ql[2]{
                \text{\lower.6ex\hbox{$\scriptstyle#1\!$}$\backslash$ \raise.6ex\hbox{$\!\scriptstyle#2$}}
  }

\newcommand\qq[2]{
                \text{\raise.8ex\hbox{$#1\!$}/\lower.8ex\hbox{$#2$}}
}
 \title{
Constructing the virtual fundamental class of a Kuranishi atlas}
\author{Dusa McDuff}
 \address{Department of Mathematics,
 Barnard College, Columbia University}
\email{dusa@math.columbia.edu}
\thanks{partially supported by NSF grant DMS 1308669}

\date{October 4, 2017, revised September 2, 2018}

\begin{document}
\begin{abstract}
Consider a  space $X$, such as a compact space of $J$-holomorphic stable maps, that is the zero set of a Kuranishi atlas.  
This note explains how to define the  virtual fundamental class of $X$ by
representing $X$ via the zero set of a map $\Sss_M: M\to E$, where $E$ is  a finite dimensional  vector space and 
the domain $M$ is an oriented, weighted branched topological manifold.  Moreover, $\Sss_M$ 
 is equivariant  under the action of the global isotropy group $\Ga$ on $M$ and $E$.  This tuple $(M,E, \Ga, \Sss_M)$ together with a homeomorphism from
 $\Sss_M^{-1}(0)/\Ga$ to $X$  forms a single finite dimensional model (or chart) for $X$.  The construction assumes only  that the
 atlas satisfies a topological version of the index condition that can be obtained from a standard, rather than a smooth, gluing theorem.  However if $X$ is presented as the zero set of an sc-Fredholm operator on a strong polyfold bundle, we outline a much more direct construction of the branched manifold $M$ that uses an sc-smooth partition of unity.
\end{abstract}

\maketitle

\tableofcontents
\section{Introduction}\label{s:main}
\subsection{Statement of main results}
Let $X$ be a compact space that is locally the zero set of a Fredholm operator $\Ff$ of index $d$, such as 
a moduli space of $J$-holomorphic stable curves.
The question  of how to define its fundamental class is central to
  symplectic geometry, since so much information about the properties of this geometry
   depends on the ability to \lq count' the number of elements in $X$.
There are many possible approaches to this problem, e.g. \cite{FO,TF, HWZ, H:sur}.  In this note we develop the work of McDuff--Wehrheim~\cite{MW1,MW2,MWiso} and Pardon~\cite{P} that uses atlases, in an attempt to clarify the passage from atlas to
 virtual fundamental class.
 
 A $d$-dimensional atlas consists of a family of charts $\bK_I$ indexed by subsets $I\subset \{1,\dots, N\}=: A$, together with coordinate changes $\Hat\Phi_{IJ}$ for $I\subset J$,
where the chart $\bK_I$ is a tuple
$$
\bK_I = (U_I, E_I, \Ga_I, s_I, \psi_I),
$$
consisting of a manifold $U_I$ of dimension $d + \dim E_I$, a real vector space $E_I$, actions of the group $\Ga_I$ on $U_I$ and on $E_I$, a $\Ga_I$-equivariant map $s_I: U_I\to E_I$, and finally the footprint map $\psi_I: s_I^{-1}(0)\to X$ that induces a homeomorphism
from $(s_I^{-1}(0))/\Ga_I$ onto an open subset $F_I$ of $X$. The charts $\bK_i$ that are indexed by sets $\{i\}$  of length one are called basic charts, and we assume that their footprints $(F_i)_{1\le i\le N}$ cover $X$, while the other charts $\bK_I$ with $|I|>1$  form transition data.   In applications, the corresponding vector spaces $E_i$ cover the cokernel of the Fredholm operator $\Ff$ at the points in the footprint $F_i\subset X$, and are called obstruction spaces because they obstruct the existence of solutions when $\Ff$ is deformed.  The essence of the problem lies in trying to assemble these local finite dimensional models for $X$ into one structure that retains enough information to determine its fundamental class, which (when $d=0$) one can think of as the number of solutions of a \lq\lq generic" perturbation of $\Ff$.\MS

  The paper \cite{MWiso} explains one way
 to use a $d$-dimensional oriented atlas 
to define a \c{C}ech homology class $[X]^{vir}_\Kk\in \check{H}_d(X;\Q)$.   Roughly speaking, the idea is this.
Using the  coordinate changes to identify different domains, one
  constructs a metrizable, Hausdorff space  $|\Kk| = \bigcup_I U_I/\!\!\!\sim$ that supports a (generalized) orbibundle $|\bE_\Kk|\to |\Kk|$  with a canonical section $|\s|: |\Kk|\to |\bE_\Kk|$ together with a natural identification  $$
  \io_X: X\stackrel{\cong}\to |\s|^{-1}(0).
  $$  With some difficulty, one then defines a multi-valued perturbation section $|\nu|: |\Vv|\to |\bE_\Kk|$ on a subset $|\Vv|\subset |\Kk|$, 
  such that  $|\s+\nu|$ is transverse to $0$. Finally, one shows that 
the perturbed zero set  $| \s+\nu|^{-1}(0)$ represents a unique element in $ \check{H}_d(X;\Q)$. 
\MS

Because it uses the notion of transversality, the above construction  requires that the atlas have some smoothness properties.\footnote
{
\, See \cite{Castell1,Castell2} for a weak form of these requirements.} 
In particular, the transition maps between charts must satisfy the 
so-called tangent bundle (or index) condition.
On the other hand, Pardon~\cite{P} introduces a new way to extract topological information from an atlas that 
satisfies a topological version of this condition that he calls the  submersion axiom.  Instead of gluing the chart domains together to form a topological space $|\Kk|$, Pardon works with $K$-homotopy sheaves of (co)chain complexes defined on homotopy colimits of spaces that are obtained from the chart domains. 
This gives a flexible way of assembling  local homological information  into a global object.
 Though this approach may be useful in many contexts,  it is hard for a nonexpert in sheaf theory to understand where the technical difficulties are, and what actually has to be checked to ensure that the
method works in any particular case.  This  becomes  an issue if one wants to extend the method to cases (such as  Hamiltonian Floer theory, or symplectic field theory) in which one must deal with a family of related moduli spaces and so should work on the chain level.  The current paper was prompted by the desire to develop a different approach, that would replace Pardon's sophisticated sheaf theory by more elementary arguments that yet do not require smoothness. 

This note only considers the simplest case, appropriate to Gromov--Witten theory, in which the aim is to construct a homology class $[X]^{vir}_\Kk\in \check{H}_d(X;\Q)$.  Working with Pardon's  submersion axiom, we 
 define a consistent thickening of the domains of the atlas charts to make them all have the same dimension.
In the case with trivial isotropy, one thereby constructs an oriented topological manifold $M$ of dimension $D: = d + \dim E_A$, together with a map $\Sss_M :M\to E_A$ whose zero set can  be identified with $X$.
%It is then a trivial matter to define $[X]^{vir}_\Kk\in \check{H}_d(X;\Z)$. 
If the isotropy is nontrivial,  $M$ is a 
%D weighted 
branched manifold with a weighting function $\La$ and  a global action of the total isotropy group $\Ga_A$, and there is a homeomorphism
$\Sss_M^{-1}(0)/\Ga_A\stackrel{\cong}\to X$.\footnote
{
\, Another way to say this is that $M: = |\Hat\bM|_\Hh$ is the Hausdorff realization of a topological groupoid $\Hat\bM$ that is \'etale but not proper: see  \S\ref{ss:defs}, \S\ref{ss:br} for relevant definitions.  However, just as in the case of the construction of the zero set in \cite{MWiso}, it is most natural to construct a topological category $\bM$ in which not all morphisms are invertible, i.e. it is a monoid, rather than a groupoid.  
%Note also that if one adds morphisms to this category corresponding to the action of $\Ga_A$ then one obtains 
%an ep groupoid, i.e. an orbifold.
 }
 (A typical example of such a manifold $(M,\La)$ is the union of two circles, each of weight $\frac 12$, identified along a closed subarc $A$, so that the points $x\in A$ have weight $\La(x) = 1$, while the others all have weight $\La(x) =\frac12$. See also \S\ref{ss:exam}.) 

Here is the first main result.  (See Theorem~\ref{thm:M}  for a more precise statement.) 
\MS

\NI{\bf Theorem A:}  {\it Let $\Kk$ be a   $d$-dimensional Kuranishi atlas on a compact  space $X$ that satisfies the submersion condition \eqref{eq:subm0} and has
 total obstruction space $E_A: =\prod _{i\in A} E_i$ and total isotropy group
$\Ga_A: = \prod _{i\in A} \Ga_i$.   Let $D = d+\dim E_A$.  Then there is an associated 
weighted branched $D$-dimensional manifold $(M,\La)$  with an action of  $\Ga_A$, and a $\Ga_A$-equivariant map $\Sss_M: M\to E_A$ 
with a compact zero set $\Sss_M^{-1}(0)$.  Moreover, there is a map
 $\psi: \Sss_M^{-1}(0)\to X$ that induces a homeomorphism $\Sss_M^{-1}(0)/\Ga_A\stackrel{\cong}\to X$.   }
\MS 

It is immediate from the construction that the  bordism class of a neighborhood of $\Sss_M^{-1}(0)$ in $M$ depends only on the concordance class  of $\Kk$.\footnote
{
\, Two atlases $\Kk^0,\Kk^1$ on $X$ are said to be concordant if there is an atlas $\Kk^{01}$ on $[0,1]\times X$ whose restriction to 
$\{\al\}\times X$ is $\Kk^\al$, for $\al = 0,1$:  see   \cite[Def.~4.1.6]{MW1}. Note also that as here, when there is no danger of confusion, we often abbreviate   
 \lq Kuranishi atlas' to \lq atlas'.
}
Further,  if $\Kk$ and hence $(M,\La)$ is oriented, we show in Lemma~\ref{le:fundM} that $(M,\La)$ carries a  fundamental class $\mu_M$ in rational \c{C}ech homology $\check{H}_*$.  Hence we have the following.
\MS

\NI{\bf Theorem B:}  {\it If $\Kk$ is an oriented atlas on $X$ as above,  there is a unique element $[X]^{vir}_\Kk\in \check{H}_d(X;\Q)$ that is defined as follows.  
 For $b\in \check{H}^d(X;\Q)$ and $D = d+\dim E_A$, we have
\begin{align} \label{eq:Xvir}
&\langle [X]^{vir}_\Kk,\ b \rangle : = (\Sss_M)_*( \Hat{b})\in \check{H}_{\dim E_A}(E_A, E_A\less \{0\};\Q) \cong \Q,
\end{align}
where $\Hat{b}$ is the image of $b$ under the composite % the Alexander--Lefschetz duality isomorphism
$$
 \check{H}^d(X;\Q)\stackrel{\psi^*}\longrightarrow  \check{H}^d(\Sss_M^{-1}(0);\Q)\stackrel{\Dd}\longrightarrow \check{H}_{\dim E_A} (M, M\less \Sss_M^{-1}(0); \Q),
 $$
and $\Dd$ is given by cap product with the fundamental class $\mu_M$.  Moreover, $[X]^{vir}_\Kk$ depends only on the oriented  concordance class of $\Kk$, and in the smooth case agrees with the class defined in \cite{MWiso}.
}
\MS

A key element of the proof of Theorem~A is Pardon's notion of {\it deformation to the normal cone}, which allows one to assemble different chart domains into a family of topological  manifolds $Y_J$, albeit ones of the  wrong dimension: see Proposition~\ref{prop:Y}.  
 The second key point is the existence of compatible  collars for these manifolds $Y_J$. Remark~\ref{rmk:out} outlines the proof in more detail.
 
 As we explain in Remark~\ref{rmk:sm2}, if we start with a smooth atlas then the proofs of the above results can be somewhat simplified. In particular, by \cite{Mbr} we can construct  $M$ to be a simplicial complex so that there is no need to use so much rational \c{C}ech homology when proving Theorem~B. 
Further, if one works with polyfolds, then  the proof can be radically simplified.
Indeed, it is not difficult 
to define a smooth Kuranishi atlas on any space $X$
that appears as the (compact) zero set of  a polyfold bundle~\cite{HWZ,H:sur,Yang,MWgw}.  Because the polyfolds of Gromov--Witten theory support sc-smooth partitions of unity, if the isotropy is trivial, one can even define such an atlas with just one chart.   In other words, one obtains a finite dimensional model 
$$
(U, \R^N, s, \psi),  \quad \psi: s^{-1}(0)\stackrel{\cong}\longrightarrow X,
$$
for the whole of $X$, in which $U$ is a smooth manifold of dimension $d+ N$ and $s:U\to \R^N$ is a smooth map.  
As we show in Remark~\ref{rmk:poly}  this construction can be adapted in  the presence  of isotropy.  However,
the domain of the single chart is no longer a manifold, but a branched manifold with  action of the total isotropy group $\Ga_A$.

Another simple example is the calculation of the Euler class of an oriented vector bundle $\pi:\Ee\to X$  of rank $2k$ over a compact manifold  $X$. If $\Ee'\to X$ is an oriented complement to $\Ee$  of rank $2\ell$  so that there is a vector bundle isomorphism $\phi: \Ee\oplus \Ee' \cong \R^N\times X$, where $N = 2k+2\ell$, let 
\begin{align}\label{eq:Euler}
M=\Ee',&\qquad \Sss: M\to \R^N, \;\; (e',x)\mapsto \pr_{\R^N} \bigl(\phi(e',x)\bigr).
\end{align}
Then $\Sss^{-1}(0) \cong X$, and it is easy to check that the class $[{X}]^{vir}_\Kk$ defined by \eqref{eq:Xvir} is Poincar\'e dual to the Euler class of $\Ee\to X$: see 
Lemma~\ref{le:Euler}.  This is an instance of the construction in Pardon~\cite[Defn.~5.3.1]{P} for the bundle $\pi:\Ee\to X$ with section $\s\equiv 0$
in which the thickening $\la: \R^N\times X\to \Ee'$ is given by the projection.

Finally note that the methods of this paper should extend, e.g. to a more general notion of atlas, or to 
spaces more general than topological manifolds: see Remark~\ref{rmk:genl}.

\subsection{Basic definitions and facts about atlases}  \label{ss:defs}

A {\bf weak Kuranishi atlas} $\Kk$ of dimension $d$ on a compact metrizable space $X$ consists of the following data.\footnote
{
\, These are essentially the same definitions as in \cite{MWiso},  except that the smoothness requirements mentioned in Remark~\ref{rmk:submer}~(ii) below have been replaced by an equivariant version of Pardon's submersion axiom.   
The notion of topological atlas introduced in \cite{MW1} is somewhat different; in particular the domains there need not be manifolds.
For more details on all topics mentioned in this section, see the original papers \cite{MW1,MW2,MWiso} or \cite{Mnot}.
}
\begin{itemlist}
\item   {\bf (footprint cover)} a finite  open cover of $X$ by nonempty sets $(F_i)_{i\in A}$;
\item a {\bf poset} $\Ii_\Kk =  \bigl\{I\subset A \ \big| \ F_I: = \bigcap_{i\in I} F_i \ne \emptyset\bigr\}$ that indexes the charts;
\item {\bf (charts)} $\forall I\in \Ii_\Kk$,  $F_I$ is the footprint of a chart $\bK_I: = (U_I,\Ga_I, E_I, s_I,\psi_I)$, where
\begin{itemize}\item[-]  $U_I$ is a finite dimensional topological manifold of dimension $d + \dim E_I$; 
\item[-] $E_I: = \prod_{i\in I} E_i$ is a product of even dimensional\footnote
{
\, For simplicity, we assume that $E_i$ is even dimensional so that the orientation of a product of the $E_i$ does not depend on their order.  In the Gromov--Witten situation we may always choose the $E_i$ to have a natural complex structure since the target of the linearized Cauchy--Riemann operator is a complex vector space of $(0,1)$-forms.
}
 vector spaces such that $\dim U_I - \dim E_I = d$;
\item[-] $\Ga_I = \prod_{i\in I} \Ga_i$ is a product of finite groups that acts on $U_I$, and acts by a product of linear actions  on $E_I$;
\item[-] $s_I: U_I\to E_I$ is a $\Ga_I$-equivariant map;
\item[-] the footprint map $\psi_I: s_I^{-1}(0) \to X$ induces a homeomorphism 
\begin{align}\label{eq:foot0}
\qu{s_I^{-1}(0)}{\Ga_I}\stackrel{\cong}\longrightarrow  F_I;
\end{align}
\end{itemize}
\item {\bf (coordinate changes)} if  $I\subset J$   there is a coordinate change $\Hat\Phi_{IJ}:  \bK_I\to \bK_J$ given by the following data,
where we identify $E_I$ as a subspace of $E_J$ in the obvious way:
 \begin{itemize}\item[-] 
 a relatively  open, $\Ga_J$-invariant subset $\TU_{IJ}$ of $s_{J}^{-1} (E_I)\subset U_J$ containing $s_J^{-1}(0)$ and with a free action of $\Ga_{J\less I}$, 
 \item[-] 
a covering map $\rho_{IJ}: \TU_{IJ}\to U_I$ that quotients out by the (free) action of $\Ga_{J\less I}$ and is equivariant with respect to the projection $\Ga_J\to \Ga_I$, further 
\begin{align*}%\label{eq:foot2}
& s_I\circ \rho_{IJ} = s_J, \quad  \psi_J  = \psi_I\circ \rho_{IJ}  \mbox { on }  s_J^{-1}(0)\subset \TU_{IJ},
 \end{align*}
\item[-]   if $I\subset J \subset K$, then 
\begin{align}\label{eq:cocyl} 
& \rho_{IK} = \rho_{IJ}\circ \rho_{JK} \mbox { whenever  both sides are defined}
\end{align}
\item[-]   in an {\bf atlas} (rather than a weak atlas) we require in addition that the domain $\rho_{JK}^{-1}(\TU_{IJ})\cap \TU_{JK}$
of $ \rho_{JK}\circ \rho_{IJ}$  is a subset of the domain $\TU_{IK}$ of $\rho_{IK}$.  
\item[-]   in a {\bf tame
atlas} we require that both sides of \eqref{eq:cocyl} have the same domain and that
$\TU_{IJ} = s_J^{-1}(E_I)$.
\end{itemize}
%where above we have identified $E_I$ as a subspace of $E_J$ in the obvious way.
\item {\bf (equivariant submersion condition)}\  for each $I\subset J$,  each point $x\in \TU_{IJ}\subset U_J$ has a product neighborhood that is compatible with the section $s_{J\less I}$;  more precisely for each such $x$ with stabilizer subgroup
$\Ga_x\subset \Ga_I$, there is a $\Ga_x$-equivariant local homeomorphism
of the form
\begin{align}\label{eq:subm0}
\phi_x^E: (E_{J\less I,\de}\times W_x,  \{0\}\times W_x) \to (U_J, \TU_{IJ}),
\end{align}
 where $E_{J\less I,\de}$ is a $\de$-neighborhood of $0$ in $E_{J\less I}$  and $W_x$ is a $\Ga_x$-invariant neighborhood of $x$ in $\TU_{IJ}$,
 such that
%where $W_x\subset  \TV_{IJ}$ that is compatible with the section $s_J$; 
\begin{align*}%\label{eq:subm1}
&s_{J\less I}\circ \phi_x^E(e,y) =  e,\qquad e\in E_{J\less I,\de}.
\end{align*}
\end{itemlist}

\begin{rmk}\rm \label{rmk:submer} (i)  Although the  submersion axiom in \cite{P} does not assume equivariance, this is needed in our set-up  in order that $M$ support an action of $\Ga_A$.   Notice that because $\Ga_{J\less I}$ acts freely on $\TU_{IJ}$, the stabilizer $\Ga_x$ of $x\in \TU_{IJ}$ lies in the subgroup $\Ga_I$ of $\Ga_J\cong \Ga_I\times \Ga_{J\less I}$.  The standard proof of  the submersion axiom for Gromov--Witten moduli spaces  adapts easily to yield $\Ga_x$-equivariance
because it is an application of the gluing theorem at the stable map $x$.      The process of gluing depends on various choices, for example of Riemannian metrics and of the 
complement to the image of the linearized Cauchy--Riemann operator at $x$, and  these can always be chosen  invariant under the finite stabilizer subgroup of $x$.  This equivariance is built into the smooth index condition, since the latter is expressed in terms of the equivariant section maps $s_{J\less I}$. \MS 

\NI (ii)   {\bf (The smooth case)}\
In this case the manifolds $U_I$ are assumed to be smooth, all structural maps (the group action on $U_I$, the section $s_I$, and coordinate changes $\rho_{IJ}$) are smooth, and  the submersion axiom is replaced
by % also respects the smooth structures.  This assumption is equivalent to %usually expressed by 
the requirement that $\TU_{IJ}$ be a submanifold of $U_J$
 such that 
\begin{align}\label{eq:prod2}& \mbox{the derivative of $s_{J\less I}:U_J\to E_{J\less I}$  induce  an isomorphism}\\  \notag& \mbox{from the %}\\ \notag
%&&\quad\quad \mbox{
normal bundle of $\TU_{IJ}$ in $U_J$ to $E_{J\less I}\times \TU_{IJ}$.}
\end{align} 
%Since $s_{J\less I}$ is $\Ga_{J\less I}$-equivariant, the equivariance requirement is now automatic.  
In this case we claim that each of the maps $\tau_{IJ}$ in Proposition~\ref{prop:M} can be chosen to be a local diffeomorphism onto its image,  so that  $M$ is a smooth manifold
if the isotropy is trivial, and otherwise is a smooth branched manifold. The construction of such $M$ is sketched in Remark~\ref{rmk:sm2}. 
\MS

\NI (iii)   {\bf (Orientations)}\  We will consider an atlas to be oriented if each domain $U_I$ (resp. each obstruction space $E_I$) has a
$\Ga_I$-invariant  orientation that is 
 respected by the coordinate changes.   In fact, in the current situation, since we have assumed that the $E_i$ are all even dimensional (e.g. that they are all complex vector spaces), then  if they are also  invariantly oriented the $E_I$ inherit natural orientations, and  the local product structure given by the submersion condition permits the transfer of an orientation between charts.    In the smooth case, a slightly more general notion of orientation is discussed extensively
in \cite{MW2,MWiso}.  \hfill$\er$
\end{rmk}

 We now briefly recall some other terminology that will be useful later.
An atlas $\Kk' = \bigl(\bK_I', \Hat\Phi_{IJ}'\bigr)$ is a {\bf shrinking} of $\Kk= \bigl(\bK_I, \Hat\Phi_{IJ}\bigr)$ if 
 \begin{itemize}\item[-]  it has the same 
index set $\Ii_\Kk$, obstruction spaces $E_I$ and groups $\Ga_I$, 
\item[-] 
 each chart domain $U_I'$ is a precompact subset of $U_I$, denoted $U_I'\sqsubset U_I$ ,
\item[-]  the coordinate changes are given by restriction.
 \end{itemize}
 For short, in this situation we write
 \begin{align}\label{eq:Uu}
 \Uu'\sqsubset \Uu, \;\; \mbox{ where }\;\;  \Uu': ={\textstyle  \bigsqcup_{I\in \Ii_\Kk}  U'_I, \;\; \Uu: = \bigsqcup_{I\in \Ii_\Kk}  U_I}.
 \end{align}
It is shown in \cite[\S3.3]{MW1} and \cite[\S2.5]{MWiso}  that every weak atlas has a tame shrinking $\Kk'\sqsubset \Kk$ that is unique up to a natural equivalence relation  called concordance.  A tame atlas $\Kk$ is called {\bf preshrunk} if  there is a double shrinking $\Kk\sqsubset \Kk'\sqsubset \Kk''$ such that both $\Kk$ and $\Kk'$ are tame.

Each atlas\footnote
{
\, The extra assumption in the definition of atlas stated just after \eqref{eq:cocyl} implies that the set $\Mor_{\bB_\Kk} $ defined below  is closed under composition.
}
 $\Kk$ determines a topological category $\bB_\Kk$ with 
\begin{align}\label{eq:Kcat}
&{\textstyle  \Obj_{\bB_\Kk} = \bigsqcup_{I\in \Ii_\Kk} U_I,  \quad \Mor_{\bB_\Kk} = \bigsqcup_{I\subset J} \TU_{IJ}\times \Ga_I,}
\vspace{.07in}\\ \notag
&  \quad s\times t:   \Mor_{\bB_\Kk} \to  \Obj_{\bB_\Kk} \times  \Obj_{\bB_\Kk} ,\;\; \\ \notag
& \quad (I,J,y,\ga)\mapsto \bigl( (I, \ga^{-1} (\rho_{IJ} (x)), \, (J,y)\bigr).
\end{align}
We denote by $|\Kk|: = |\bB_\Kk|$ its (geometric or naive) realization.  Thus
$$
{\textstyle   |\bB_\Kk|: = \bigsqcup_I U_I/\!\sim, }
$$
where $\sim$ is the equivalence relation on $ \Obj_{\bB_\Kk}$ that is generated by the morphisms, i.e.
$
(I,x) \sim (J,y)$ if and only if there is a chain of morphisms
$$
(I,x) = (I_0, x_1) \to (I_1,x_1)\leftarrow (I_2,x_2) \to \dots  \leftarrow (I_k,x_k) = (J,y).
% \Longleftrightarrow \exists %\Mor_{\bB} \bigl((I,x) , (J,y)\bigr) \ne \emptyset.
$$
Though for a general atlas the 	quotient topology is nonHausdorff,  it is shown in \cite[Thm~3.1.9]{MW1} (see also \cite[\S2.5]{MWiso}) that if $\Kk$ is preshrunk and tame the quotient topology is Hausdorff and the natural maps
\begin{align}\label{eq:piK}
\pi_\Kk:U_I\to |\Kk|
\end{align}
induce  homeomorphisms from $U_I/\Ga_I$ onto their images.  Further, 
%if $\Kk$ is also {\bf preshrunk} (i.e. there is a double shrinking $\Kk\sqsubset \Kk'\sqsubset \Kk''$ where both $\Kk$ and $\Kk'$ are tame), 
the quotient topology on $|\Kk'|$ restricts to a metrizable topology on $|\Kk|$ that agrees with the quotient topology on each set $\pi_\Kk(U_I)$.    We will say that $\Kk$ is {\bf good}  if its realization $|\Kk|$ has these  properties.\footnote{
\, The proof given in \cite{MW1} that  preshrunk and tame atlases are good is abstract, i.e. the argument only uses properties of the 
objects and maps in the category $\bB_\Kk$.  However, because  the atlas domains are often constructed as subsets  of an ambient  Hausdorff metrizable space  $\Ss$ (such as a space of  stable maps),  one can sometimes   use the existence of $\Ss$ to
 bypass some of the arguments in \cite{MW1}.}

\NI{\it From now on we assume that $\Kk$ is good in this sense, e.g. preshrunk and tame.}
\MS

There is a similar category $\bE_\Kk$ formed by the obstruction bundles with 
\begin{align*}%\label{eq:KEcat}
&{\textstyle  \Obj_{\bE_\Kk} = \bigsqcup_{I\in \Ii_\Kk} U_I\times E_I,  \quad \Mor_{\bB_\Kk} = \bigsqcup_{I\subset J} \TU_{IJ}\times E_I\times  \Ga_I,}
\vspace{.07in}\\ \notag
&  \quad s\times t:   \Mor_{\bE_\Kk} \to  \Obj_{\bE_\Kk} \times  \Obj_{\bE_\Kk} ,\;\; \\ \notag
& \quad \bigl(I,J,y,e,\ga\bigr)\mapsto \Bigl( \bigl(I, \ga^{-1}( \rho_{IJ} (y), e)\bigr), \, \bigl(J,y, e\bigr)\Bigr).
\end{align*}
The projections $\pr_I: U_I\times E_I\to U_I$, sections $s_I$ and footprint maps $\psi_I$ fit together to give functors
$$
\pr:  \bE_\Kk \to \bB_\Kk, \quad  \s: \bB_\Kk \to \bE_\Kk, \quad \psi: \s^{-1}(0) \to \bX,
$$
where $\bX$ is the category with objects $X$ and only identity morphisms, and one can show that
$\psi$ induces a homeomorphism $|\psi|: |\s|^{-1}(0) \to X$.

\MS

\NI {\bf Reductions and  zero sets}

The situation when all the obstruction spaces $E_I$ vanish is considered in \cite{Morb}. In this case, the category $\bB_\Kk$ is
\begin{itemize}\item  {\bf \'etale},  i.e. the object and morphism spaces are manifolds, and the  source and target maps are local homeomorphisms, and 
\item {\bf  proper}, i.e. the equivalence relation $\sim$ on the object space generated by the morphisms is closed.\footnote
{\,
If $\Obj_\bB$ is a separable, locally compact, metric space (as is the case for the categories considered in this paper), then this properness condition implies that 
the realization $|\bB|$ is Hausdorff; for a proof see \cite[Lemma~3.2.4]{MW1}.  If in addition $\bB$ is a groupoid, then this condition is equivalent to the more standard requirement that 
the map $s\times t: \Mor_{\bB} \to  \Obj_{\bB} \times  \Obj_{\bB}$ is proper.}
\end{itemize}  
Moreover by \cite[Prop.2.3]{Morb}  it has a natural completion to a ep (\'etale, proper) groupoid $\Hat\bB_\Kk$ (i.e. a category in which all morphisms are invertible) that also has realization $|\Kk|$.  Thus $\Hat \bB_\Kk $ provides an orbifold structure on $|\Kk|$.

If the obstruction spaces do not vanish, then the manifolds $U_I$ have varying dimensions.  However, if $\nu_I: U_I\to E_I$ is a perturbation section such that $s_I + \nu_I: U_I\to E_I$ is transverse to $0$, then the perturbed zero set $Z_I: = (s_I + \nu_I)^{-1}(0)$ has fixed dimension $d$. 
Hence, as is shown in \cite[Lemma~7.2.7]{MW2}, if the isotropy groups vanish and if we can choose the  $\nu_I$ compatibly, i.e. they form a functor
$$
\nu: \bB_\Kk\to \bE_\Kk,
$$
then  these zero sets fit together to form a manifold.
However, in general the domains $U_I$ overlap too much for there to be such a functor.\footnote
{
\, See the beginning of \cite[\S7.1]{MW1}. The relation between $\Uu$ and its reduction $\Vv$ is similar to that
between the cover of a simplicial space by the stars  of its vertices and the cover by the stars of its first barycentric subdivision.
In particular, though the footprints $(F_i)_{1\le i\le N}$ of the basic charts cover $X$, the corresponding sets $(G_i: = F_i\cap |V_i|)_{1\le i\le N}$ are disjoint and do not form a cover: see Figure~\ref{fig:fund1}.
}  

We deal with this by passing to a {\bf reduction}
$\Vv$, i.e. a family of  $\Ga_I$-invariant, precompact open subsets $V_I\sqsubset U_I$ with the 
following properties:
\begin{align}\label{Vv} 
&\bullet\quad \mbox{the  footprints } \bigl(G_I: = \psi_I(V_I\cap s_I^{-1}(0))\bigr)_{I\in \Ii_\Kk}\; \mbox{ cover } X, \hspace{.5in}\\ \notag  
&\bullet\quad  \pi_\Kk(\ov V_I) \cap \pi_\Kk(\ov V_J) \ne \emptyset \mbox{  only if } I\subset J \mbox{ or } J\subset I, \hspace{.5in}
\end{align} 
where $\pi_\Kk: U_I\to |\Kk|$ is the projection in \eqref{eq:piK}.
In  the construction given in \cite[\S7.3]{MW2} for the trivial isotropy case, we define the perturbation section as a functor 
\begin{align*}%\label{eq:nu}
\nu: \bB_\Kk\big|_\Vv\to \bE_\Kk\big|_\Vv
\end{align*} 
on  the full subcategory
$\bB_\Kk\big|_\Vv$ of $\bB_\Kk$ with objects $\bigsqcup_I V_I$.

  If the isotropy groups are nontrivial then it is (in general) no longer possible to choose a transverse equivariant section $\nu$,
 even on a reduction $\Vv$.
 However, because the morphisms in $ \bB_\Kk\big|_\Vv$ are described so explicitly, we show in \cite[Proposition~3.3.3]{MWiso} that
 we may construct the perturbation section as a (single valued) functor
 $$
\nu:  \bB_\Kk\big|_\Vv^{\less \Ga} \to  \bE_\Kk\big|_\Vv^{\less \Ga},
$$
where 
$ \bB_\Kk\big|_\Vv^{\less \Ga}$ is the (non full, non proper) subcategory of $\bB_\Kk\big|_\Vv$   obtained by discarding the morphisms coming explicitly  from the group actions.
Thus
\begin{align}\label{eq:TVIJ}
\Mor_{ \bB_\Kk\big|_\Vv^{\less \Ga}} & ={\textstyle  \bigsqcup_{I\subset J}} \TV_{IJ}\quad \mbox{where } \;\;
  s\times t: \; (I,J, y)\mapsto \bigl( (I,\rho_{IJ} (y)), (J,y)\bigr) \\ \notag
  & \qquad \qquad \qquad \qquad \mbox{ and }  \TV_{IJ} = V_J\cap \rho_{IJ}^{-1}(V_I) \subset \TU_{IJ}.
  \end{align}
We show in \cite[Theorem~3.2.8]{MWiso} that
if $(s + \nu)\pitchfork 0$, the full subcategory of  $ \bB_\Kk\big|_\Vv^{\less \Ga}$ with objects  $$
{\textstyle \bigsqcup_I} \bigl(Z_I: = (s_I + \nu_I)^{-1}(0)\bigr), \quad \mbox{ and \;\; weight}(Z_I) = 1/{|\Ga_I|},
$$
can be completed to a weighted  \'etale groupoid  whose realization is therefore a weighted  branched manifold
as defined  in \S\ref{ss:br}.
We will see below that in the current context the branched manifold structure of $M$ appears in a similar way.

\MS

\subsection{The weighted branched manifold  $(M, \La)$}\label{ss:br}

We will construct $M$ from the realization of an   \'etale category $\bM$ whose objects are thickened versions of the domains $V_I$ of a reduction $\Vv$ of the atlas $\Kk$, and whose morphisms have exactly the same structure as those in the category $\bB_\Kk|_\Vv^{\less \Ga}$  defined in \eqref{eq:TVIJ}.
In particular,  in general $\bM$ is not proper, so that its realization $|\bM|$ is not   Hausdorff  but rather branches along  its  locus of nonHausdorff points 
(think of two copies of a circle attached long a subarc.)

We begin with some relevant definitions from \cite{Mbr}.  If $\bG$ is a wnb groupoid as described below,  its realization $|\bG|$ with the
quotient topology is in general not Hausdorff.  Hence we consider its  maximal Hausdorff quotient $|\bG|_{\Hh}$, which has the following universal property: any continuous map from $|\bG|$ to a Hausdorff space factors through $|\bG|_{\Hh}$.
In the following we write $|\bG|$ for the realization $\Obj_{\bG}/\!\!\sim$ of an \'etale  groupoid $\bG$, and $|\bG|_\Hh$ for its maximal Hausdorff quotient.\footnote
{\, The appendix to \cite{MWiso}  gives succinct proofs of the results we use; in particular, the existence of   $|\bG|_{\Hh}$ is established in \cite[Lemma~A.2]{MWiso}. Lemma~\ref{le:Hcomplet} gives an explicit description of $|\bG|_{\Hh}$ in the case we need here.}
We denote the natural maps by
$$
\pi_\bG:\Obj_{\bG}\to |\bG|,\qquad   \pi_{|\bG|}^\Hh: |\bG|\longrightarrow |\bG|_\Hh, \qquad 
\pi^\Hh_\bG:= \pi_{|\bG|}^\Hh\circ \pi_\bG : \Obj_{\bG}\to |\bG|_\Hh .
 $$

 \begin{defn}[\cite{Mbr},Def.~3.2]\label{def:brorb}
A   {\bf weighted nonsingular branched groupoid}  (or {\bf wnb groupoid)} of dimension $d$ is a pair $(\bG,\La_{\bG})$ consisting of a nonsingular\footnote
{
\, i.e. there is at most one morphism between any two objects. Further, we restrict here to rational weights, but clearly this condition could be generalized.},
 \'etale groupoid $\bG$ of dimension $d$, together with a rational weighting function $\La_{\bG}:|\bG|_{\Hh}\to \Q^+: = \Q\cap (0,\infty)$ that satisfies the following compatibility conditions.
For each $p\in |\bG|_{\Hh}$ there is an open neighborhood $N\subset|\bG|_{\Hh}$ of $p$, a collection $U_1,\dots,U_\ell$ of disjoint open subsets of $(\pi_{\bG}^{\Hh})^{-1}(N)\subset \Obj_{\bG}$ (called {\bf local branches}), and a set of positive rational weights $m_1,\dots,m_\ell$ such that the following holds: 
\SSS

\NI
{\bf (Cover) } $( \pi_{|\bG|}^{\Hh})^{-1}(N) = |U_1|\cup\dots \cup |U_\ell| \subset |\bG|$;
\SSS

\NI
{\bf (Local Regularity)}  
for each $i=1,\dots,\ell$ the projection 
$\pi_{\bG}^{\Hh}|_{U_i}: U_i\to |\bG|_{\Hh}$ is a homeomorphism onto a relatively closed subset of $N$;
\SSS

\NI
{\bf (Weighting)}  
for all $q\in N$, 
the number 
$\La_{\bG}(q)$ is the sum of the weights of the local
branches whose image contains $q$:
$$
\La_{\bG}(q) = 
{\textstyle \sum_{i:q\in |U_i|_{\Hh}} m_i.}
$$
\end{defn}

Now we can formulate the notion of a weighted branched manifold.\footnote
{\, In distinction to \cite{Mbr} and \cite{MWiso}, we will not assume from the outset that  a weighted branched manifold is oriented, since there is no need for this hypothesis until it comes to considering the fundamental class.    Analogous definitions for cobordisms may be found in 
 \cite[App.~A]{MWiso}.}

\begin{defn}\label{def:brman} 
A {\bf weighted branched manifold} of dimension $d$ is a pair $(Z, \La_Z)$ consisting of a topological space $Z$ together with a function $\La_Z:Z\to \Q^+$ and an equivalence class\footnote{
\, The precise notion of equivalence is given in \cite[Definition~3.12]{Mbr}. In particular it ensures that the induced function $\La_Z: = \La_\bG\circ f^{-1}$ and the dimension of $\Obj_{\bG}$ is the same for equivalent structures $(\bG,\La_\bG, f)$. 
} 
of tuples $(\bG, \La_\bG, f)$, where $(\bG,\La_\bG)$  is a   $d$-dimensional wnb groupoid  and $f:|\bG|_\Hh\to Z$  is a homeomorphism that induces the function $\La_Z: = \La_{\bG}\circ f^{-1}$.
\end{defn}

We define the weighted branched manifold $(M,\La)$ of Theorem~A as the realization of a category $\bM$ constructed as follows.
First choose a $\Ga_i$-invariant norm $\|\cdot \|$  on each $E_i$, and  for any $J\subset A$ give  the vector space $E_J: =  \prod_{i\in J}E_i$  the sup norm
$$
\|e_J\| = {\textstyle \sup_{i\in J}} \|e_i\|.
$$
Further, denote
\begin{align}\label{eq:Eeps}
E_{J,\eps}: = \{e_J\in E_J \ | \ \|e_J\|<\eps\},
\end{align}
and 
\begin{align}\label{eq:ve}
\ve: =(\eps_I)_{I\in \Ii_\Kk},\quad  \mbox { where } &  I\subsetneq J \Longrightarrow 0<\ka\, \eps_I< \eps_J,\\ \notag
&\qquad \qquad \qquad  \mbox { for }  \ka: = \max \{ |J| \ | \ J\in \Ii_\Kk\}.
\end{align}
Given  a reduction $\Vv$ of an atlas $\Kk$ as in \eqref{Vv}, for each $I\subset J$
we denote
\begin{align}\label{eq:VIJ}
V_{IJ} = V_I\cap \pi_\Kk^{-1}(\pi_\Kk(V_J))\subset V_I,\quad  \TV_{IJ} = V_J\cap \pi_\Kk^{-1}(\pi_\Kk(V_I)) \subset V_J,
\end{align} 
where $\pi_\Kk:V_I\to |\Kk|$ is the obvious projection.   Thus $\rho_{IJ}(\TV_{IJ}) = V_{IJ}$.
Observe also  that the group $\Ga_A$ acts on $E_{A\less J,\eps_J} \times V_J$
by 
\begin{align}\label{eq:GaA}
\ga\cdot (e,x) =  \bigl(\ga|_{A\less J}(e), \ga|_J (x)\bigr),\quad \ga \in \Ga_A,
\end{align}
where $\ga|_J$ denotes the projection of $\ga\in \Ga_A: = \prod_{i\in A} \Ga_i$ to $\Ga_J:= \prod_{i\in J} \Ga_i$.
\MS

The following result is  the key step in the proof of Theorem~A.  A more precise version is stated and  proved in Proposition~\ref{prop:M1} below.
% compare with \eqref{eq:TVIJ} above.    

 \begin{prop}\label{prop:M} 
Let $\Kk$ be a good  atlas on $X$ of dimension $d$.   Then there is a reduction $\Vv$ and choice of constants $\vd>0$ 
such that the following holds. 
\begin{itemlist}\item [{\rm (i)}] 
There is    an  \'etale category $\bM$  of dimension $D: = d+\dim E_A$ with
\begin{align}\label{eq:Mm}
\Obj_{\bM} &= {\textstyle  \bigsqcup_{I\in \Ii_\Kk} } M_J: = E_{A\less J,\de_J} \times V_J, \qquad 
\Mor_{\bM} ={\textstyle \bigsqcup_{I\subset J,\, I,J \in \Ii_\Kk} \TM_{IJ} } \\ \notag
s\times t: &\;\;\Mor_{\bM}\to \Obj_{\bM}\times \Obj_{\bM},\quad (I,J,y)\mapsto \Bigl(\bigl(I,\tau_{IJ}(y)\bigr), \bigl(J,y\bigr)\Bigr),
\end{align}
where  $\TM_{IJ}\subset M_J$ is an open $\Ga_A$-invariant 
subset containing $\{0\}\times \TV_{IJ}$ whose closure $cl(\TM_{IJ})$ is disjoint from $cl(\TM_{HJ})$ unless $I\subset H$ or $H\subset I$,
and  the map  %that supports a free action of $\Ga_{J\less I}$  and 
\begin{align*}%\label{eq:tau}
\tau_{IJ}: \TM_{IJ} \to M_{IJ}: = E_{A\less I,\de_I} \times V_{IJ} \subset M_I
\end{align*}
is a $\Ga_A$-equivariant covering map onto $M_{IJ}\subset M_I$ such that 
\begin{itemize}\item[-] 
$\tau_{IJ}$ restricts to $\rho_{IJ}$ on $\{0\}\times \TV_{IJ}$; 
\item[-] if $H\subset I\subset J$ then   $\tau_{HJ} = \tau_{HI}\circ \tau_{IJ}$ on $\TM_{IJ}\cap \TM_{HJ}$; and
\begin{align}\label{eq:clgraph}
{\rm graph \,} \tau_{IJ}\subset M_I\times M_J\ \;\; \mbox{ is closed.}
\end{align}
\end{itemize}

\vspace{.05in} % .  

\item [{\rm (ii)}]  
$\bM$ supports an action of  $\Ga_A$  by \eqref{eq:GaA} on objects, and  by
\begin{align*}%\label{eq:Gaact}
%& \de: \bigl(I,(e,x)\bigr)\mapsto \bigl(I,\, (\de|_{A\less J}^{-1}e, \de|_J^{-1}x)\bigr),\quad y=(e,x)\in E_{A\less J,\eps_J} \times V_J,\\
%\notag
& \bigl(I,J, y \bigr)\mapsto \ga\cdot (I,J,y): = \bigl(I, J,\, \ga^{-1}\cdot y\bigr),\quad \ga\in \Ga_A,\; y\in \TM_{IJ},
\end{align*}
on morphisms.
\item  [{\rm (iii)}]   There is a $\Ga_A$-equivariant functor $\Sss: \bM \to \bE_A$, where the category  $\bE_A$ has objects $E_A$ and only identity morphisms, % morphisms given by the action of $\Ga_A$, 
that is given on objects  by maps $\Sss_J: M_J\to E_A$  such that  
%of the  following form
%where
\begin{align}\label{eq:sJ}
&\Sss_J(0, x) = s_J(x),\quad \Sss_J^{-1}(E_J)\subset \{0\}\times V_J
% \Sss_J(e_{A\less J}, x) & = 
% \bigl(\la_J(e_{A\less J}, x) e_{A\less J}, \Sss_{J,2}(e_{A\less J}, x)\bigr) \in E_{A\less J}\times E_J, \quad \mbox{ where  } \\ \notag
%% (\Sss_J)^{-1}(0) & = \bigl\{(0,x)\in E_{A\less J}\times V_J: s_J(x) = 0\bigr\},
%\la_J(e_{A\less J}, x)  \; &\in \; (0,\infty) \subset \R,
\end{align}
so that 
$$
(\Sss_J)^{-1}(0)  = \bigl\{(0,x)\in E_{A\less J}\times V_J: s_J(x) = 0\bigr\}.
$$
%where for $I\subset J$ we denote $\Sss_{J|J}(e_{A\less J}, x) := \pr_{E_J}\circ \Sss_J(e_{A\less J}, x) $.
%In particular, $\Sss_J^{-1}(0) = \{0\}\times s_J^{-1}(0)\subset  \{0\}\times V_J$.
\end{itemlist}
\end{prop}

The following result explains the construction  and properties of the weighted branched manifold $(M,\La)$  mentioned in Theorem~A.
Note that  $\Sss$ denotes a functor $\bM\to \bE_A$, while $\Sss_M:M\to E_A$ is the corresponding function on $M$.

\begin{thm}\label{thm:M}   \begin{itemlist}\item[{\rm (i)}] The category $\bM$ constructed in Proposition~\ref{prop:M} has a unique completion to a  wnb groupoid $\Hat{\bM}$  
with the same objects as $\bM$ and the same realization
$|\Hat{\bM}| = |{\bM}|$. 
\item[{\rm (ii)}]  If we denote the composite  $\Obj_{\bM}\to |\bM|\to |\Hat{\bM}|_\Hh$ by $y\mapsto |y|\mapsto
\pi_{\bM}^\Hh(|y|)$, the function $\La: M: = |\Hat{\bM}|_\Hh\to \Q^+$ defined by
$$
\La(p) := \tfrac 1{|\Ga_I|}\,\cdot  \# \bigl\{ y\in M_I \,\big|\, \pi_{\bM}^\Hh(|y|) =p \bigr\} \qquad \text{for}\;\; p\in |M_I|_\Hh
$$
is a weighting function 
that gives
% on the Hausdorff quotient of the perturbed zero set $|\bZ^\nu|_\Hh$. Together, these give
  $(M, \La)$ the structure of a  weighted branched manifold.
  \item[{\rm (iii)}]  The group action by $\Ga_A$ and functor $\Sss$  extend to $\Hat\bM$, so that there is a $\Ga_A$-equivariant map $\Sss_M: M\to E_A$.  
    Moreover,
%     the image 
% $|\Sss^{-1}(0)|_\Hh$ of
  the zero set  $\Sss_M^{-1}(0)$  is a compact subset of $M$, and 
  the footprint maps $\psi_I$ induce   a homeomorphism
  $$
\upsi: \;  \Sss_M^{-1}(0)/\Ga_A \stackrel{\cong}\longrightarrow X.
  $$
    \item[{\rm (iv)}]   If $\Kk$ is oriented, so are $\bM$ and $\Hat\bM$.
\end{itemlist}
\end{thm}

The  category $\bM$ has the same structure as the category $\bZ^\nu$ considered in \cite[Thm.~3.2.8]{MWiso}, formed by the perturbed zero set of the atlas $\Kk$; and the proof of Theorem~\ref{thm:M} (which is given in  \S\ref{ss:topol})
is essentially the same as the corresponding result for $\bZ^\nu$. Condition~\eqref{eq:clgraph} that $\tau_{IJ}$ has closed graph is automatically satisfied
 in the case of
$\bZ^\nu$, and is an important ingredient of the analysis of the branching structure of $M$. For example, if  the isotropy groups are trivial, then 
the maps $\tau_{IJ}$  are homeomorphisms onto their images, and Lemma~\ref{le:complet} implies that the only morphisms in the groupoid completion $\Hat{\bM}$ are those given by the $\tau_{IJ}$ and their inverses.  Hence,  condition in \eqref{eq:clgraph} implies that the equivalence relation on $\Obj_{\bM}$ has closed graph, so that the quotient space $|\Hat{\bM}|$ is Hausdorff, and therefore  a manifold.
An example with nontrivial isotropy is described in Example~\ref{ex:sphere}~(IV).

\begin{proof}[Proof of Theorem~A]  This is an immediate consequence of Theorem~\ref{thm:M}.
\end{proof}

\begin{rmk} \rm   Instead of taking $M$ to be a weighted branched manifold with action of $\Ga_A$, one could add the morphisms in $\Ga_A$ to the completed category $\Hat{\bM}$ to obtain an \'etale groupoid $\Hat{\bM}\times\Ga_A$.  In general, this groupoid is not proper. However, it does inherit a weighting function and  so the realization $|\Hat{\bM}\times\Ga_A|_\Hh$ is a weighted branched orbifold $M/\Ga_A$: for an explicit example see \S\ref{ss:exam}~(VI). Note also that the action of the group $\Ga_A$ on $\bM$ only affects the fundamental class $\mu_M$ (and hence $[X]^{vir}_\Kk$)  via the weighting function $\La$ whose values depend on the groups $\Ga_I$ as well as on the category $\bM$. \hfill$\er$
\end{rmk}

 \begin{rmk}\label{rmk:out}\rm {\bf (Outline of the argument)} 
We will explain the main points of the proof of Proposition~\ref{prop:M}  in \S\ref{s:out}.  The first step is to use \lq deformation to the normal cone' (see \cite{P}) to  construct  manifolds $(Y_{\Uu,J,\,\ve})_{J\in \Ii_\Kk}$ of dimension $d + \dim E_A + |J|-1$ with a natural boundary that lies over the boundary of a simplex $\De_J$ of dimension $|J|-1$.  We next consider the open submanifold     $Y_{\Vv,J,\,\ve}\subset Y_{\Uu,J,\,\ve}$ corresponding to a reduction, and show that this has a  partial boundary collar with \lq corner control': see Proposition~\ref{prop:col}.   
Then we use the collar to construct the
covering maps $\tau_{IJ}: \TM_{IJ}\to M_I$.  Since the general definition of these maps  is quite complicated, 
we explain in Example~\ref{ex:MJ} how this  works for an atlas with just three basic charts.   
 Proposition~\ref{prop:M1} gives the general construction. 
 
   \S\ref{ss:Y}  contains technical details about compatible shrinkings, and the proof that each $Y_{\Uu,J,\,\ve}$ is a manifold.   The argument here is based on the existence of the local product structures provided by the submersion axiom.  As we show in Step 1 of the proof of Proposition~\ref{prop:col} in Lemma~\ref{le:col}, this axiom also allows one to construct local collars  that are compatible with the covering maps $\rho_{IJ}$ and with projection to the vector spaces $E_{J\less I}$.  In Step 2 of this proof we explain a standard method (described in Hatcher~\cite{Hat}) for assembling these local collars into a global collar for each $Y_{\Vv,J,\,\ve}$, and show in Step 3  how to arrange  that these collars have  the consistency properties listed in Proposition~\ref{prop:col} that are needed in the definition of the maps $\tau_{IJ}$.  This last step works under the assumption that the
 domains of the local collars are compatible with the reduction $\Vv$ and choice $\ve$ of thickening constants in a rather subtle way, which is summarized in the notion of compatible reduction $(\Vv,\ve)$ in Definition~\ref{def:compatV}.
 \hfill$\er$
\end{rmk}

\begin{rmk}\label{rmk:genl}\rm {\bf (Generalizations)}   The construction  of $\bM$ could be generalized in various ways. The argument relies in an essential way on 
 the submersion property in order to construct the collars in Proposition~\ref{prop:col}, i.e. on the
fact that along $\TU_{IJ}$ the space  $U_J$ is locally the product of the vector space $E_{J\less I}$ with the domain $U_I$.
 However, it does not use the fact that the domains  $U_I$ themselves are topological manifolds: for example, since all we want in the end is information on homology,  it would no doubt suffice if they were (locally compact, metrizable)  homology manifolds of dimension $\dim E_I + d$.    One could also consider atlases (or equivalently categories $\bB_\Kk$) whose charts are indexed by a poset more general than that given by the subsets of $A$.  However, one does need to be able to restrict attention to a subcategory such as $\bB_\Kk|_\Vv$ in which there are morphisms between the elements of two components of the domain only if the indices of those components are comparable in the given poset.  Some possible generalizations of this kind are discussed in the last section of \cite{Mnot}.
 \hfill$\er$
\end{rmk}

\begin{rmk}\label{rmk:poly} \rm  {\bf (The polyfold approach)}   If $X$ is the zero set of a Fredholm section $\s$ of a polyfold bundle $\Ee\to \Ss$ of index $d$, then one can use the fact that the realization $|\Ss|$ supports partitions of unity to give a very simple construction for a weighted branched manifold $M$ and section $\Sss$ whose corresponding relative Euler class agrees with that of $\s: \Ss\to \Ee$.    (In the applications of interest to us  $\Ss$ is a category\footnote
{
\, One can think of $\Ss$ as an infinite dimensional version of an ep groupoid, where the objects $\Obj_{\Ss}$ do not form a set but nevertheless 
the quotient $|\Ss| = \Obj_{\Ss}/\!\!\sim$ is a topological space, where $\sim$ is defined by setting $x\sim y \Longleftrightarrow \Mor_\Ss(x,y)\ne \emptyset$.  % although for each point $q\in \Obj_\Ss$ there is a set $\
} 
 whose realization is a space of stable maps with the Gromov topology: see \cite{H:sur,Hbigbook}.)
Here is a very brief outline of the construction: for full details see \cite{MWgw}. 

Given $x\in X$ with stabilizer subgroup  $\Ga_x$, choose a lift $q_x\in \Obj_{\Ss}$, and  a $\Ga_x$-invariant open neighborhood $\Oo\subset \Obj_\Ss$ 
of $q_x$ such that the map $\Oo\to |\Oo|\subset |\Ss|$ factors through a homeomorphism $\Oo/\Ga_x\stackrel{\cong}\to |\Oo|$.  Because $\s$ is Fredholm, there is
 a $\Ga_x$-equivariant  linear map $\la: E\to {\rm Sect\,}(\Ee|_\Oo)$ from a $\Ga_x$-invariant normed linear space $E$ to
 a subspace of $sc^+$-smooth sections that covers the 
 cokernel of the linearization of $\s$ at $x$.  
 It follows that there is $\eps>0$ such that the set 
 \begin{align}\label{eq:UU0}
U: = \bigl\{(e,q) \in E\times \Oo\ \big|\ \s(q) = \la(e,q), \|e\|<\eps\bigr\}
\end{align} 
is a manifold of dimension $d + \dim E$.  (The proof  involves a nontrivial amount of analytic detail that will appear in \cite{MWgw}.)  
 Choose a finite covering of the compact set $X: = |\s^{-1}(0)|$ by the footprints $\bigl(\psi_i(s_i^{-1}(0))\bigr)_{i\in A}$ of such charts $$
 \bK_i: = \bigl(U_i, E_i, \Ga_i, s_i, \psi_i\bigr),\qquad  \; s_i(e,q) = e,
 $$ 
 and let $(|\Oo_i|)_{ i\in A}$ be the associated open cover of a neighborhood of $X$ in the ambient space
 $|\Ss|$. 
 Just as in \cite{Morb}, one can use the groupoid structure of $\Ss$ to show that the $\bK_i$ form the basic charts for a tame Kuranishi atlas  
 $\Kk_{\Oo,\la}$   whose transition charts are given by tuples of composable morphisms.
 Instead of giving more detail about this construction, we will 
outline how to modify these definitions so that   the domains of the charts all have the same dimension $d+ \dim E_A$.

 First choose a  family of bump functions $(\si_i)_{i\in A}$ with $\supp \si_i\subset |\Oo_i|$ such that $$
 X = |\s^{-1}(0)| \;\subset\; \textstyle{\bigcup_i  \{x\; |\; \exists i, \si_i(x)>0\}.}
 $$
 Then choose an ordering of the elements $i\in A$ and a 
 reduction $\bigl(|\Ww_I|\bigr)_{I\in \Ii_\Kk}$  of the covering $\bigl(|\Oo_i|\bigr)_{i\in A}$ with the following properties:
 % of a neighborhood of $X = |\s^{-1}(0)| $ in $|\Ss|$.  Thus, as in \eqref{Vv} we assume that
 \begin{itemize}\item for each $I\in \Ii_\Kk$, $|\Ww_I|\subset |\Oo_I|: = {\textstyle \bigcap_{i\in I}}\; |\Oo_{i}|,$
 \item $X\subset  {\textstyle \bigcup_{I\in \Ii_\Kk}}\; |\Ww_I|$;
 \item $|\Ww_I|\cap |\Ww_J|\ne \emptyset  \;\Longrightarrow I\subset J \mbox{ or } J\subset I$;
 \item  if $i\notin J$ then $s_i\equiv 0$ on $|\Ww_J|$.
 \end{itemize}
 %Next choose an ordering of the elements $i\in A$.
 % and  a  partition of unity $(\rho_i)_{i\in A}$ subordinate to the cover 
 %$\bigl( |\Ww_i| \bigr)_{i\in A}$.  
 Then, given 
 $I = \{i_0,\dots, i_k\}$ where $i_0< i_1<\dots < i_k$, the space $M^\Ww_I$ consists of all tuples
\begin{align*}
 \Bigl\{ \bigl(e_A, q_{i_0}, \Psi_{q_{i_0} q_{i_1}}, \cdots, q_{i_k} \bigr)\  \big| \  & |q_{i_0}|\in |\Ww_I|, \ 
\Psi_{q\,q'}\in \Mor(q, q'),\;  \|e_A\|< \eps,     \\
% &\hspace{1.5in}
 & \s(q_{i_0})
 = {\textstyle \sum_j }\  \si_{i_j}(|q_{i_0}|)\ \Psi^*\bigl(\la_{i_j}(e_{i_j})(q_{i_j})\bigr) \in \Ee_{q_0} \Bigr\},
\end{align*}
 where $(q_{i_0}, \Psi_{q_{i_0} q_{i_1}},q_{1_i},  \Psi_{q_{i_1} q_{i_2}},\cdots, q_{i_k})$ is a 
 composable $k$-tuple of morphisms from a point $q_0\in \Oo_{i_0}$ to $q_k\in \Oo_{i_k}$.  
 By \cite[Thm~7.4]{Hbigbook}, we may choose the $\si_j$ so that for each $i, j\in A$ the function
 $$
\Oo_i\to  [0,1], \qquad  q \mapsto \si_j(|q|)
$$
is sc-smooth.  It follows that 
if  $\eps>0$ is suitably small,
 then, for each $I$,  $M^\Ww_I$ is a manifold of dimension $d+ \dim E_A$ with action of $\Ga_A$. Moreover, much as in \cite[Prop.2.3]{Morb}, for each $I\subset J$ one can define a $\Ga_A$-equivariant
covering map
 $$
\tau_{IJ}:  M^\Ww_J\supset \TM^\Ww_{IJ}\; \to\; M^\Ww_{IJ}\subset M^\Ww_I
 $$  
  by taking an appropriate  combination of the structural maps in $\Ss$ (such as compositions and source/target maps),  where 
  $M^\Ww_{IJ}$ (resp. $ \TM^\Ww_{IJ}$) consists of all elements in   $M^\Ww_{I}$ (resp. $ M^\Ww_{J}$) with $|q_{i_0}| \in |\Ww_I|\cap 
  |\Ww_J|$.
  This gives a category $\bM$ whose structure is precisely as described in Proposition~\ref{prop:M}.  The resulting VFC  $[X]^{vir}$ is independent of all choices, and can be shown to agree with that defined by the polyfold Fredholm section $\s: \Ss\to \Ee$.  
  
  Notice that  the equation  satisfied by the elements in  $M_I^\Ww$ involves the bump functions $\si_j$, while the equation \eqref{eq:UU0} defining the chart domains 
of the atlas  $\Kk_{\Oo,\la}$   does not.  Hence 
   the  weighted branched  manifold $(M^\Ww,\La)$ constructed above
  is not identical to the manifold obtained from the atlas $\Kk_{\Oo,\la}$ by the collaring construction described below. Nevertheless, these two constructions are closely related and, by adapting the arguments in \S\ref{ss:topol}  one can show that they define the same virtual fundamental class  $[X]^{vir}$.      For more details, see \cite{MWgw}.  
 \hfill$\er$
\end{rmk}

\subsection{Examples and list of main notations}\label{ss:exam}

We end this introduction by giving some examples.  Though these not needed for the proofs of the main results, readers unfamiliar with the description of   orbifolds  via atlases might find it useful to read at least some of  this section before proceeding further.

We begin by discussing the definition of the  relative Euler class of an oriented  vector bundle $\pi:\Ee\to W$ over a manifold that is equipped with a section $\s:W\to \Ee$
whose zero set $X: = \s^{-1}(0)$ is compact.  In particular, we explain why the method  outlined  in equation \eqref{eq:Euler} 
does compute the Euler class $e(\Ee)$ of $\Ee\to W$ when $W$ is compact and $\s\equiv 0$.  In Remark~\ref{rmk:orb}, we describe how to extend the construction to orbibundles.
Finally, we 
 show in detail how our main construction works  to calculate 
the Euler class of the tangent bundle of $S^2$, starting from the atlas  defined in  
\cite{MWiso}.  Our approach easily generalizes to the 
 football
orbifold $S^2_{p,q}$, which is $S^2$ with orbifold points of orders $p,q$ at the two poles.  
\MS

Let $\pi:\Ee \to W$  be an oriented, vector bundle over the manifold $W$, together with a section $\s:W\to \Ee$ with compact zero set $X\subset W$.  As always (see Remark~\ref{rmk:submer}~(iii)), we suppose that $\Ee$ has even rank to avoid problems with orientation.\footnote
{
\, Of course, over $\Q$ the Euler class vanishes for bundles of odd rank anyway.}
We build a (Kuranishi) atlas whose charts are defined using tuples 
$$
(\Oo, E, \tau, s),
$$
where 
\begin{itemize}\item
$\Oo\subset W$ is open, 
\item
$E$ is an even dimensional, oriented  vector space, 
\item $\la: E\times \Oo\to \Ee|_\Oo$ is a surjective orientation-preserving bundle homomorphism over $\id_\Oo$, and 
\item $\la$ pushes $s:\Oo\to E$  forward to $\s|_\Oo$, i.e. $\la(s(x), x)) = \s(x)\in \Ee|_x, \;\;\forall x\in \Oo$.
\end{itemize}
Given such a tuple, the corresponding chart  
$$
\bK: = (U, E, s, \psi), \quad\mbox{ with footprint}\;\; F,
$$
is defined by setting
$$
U = \bigl\{(e,x) \in E\times \Oo \ | \ \la(e,x) = \s(x)\bigr\}, \quad s(e,x) = e, \quad  \psi(0,x)\mapsto x\in X.
$$ 
One obtains an atlas as defined in \S\ref{ss:defs} by taking the basic charts to be a finite family $\bigl(\bK_i\bigr)_{i=1,\dots A}$ of charts of this form whose footprints $(F_i)$ cover the compact set $X = \s^{-1}(0)$, and the transition charts $\bigl(\bK_I\bigr)_{I\in \Ii_\Kk}$ to be the corresponding charts 
$(U_I, E_I, s_I, \psi_I)$ with footprints $F_I: = \bigcap_i F_i$ that are formed just as above but now with 
$E_I = \prod_{i\in I} E_i, \la_I = \sum_{i\in I} \la_i$. In particular,
$$
 U_I = \bigl\{ \bigl((e_i),x\bigr) \in E_I\times \Oo_I\ \big| \  \sum \la_i(e_i,x) = \s(x)\bigr\},\quad\mbox{ where } \;  \Oo_I: = {\textstyle \bigcap_{i\in I}} \Oo_i.
 $$
This gives an atlas in which the coordinate changes $\bK_I\to \bK_J$ are given by the obvious identifications
$$
\TU_{IJ}: = \bigl\{ (e,x)\in U_J \ \big| \ e\in E_I, \ x\in \Oo_J \bigr\} \stackrel{\cong} \to U_{IJ} = \bigl\{ (e,x)\in U_I \ \big| \ x\in \Oo_J\bigr\}.
$$
To see that the submersion condition holds,
choose for each $I$ a right inverse $\si_I: \Ee|_{\Oo_I}\to E_I\times \Oo_I$ to $\la_I$, so that  $\la_I\circ \si_I = \id$,
and define
$$
\Ee'_{J\less I} = \bigl\{ \bigl(e' -\si_I(\s(x)), x\bigr)\ \big| \ e' \in E_{J\less I}, x\in \Oo_{IJ}\bigr\} \subset E_J\times \Oo_J.
$$
Then 
$\Ee'_{J\less I}$ is an affine subbundle of $E_J\times \Oo_J\to \Oo_J$, and we may identify $U_J$ with the pullback of 
$\Ee'_{J\less I}$ to $\TU_{IJ}$ by the projection $\TU_{IJ}\to U_J, \; (e,x)\mapsto x.$
\MS

Since there is such an atlas for every collection of  charts $\bK$ whose footprints cover $X$, 
any two such atlases $\Kk^0, \Kk^1$  are {\bf directly commensurate}, i.e. there is an atlas $\Kk$ whose  charts include those of $\Kk^0$ and $\Kk^1$.  Therefore  $\Kk^0, \Kk^1$ are cobordant by \cite[\S6.2]{MW2}.  Hence, they define 
cobordant manifold models $(M, E_A, \Sss)$ by Theorem~A and the same class $[X]^{vir}_\Kk$ byTheorem~B.
\MS

If the bundle $\Ee\to W$ is smooth, then we can define the VFC either as in the proof of Theorem~B given in \S\ref{ss:topol}, or via an inverse limit of the homology classes 
of the zero sets of a family of perturbed sections $\s + \nu_k$ of $\Ee\to W$. As explained in the proof of Theorem~B, these two approaches give the same answer. If $W$ is just a topological manifold, it is of course easiest to represent the Euler class by starting with an atlas with just one basic chart (and hence just one chart).  In this case, our general method  of building an atlas gives  the tuple described in \eqref{eq:Euler}.  We now show that
 if $\s\equiv 0$ so that $X = \s^{-1}(0) = W$ is a compact manifold,  then $[X]^{vir}_\Kk$ as defined in  \eqref{eq:Xvir}
  is Poincar\'e dual to the usual  Euler class $e(\Ee)\in  H^{2k}(X;\Z)$, where 
$2k = {\rm rank\,} \Ee$.  In the following lemma, we use simplicial (co)homology instead of the \c{C}ech theory discussed in \S\ref{ss:app},  since all spaces are manifolds, and take coefficients $\Z$ since the isotropy is trivial.  

\begin{lemma}\label{le:Euler}  If $\Ee\to X$ is an oriented $2k$-dimensional vector bundle over an oriented  $(2k+d)$-dimensional manifold $X$ with $\s\equiv 0$ and atlas $\Kk$ as above, then
$$
[X]^{vir}_\Kk = \mu_X \cap e(\Ee) \in  H_d(X),
$$
where $\mu_X$ is the fundamental class of $X$ and $e(\Ee)\in H^{2k}(X;\Z)$ is the  Euler class of $\Ee$.
\end{lemma}
\begin{proof}  By Theorem~B and the above remarks, it suffices to calculate $[X]^{vir}_\Kk$ using an  atlas with one chart as in \eqref{eq:Euler}.
Thus we may take
$$
M=\Ee', \qquad \Sss: M\to \R^N, \;\; (e',x)\mapsto \pr_{\R^N} \bigl( \io (e',x)\bigr),
$$
where $\Ee'$ has rank $2\ell$, $N = 2k + 2\ell$, $\io: M\to \Ee'\oplus \Ee$ is the inclusion and $\pr_{\R^N}$ is the projection
$$
\pr_{\R^N}: \Ee\oplus \Ee'  \cong   \Oo_X^N: = \R^N\times X \to \R^N.
$$
  Denote  the Thom classes of $\Ee, \Ee'$ by $\tau_\Ee, \tau_{\Ee'}$ and their pullbacks to $\Oo_X^N$ 
by $$
\Ti \tau_\Ee\in H^{2k}(\Oo_X^N, \Oo_X^N\less \Ee'), \quad 
\Ti \tau_{\Ee'}\in H^{2\ell}(\Oo_X^N, \Oo_X^N\less \Ee).
$$  Then, if $\tau_{\R^N}\in H^N(\R^N, \R^N\less \{0\})$ is the canonical generator, we have
$$
\Sss^*(\tau_{\R^N}) = \io^*(\tau_{\Oo_X^N}) = \io^*(\Ti\tau_\Ee\cup  \Ti\tau_{\Ee'}) \;\; \in \; H^N\bigl(M, M\less X\bigr).
$$
We may identify  $ \mu_M\cap  \tau_{\Ee'}$ with the fundamental class $\mu_X \in H_{2k+d}(X)$, where  
$\mu_M\in H_{2k+d}(M,M\less X)$ is the restriction of the fundamental class of $M$.
Then for any class $b\in H^d(X)$,  we use the cap product in \eqref{eq:commcap} with $Y=M, A=\emptyset$ and $Y\less U = X$, and the relation between cap and cup products for even dimensional classes, to obtain
\begin{align*}
\bigl\langle [X]^{vir}_\Kk,\, b \bigr\rangle:  & =\bigl\langle \Sss_*(\mu_M\cap b),\, \tau_{\R^N}\bigr\rangle\\
& = \bigl\langle  \io_*(\mu_M\cap b),\,\Ti \tau_\Ee\cup \Ti \tau_{\Ee'} \bigr\rangle \\
& = \bigl\langle  \io_*(\mu_M\cap b \cap \tau_{\Ee'} ),\,\Ti \tau_\Ee \bigr\rangle\\
& =   \bigl\langle  \mu_X \cap b,  \io_X^*(\tau_\Ee) \bigr\rangle\; = \;   \bigl\langle  \mu_X \cap e(\Ee), b \bigr\rangle,
\end{align*}
where  we have written $\io_X: X\to \Ee$ for the inclusion and used the fact that $\io_X^*(\tau_\Ee)$ is the Euler class $e(\Ee)\in H^{2k}(X)$ of $\Ee$.
\end{proof}

\begin{rmk}\label{rmk:orb}\rm (i)
The above construction easily adapts to the case of an oriented  orbifold bundle $\Ee\to W$ over an oriented  orbifold $W$, where now we should think of the spaces
$\Ee, W$ as the realizations of suitable ep categories $\bE, \bW$.
Thus, one can build an atlas whose  basic charts are as above with the addition of a group action, while
the transition charts are made using composable  tuples of morphisms in $\bE$.   For details, see \cite[\S 5.2]{Mnot}.  One can then piece the corresponding fattened  charts together by the method explained in \S 2,3 below  to obtain a tuple
$(M,E_A,\Sss)$ as in Theorem~A.  However, we can also build the category $\bM$   directly from the set of basic charts
$(U_i,E_i,\Ga_i,s_i,\psi_i)$, using a partition of unity, and an associated reduction as explained in 
Remark~\ref{rmk:poly}.
\MS

\NI (ii)  In Gromov--Witten theory it sometimes happens that  the space of $J$-holomorphic maps in class $A$ does form a compact manifold (or orbifold) $X$ such that the rank of the cokernel of the linearized Cauchy--Riemann operator $D_x$ at $x\in X$  is independent of $x$.  In this case,  these cokernels fit together to form  a bundle $\Ee\to X$ such that the map $\s$ induced by the Cauchy--Riemann operator is zero.    We explain in \cite[Remark~5.2.4]{Mnot} why one can choose a Gromov--Witten  type atlas (constructed as in \cite[\S4]{Mnot} or \cite{P}) with precisely the structure considered above.
\MS

\NI (iii)  In \cite[Prop.~5.3.4]{P}, Pardon proves the analog of Lemma~\ref{le:Euler} in the smooth case,  using a transverse perturbation of $\s$ as 
in Step 3 of the proof of  Lemma~\ref{le:fundM}.
\hfill$\er$
\end{rmk}
\MS

\begin{example}\label{ex:sphere}\rm {\bf (The tangent bundle  of the $2$-sphere and the football)}
We now  illustrate the construction in the proof of Theorem~A  in the case of the bundle $\pi:\rT S^2 \to S^2$ with section $\s \equiv 0$, starting from the Kuranishi  atlas 
with two basic charts that was constructed in 
\cite[Example~3.4.2]{MWiso}.
We organize the details  into several steps.\MS

\NI {\bf (I)} {\it Atlas for the tangent bundle of the $2$-sphere.}
To build a Kuranishi  atlas whose associated \lq bundle' $\pr: |\bE_\Kk|\to |\Kk|$ models $\rT S^2$, cover $S^2$ by two copies $D_1,D_2$ of the unit disc in $\C$, whose intersection $D_1\cap D_2 = : D_{12}=:A\cong [0,1]\times S^1$ 
is an annulus, and for $i=1,2$ define $$
\bK_i: = (U_i: = D_i,\ E_i: = \C,\ s_i : = 0,\ \psi_i: = \id).
$$
  For $i=1,2$, choose unitary trivializations
$T_i: D_i\times \C\to \rT S^2|_{D_i}, (x,e)\mapsto T_{i,x}(e)$ %that depend only on the absolute value $|x|$ of $x\in D_i$, 
and then define the transition chart 
$$
\bK_{12}: = \bigl( U_{12}\subset E_1\times E_2\times A ,\ E_1\times E_2,\ s_{12} = \pr_{E_1\times E_2},\ \psi_{12} = \pr_A|_{0\times 0\times A}\bigr)
$$
by setting
$$ 
U_{12}: = \bigl\{ (e_1,e_2,x) \ \big| \ x\in A,\; T_{1,x}(e_1) + T_{2,x}(e_2) = 0\bigr\}.
$$
The coordinate changes $\Hat\Phi_{i,12}$ are given by taking $U_{i,12} = \{(0,0)\}\times A$ and
$\rho_{i,12}(0,0,x) = x$.
To justify this choice
of Kuranishi atlas, note that one can construct a commutative diagram 
\[
\xymatrix{
|\bE_\Kk| \ar@{->}[d]\ar@{->}[r] & \rT S^2\ar@{->}[d]\\
|\bB_\Kk|  \ar@/^1pc/[u]^{|\s |} \ar@{->}[r] & S^2   \ar@/^1pc/[u]^{s\equiv 0} ,
}
\]
where the top horizontal map restricts on $U_{12}\times E_{12}$ to the map $$
\bigl((e_1,e_2,x),e_1',e_2')\mapsto 
T_{1,x}(e_1') + 
T_{2,x}(e_2')\in \rT_x S^2\subset \rT S^2|_A.
$$ 
Thus it  takes 
$$
{\rm graph }\, s_{12} = \bigl\{ \bigl( (e_1,e_2,x),\ e_1,e_2\bigr)\ \big| \ (e_1,e_2,x)\in U_{12}\bigr\} \;\subset\; U_{12}\times E_{12}
$$
to the zero section of $\rT S^2$.

This construction is generalized to other (orbi)bundles in \cite{Mnot}.  \hfill$\er$\MS

\NI {\bf (II)} {\it Calculating the Euler class.}
In order to calculate the Euler class of $\rT S^2$ it is convenient to 
identify the annulus $A$ with $[0,1]\times S^1$, and then consider the corresponding trivialization  $\rT S^2|_A\equiv A\times \R_t\times \R_{\theta}$ where $t\in [0,1]$ and $\theta\in S^1\equiv \R/2\pi\Z$ are  coordinates.  Then for $i=1,2$  there is a  section $\nu_i: U_i\to E_i$ with one transverse zero such that $$
T_{i,x}(\nu_i(x)) = (x,1,0)\in A\times \R_t\times \R_{\theta}\equiv \rT S^2|_A, \quad x\in A
$$
 (Take suitably modified 
versions of the sections $\nu_1(z) = z, \nu_2(z) = -z$ where $D_i\subset \C$.)
Therefore the $\nu_i$ fit together to give a global section of $\rT S^2$ with two transverse zeros, and it follows
that the Poincar\'e dual of $e(\rT S^2)$ is represented by $2[pt]\in H_0(S^2).$

To see how $e(\rT S^2)$ is calculated via the atlas, we start by choosing
 a reduction $\Gg$ of the footprint covering.  
For example, we may take %$(V_i\sqsubset D_i)_{i=1,2}$ to be a slightly smaller disc and take
 $G_{12} = (\eps,1-\eps)\times S^1\sqsubset A$ for some $\eps\in (0,\frac 14)$ and choose $G_i\sqsubset D_i$ 
so that $$
\TV_{1,12} =(0,0) \times (\eps,\tfrac 14)\times S^1 \subset U_{12}, \quad 
\TV_{2,12} = (0,0)\times  (\tfrac 34,1-\eps)\times S^1 \subset U_{12}.
$$  Choose a cutoff function $\be:[0,1]\times S^1\stackrel{\pr} \to [0,1]\to [0,1]$ that equals $1$ in $[0,\frac 14]\times S^1$ and $0$ in $[\frac 34,1]\times S^1$.  Then the map $\nu_{12}: \TV_{12}\to E_1\times E_2$ given by
$$
\nu_{12}(e_1,e_2,x) \;=\;\bigl(\be(x) \nu_1(x),(1-\be(x))\nu_2(x)\bigr)  
\;\in\; E_1\times E_2
$$
% defines an admissible perturbation section that 
restricts to $\nu_i$ on $V_{i, 12}\subset (0,0)\times A$ for $i=1,2$, so that the tuple $(\nu_1,\nu_2, \nu_{12})$ is an admissible perturbation section
in the sense of \cite{MWiso}.
Moreover $s_{12} +\nu_{12}$ does not vanish at any point $(e_1,e_2,x_0)\in V_{12}$ because the three equations 
\begin{align*}
&\hspace{1in}T_{1,x_0}(e_1) + T_{2,x_0}(e_2)=0, \\
&
 T_{1,x_0}(e_1) + \be(x_0) (1,0)= T_{2,x_0}(e_2) + (1-\be(x_0))(1,0) = 0\;
\in  \; \{x_0\}\times  \R_t\times \R_{\theta} 
\end{align*}
together imply that the vector $(1,0)\in  \R_t\times \R_{\theta}$ is zero, a contradiction. Hence, as before,  the perturbed zero set consists of two points, each with weight one. \hfill$\er$

\MS

\NI {\bf (III)}  {\it  Construction of the corresponding manifold $M$ and section $\Sss_M:M\to E_{12}$.}
When, as in the case at hand, the isotropy groups are trivial, the current paper constructs from the above reduction $\Vv$ of  $\Kk$ a manifold $M$ that is the union of three components
$$
M = \Bigl((M_1=E_{2,\eps}\times V_1 ) \sqcup (M_2=E_{1,\eps}\times V_2 )  \sqcup (M_{12}=V_{12})\Bigr)/\!\sim,
$$
where $\sim$ identifies $(e_j,x)\in M_{i,12}$ with $\al_{i,12}(e_j,x)\in \TM_{i,12}\subset M_{12}$ where   $\al_{i,12}: = \tau^{-1}_{i,12}$.
The submersion axiom \eqref{eq:subm0}  implies that the submanifold $\TV_{i,12}$ has local product neighborhoods in $V_{12}$.  In \S2 we will describe how to assemble these into a more global structure that can be used to relate the different components $M_I$.  However, in the current situation
 there is an obvious global product structure that directly gives the needed attaching maps as follows.
First, with $i=1,j=2$, we define 
$$
\phi^E : (E_{2,\eps}\times \TV_{1,12}, \{0\}\times \TV_{1,12})\to (V_{12}, \TV_{1,12}), \quad (e_2,x)\mapsto \bigr(-T_{1,x}^{-1}(T_{2,x}(e_2)), e_2, x\bigl).
$$
Then,
the attaching map $\al_{1,12} = \tau_{1,12}^{-1}$ is given as follows: %has the form 
\begin{align*}
\al_{1,12}: &\;\;  E_{2,\eps}\times V_{1,12} \to V_{12},\\
 (e_2,x)&\; \mapsto \;
x' = \phi^E(\la \cdot e_2,x) =
\bigl(-T_{1,x}^{-1}(T_{2,x}(\la e_2)), \la e_2, \, x\bigr), \quad \la: = \sqrt{\|e_2\|}.
\end{align*}
Further, we take $\Sss_{12} = s_{12}$ where 
$$
s_{12}\bigr(-T_{1,x}^{-1}(T_{2,x}(e_2)), e_2, x\bigl)\; =\;  \bigr(-T_{1,x}^{-1}(T_{2,x}(\la e_2)), \la e_2\bigl),
$$
and then define $\Sss_1$ by pullback over $V_{1,12}$, extended over $M_1$ by a cut off function:
$$
\Sss_1(e_2,x) =  \be_{1,12} (x) \bigr(-T_{1,x}^{-1}(T_{2,x}(\la e_2)), \la e_2\bigl) + (1-\be_{1,12}(x) (0, e_2)
%\bigl( - \be_{1,12} (x)\,\tau_{1,x}^{-1}(\tau_{2,x}(e_2)), \; e_2\bigr)\;  \in\;  E_{12}.
$$
where $\be_{1,12}: V_1\to [0,1]$ equals $0$ near $x=0$ and $1$ on $V_{1,12}$.  Notice that $\tau_{1,12}$ does have closed graph in
$M_1\times M_2$ since $M_1$ contains no points $(e_2,x)$ with $x\in \{1/4\}\times S^1\subset A$, while $M_2$ contains no points $(e_1,e_2,x)$ with $x\in \{0\}\times S^1\subset A$.
There are similar formulas for $\al_{2,12}$ and $\Ss_2$.

This construction gives a $4$-manifold $M$ together with a map $\Sss_M: M \to E_{12}$ whose zero set is homeomorphic to $S^2$. 
In fact we can identify $M$ with a neighborhood of the zero section in $\rT S^2$ that has width $\eps>0$ over the discs $(V_i \less V_{i,12})_{i=1,2}$ and contains the whole of  $\rT S^2|_{G_{12}}.$  This holds because $V_{12}$ can be identified with 
$\rT S^2|_{G_{12}}.$ 
 \hfill$\er$
\MS

\NI {\bf (IV)}  {\it The normal bundle of $\Sss_M^{-1}(0)\cong S^2$ in $M$ is isomorphic to $\rT S^2$.}
%We claim that a neighborhood of this zero set is homeomorphic to the tangent bundle of $S^2$. 
 To see this, note that 
%for some neighborhood $\Nn( \TV_{1,12})$ of $ \TV_{1,12}$ in $V_{12}$ 
there is an embedding
$$
M_1\cup_{\al_{1,12}} \TM_{1,12} \to \C\times D_1
$$
given on $M_1=E_{2,\eps}\times V_1$ by the obvious inclusion (where we identify $E_2\equiv \C$) and on $ \TM_{1,12}$ by 
$$
\bigl(-T_{1,x}^{-1}(T_{2,x}(e_2)), e_2, x\bigr)\mapsto (\la^{-1} e_2, x)\in E_2\times A\subset \C\times D_1, \quad \la = \sqrt{\|e_2\|}.
$$
Identifying $A$ with  $(\eps, 1-\eps)\times S^1$ as above, we may extend this embedding over a neighborhood $\Nn_1\subset M_{12}$ of the set $\{(0,0)\}\times (\eps,\tfrac 12]\times S^1 $ so that
 it equals $$
 \bigl(-T_{1,x}^{-1}(T_{2,x}(e_2)), e_2, x\bigr)\mapsto (e_2, x), \qquad \forall\; x\in (\tfrac 12-\de, \tfrac12]\times S^1.
 $$  
 The similar embedding 
$$
(E_{1,\eps}\times V_2)\cup_{\al_{1,12}}\; \Nn_2 \to \C\times D_2
$$
is given near the  circle $  \{\tfrac 12\}\times S^1$ by the map $(e_1,- T_{2,x}^{-1}(T_{1,x}(e_1)), x)\mapsto (e_1,x)$.
Therefore this bundle over $S^2$ is determined by the clutching map $x\mapsto - T_{2,x}^{-1}(T_{1,x})$, which is homotopic to the map
$x\mapsto  T_{2,x}^{-1}(T_{1,x})$ that determines $\rT S^2$. \hfill$\er$
\MS

\NI {\bf (V)}  {\it The case of the football orbifold $S^2_{p,q}$.}
This orbifold is topologically $S^2$, but has orbifold points of orders $p,q$ at the two poles.  Thus the bundle
$\pi:\rT S^2_{p,q} \to S^2_{p,q}$ is again modelled by a Kuranishi atlas\footnote
{
The reader should beware that the words \lq orbifold  atlas' or \lq good atlas' are usually used in orbifold theory  with slightly different meaning,  which is why  \cite{Morb}  uses the words  \lq strict atlas' to denote a Kuranishi atlas with trivial obstruction spaces.  
As explained in \cite{Morb}, a strict atlas $\Kk$ for an orbifold $Z$ defines an ep groupoid  $\bG_\Kk$ whose realization is $Z$,  and hence defines an orbifold structure on $Z$.
Further, by \cite[Proposition~3.3]{Morb},  $\bG_\Kk$  is Morita equivalent to the category constructed from any  standard orbifold atlas for $Z$.  Finally  one can obtain a standard orbifold atlas for $Z$ from $\Kk$ by taking a collection of restrictions of the basic charts in $\Kk$ whose footprints cover $Z$, with transition maps induced by the morphisms in $\bG_\Kk$. 
}
 with two basic charts $\bK_1, \bK_2$ as above, with $\Ga_1 = \Z/p\Z$ acting by rotations on $D_1,E_1$ and with $\Ga_2 = \Z/q\Z$ acting by rotations on $D_2,E_2$. 
Since $s_i\equiv 0$ for $i=1,2$, the footprint maps $$
\psi_i: \s_i^{-1}(0)= U_i \to S^2_{p,q}, \quad x\mapsto |x|,
$$
simply quotient out by the action of the group $\Ga_i$.  
We choose  the trivializations $T_{i,x}$ of $\rT D_i$ to be equivariant under the rotation action of the isotropy groups, and will suppose for simplicity that 
$(p,q) = 1$  so that the domain   $U_{12}$ of the transition chart is connected.\footnote
{
\, Since all points in $U_{12}$ have trivial stabilizer,  we need $\Ga_{12}$ to act freely on $U_{12}$ in such a way that the projection $\rho_{j,12}$ quotients out by the action of $\Ga_i$, which is possible for connected $U_{12}$ only if 
$(p,q) = 1$.
}  
Then, in
 terms of the coordinates $(t,\theta)\in A$ introduced in (II) we have
\begin{align*}
U_{12} = \{(e_1,e_2,x)\in E_1\times E_2 \times A\ \big| \ |T_{1,\rho_{1,12}(x)}(e_1) | + |T_{2,\rho_{2,12}(x)}(e_2) | = 0\},\\
\rho_{1,12}(t,\theta) = (t,q\theta),\quad \rho_{2,12}(t,\theta) = (t, p \theta)\in A = [0,1]\times \R/\Z,
\end{align*}
 where we denote the image of  $(e,x)\in E_1\times D_1$ in $T_{|x|}S^2_{p,q}$ by  $|T_{1,x}(e)|$, and the equation 
 takes place in the tangent bundle of the orbifold.
Because the maps $\rho_{i,ij}$ are equivariant by hypothesis, this equation is preserved by the action
 of $\Ga_{12}$ on $U_{12}$ by
$$
\bigl(r/p,s/q)\cdot (e_1,e_2, (t,\theta)\bigr) = \bigl(r/p\cdot  e_1, s/q\cdot  e_2, (t, \theta + kr/p+\ell s/q)\bigr), \quad kq+\ell p = 1.
$$
 We may 
 calculate the Euler class by using essentially 
 the same perturbation section as before, since this may be chosen to be equivariant.  
  But now the two zeros of the section count with weights, $\frac 1p$ for the zero in $V_1$ and $\frac 1q$ for the zero in $V_2$.
 
The corresponding category $\bM$ has three components  that are given by the same formulas as before.
Again,  the attaching maps $\tau_{i,12}: \TM_{i,12}\to M_{i,12}\subset M_i$ are nontrivial covering maps.  However, 
in distinction to the case of an atlas, the $\tau_{i,12}$ do {\it not}	quotient by the induced action of $\Ga_j$ on $\TM_{i,ij}$ 
since they are constructed to be 
 $\Ga_{12}$ equivariant, and $\Ga_{12}$ acts  (often effectively) on $M_i$, via
 $$
 (\ga_1,\ga_2)\cdot (e_j,x_i) = (\ga_j\cdot e_j, \ \ga_i\cdot x_i).
 $$
 However, as explained at the end of the proof of Proposition~\ref{prop:M1} (see for example \eqref{eq:ga*}),
 they do quotient out by {\it some} action of $\Ga_j$ on $\TM_{12}$ that extends its free action on $\TV_{i,12}\subset \TM_{i,12}$.  
  For example, the map $\tau_{1,12}$ quotients out by the free action of $\Ga_q$ on 
$\TM_{1,12}\subset  E_1\times E_2 \times (\eps,\frac 14)\times S^1$ given by
$$
\ga\cdot (e_1, e_2, x)\mapsto (e_1, e_2,  \ga\cdot x).
$$
Therefore, in the quotient space $M= |\bM|$ there are $q$ branches of $M_{12}$  that come together over the $3$-dimensional branching locus
$$
Br_1: = \bigl\{|(e_1, e_2,x)|\in |M_{12}|\subset |\bM|_\Hh \ \big| \ x \in \tfrac 14 \times S^1\bigr\}.
$$
This is consistent with the requirements of  Definition~\ref{def:brorb} since the component $M_{12}$ has weight $1/pq$ while $M_1$ has weight $1/p$.

The construction of $\Sss_M:  M\to E_{12}$ is as before.  Moreover, one can identify a neighborhood of its zero set $S^2_{p,q}$  with a neighborhood of the zero section of the  tangent orbibundle to $S^2_{p,q}$ . Hence the Poincar\'e dual of $e(\rT S^2_{p,q})$ is represented by $$
(1/p + 1/q)[pt]\in H_0(S^2_{p,q}).
$$
\MS

\NI {\bf (VI)}  {\it The quotient space $|\bM|/\Ga$ for $\rT S^2_{p,q}$.} The only morphisms in the category $\bM$ come from the covering maps $\tau_{j,12}$.   Since these are $\Ga_{12}$-equivariant, we can add the  action $\Ga_{12}\times \Obj_{\bM}\to \Obj_{\bM}$ to  the morphisms in $\bM$.  The resulting quotient space $|\bM|/\Ga_{12}$ has the following structure.
\begin{itemlist}\item It 
is covered by three branches $M_1, M_2, M_{12}$  with  weights $1/p^2q$, $ 1/pq^2$ and $ 1/p^2q^2$;
\item  the two poles $[(0,0)]\in M_i/\Ga_{12}$ 
have stabilizer subgroup $\Ga_{12}$;
 \item the other points with nontrivial stabilizers lie on the two closed discs  $$
 \{0\}\times (\ov{V_i}\less \{0\})/\Ga_{12}\subset |M_i|/\Ga_{12},\quad i=1,2
 $$
with isotropy subgroups $\Ga_j, j\ne i$;
 \item  for $i=1,2$ there is  branching of order $|\Ga_j|$  over the $3$-dimensional branching locus  $Br_i$. For example, if $\Ga_1 = \{\id\}, \Ga_2= \Z/2\Z$, then 
 $|M_1|/\Ga_2$ is an orbifold  with a $2$-dimensional family of points with nontrivial stabilizer  (corresponding to the points $\{0\}\times D_1\subset E_2\times D_1$),
 while $\Ga_2$ acts freely on $M_{12}$ and the $\Ga_2$-equivariant map $\rho_{1,12}:M_{1,12}\to M_1$ quotients out by a different free action of $\Ga_2$ that lifts the rotation  action on $A$  via the projection $\TM_{1,12}\subset E_{12}\times A\to A$.  Thus there is branching of order $2$ along the boundary $Br_1$, which lies over the circle $t =\{1/4\}$.
 \end{itemlist}
 We do not consider this space further, since it plays no role in the definition of the fundamental class.\hfill$\er$
\end{example}
\MS

We end with a list of the main notation used in \S\ref{s:main} and \S\ref{s:out} 

\NI  {\bf (I: related to atlases)} 

in Theorem~A:  $\Kk, \; X,\; E_A,\;  \Ga_A,\;  \Ss_M:M\to E_A$.\;  in Theorem~B:  $[X]^{vir}_\Kk$.

in \S\ref{ss:defs}:  $\Kk, \; A, \; \Ii_\Kk,\;$ 
for $I\in \Ii_\Kk, \; \bK_I= (U_I, E_I,\Ga_I, s_I,\psi_I)$ ,\; $F_I\subset X$;

for $I\subset J$: $ \TU_{IJ}\subset U_J,\; U_{IJ}\subset U_I,\; \rho_{IJ}: \TU_{IJ}\to U_{IJ}$;
$\ga|_J, \Ga_A$ in \eqref{eq:GaA};

 constants:  $\ve$ in \eqref{eq:ve},  and $E_{I,\eps}$ for $I\subset A$ in \eqref{eq:Eeps};\; 
 $\phi_x^E$ in \eqref{eq:subm0}

$\Kk'\sqsubset \Kk$,\; $\Uu'\sqsubset \Uu = (U_I)_{I\in \Ii_\Kk}$ in \eqref{eq:Uu};  $\Ff = (F_I)_{I\in \Ii_\Kk}$ at beginning of \S3.

the categories $\bB_\Kk, \bE_\Kk$ in \eqref{eq:Kcat} ff and $|\Kk|: = |\bB_\Kk|$; $\pi_\Kk:U_I\to |\Kk|$ in \eqref{eq:piK}

reduction $\Vv = (V_I)$,\;  $G_I$ in \eqref{Vv} ff;\ $ \bB_\Kk|_\Vv$.   

 $V_{IJ}\subset V_I, \TV_{IJ}\subset V_J$
and $\pi_\Kk:V_I\to |\Kk|$  in \eqref{eq:VIJ}.

\MS

\NI  {\bf (II: related to wb manifolds)} 

in Definition~\ref{def:brorb}:  $(\bG, \La_\bG), \; |\bG|_\Hh,\; \pi^\Hh_\bG$

 in Proposition~\ref{prop:M} $(\bM, \La_\bM)$,\; $(M,\La),\; M_I, \; M_{IJ},\; \TM_{IJ}, \; \tau_{IJ}, \; \Sss_J,\;  \Sss,\; \Sss_M$
 
 in Theorem~\ref{thm:M}  $\; M= |\Hat{\bM}|_\Hh,\; |\Hat{\bM}|,\ ; |\Hat{\bM}_\Hh|$

\MS

\NI  {\bf (III: related to the manifolds $Y$)} 

 \S\ref{ss:Y} beginning: 
$t\in \De_J,\;  \p_{J\less I} \De_J,\; \io_{IJ}, \; t\cdot e, $; $\ka$,\;$I(x), I(t)$ in \eqref{eq:Ix};

$(e,x;t)\in Y_J= Y_{\Uu,J,\,\ve}$ in \eqref{eq:Y};  and $Y_{\Vv,J,\,\ve}$ in \eqref{eq:YV}

 $\pr_E, \pr_U, \pr_\De$ after \eqref{eq:Ixx};\; $\p_{J\less I} Y_J$ in \eqref{eq:YIJ};
 
 $b_H\in \De_H$ in \eqref{eq:bary}; $\io_{EU}$ in Corollary~\ref{cor:bary}; and $\io_{EV}$ in \eqref{eq:ioEV}. 
\MS

\NI  {\bf (IV: related to the collar)} 

$c_J^\De$ in \eqref{eq:collDe};   $c_J^Y: \p'Y_{\Vv,J,\,\ve}\times [0,w_j)\to Y_{\Vv,J,\,\ve}$ in \eqref{eq:coll1}

$\ov{\st}^\De_J(|x|)$ in \eqref{eq:starx};\,  $\p V_J$ in \eqref{eq:pVJ}.
$\Fr, cl$ in \eqref{eq:Fr}

\section{The main arguments}\label{s:out}

In this section, we first explain how to construct an auxiliary family of collared manifolds  and then explain in \S\ref{ss:M}  how to use this family to prove Proposition~\ref{prop:M1} and hence Proposition~\ref{prop:M}.  Finally,  we prove Theorems~A and~B in \S\ref{ss:topol}.

The  key notion is that of  the manifold $Y_{\Uu,J,\,\ve}$, which lies over the $(|J|-1)$-dimensional simplex $\De_J$. 
Its  open submanifold $Y_{\Vv, J,\,\ve}$, corresponding to a choice of reduction $\Vv\subset \Uu$,  has a partially defined boundary collar 
that is compatible both with shrinking of  chart domains and with  projection to  $\De_J$.  
 We will define the attaching maps $\TM_{IJ} \to M_{IJ}$ of the different components of $\Obj_{\bM}$ by thinking of $M_J$ as a subset of 
$Y_{\Vv, J,\,\ve}$.  

Although strictly speaking the construction of the category $\bM$ only uses the manifolds $Y_{\Vv, J,\,\ve}$,  we also consider 
the manifolds $Y_{\Uu,J,\,\ve}$ to clarify the exposition.  The latter
 has elements that 
are relatively easy to  understand (cf.  \eqref{eq:Ixx}) and it has an easily described boundary, while as we see from Proposition~\ref{prop:col} the collar is supported on only a rather complicated part of the boundary of $Y_{\Vv, J,\,\ve}$.
Further,   considering both $Y_{\Uu,J,\,\ve}$ and $Y_{\Vv,J,\,\ve}$ will allow us in \S\ref{s:det} 
 to introduce the many technical conditions satisfied by the pair $(\Vv,\ve)$ in stages, first  some conditions on $(\Uu,\ve)$ needed for
 $Y_{\Uu,J,\,\ve}$ to have good properties (Definition~\ref{def:compat}), and then more conditions needed  to construct a suitable collar on  $Y_{\Vv,J,\,\ve}$ (Definition~\ref{def:compatV}).

 The first main results of this section are Proposition~\ref{prop:Y} that describes the structure of $Y_{\Uu, J,\,\ve}$ and 
  Proposition~\ref{prop:col} that describes the properties of the boundary collars put on the manifolds $Y_{\Vv, J,\,\ve}$.  Proposition~\ref{prop:M1} then explains how to use these boundary collars to construct the attaching maps $\tau_{IJ}$ whose existence is claimed in  Proposition~\ref{prop:M}.  Since the general construction is quite complicated, we describe it first by example (see Example~\ref{ex:MJ}).    
Since the proofs of Theorems~A and~B in \S\ref{ss:topol} depend only on the statement of Proposition~\ref{prop:M},  this subsection can be read independently of \S\ref{ss:YY}, \S\ref{ss:M}.
 
\subsection{The collared manifold $Y$}\label{ss:YY}
Suppose given a tame atlas $\Kk$ with set of chart domains $\Uu: = (U_I)_{I\in \Ii_\Kk}$.   The next definition uses a choice of  constants $\ve = (\eps_I)$ as in \eqref{eq:ve}, and the following notation:
\begin{itemlist}\item[-]
$
\De_J: =\bigl \{t = (t_i)_{i\in J} \ \big|\  t_i\ge 0,\;  {\textstyle |t|: =\sum_{i\in J}} t_i = 1\bigr\}
$
is the $(|J|-1)$-simplex; \vspace{.06in}
\item[-] for  $\emptyset \ne I\subsetneq J$, we denote by  $\io_{IJ}: \De_I\to \De_J$  the natural inclusion with image 
$$
 \p_{J\less I} \De_J: = \{t\in \De_J \ \big| \  t_j=0, j\in J\less I\}  \subset \De_J;
$$ 
(we often omit $\io_{IJ}$ if there is no danger of confusion)
\item[-]
$
t\cdot e : ={\textstyle  \sum_{i\in J}} t_i e_i,\quad \mbox{ where  } t\in \De_J, e\in E_A;
$ \vspace{.06in}
\item[-]
 $\ka: = \max\{ |J|: J\in \Ii_\Kk\}$; \vspace{.06in}
\item[-] for $x\in U_J$, 
$
I(x) : = \{j: s_j(x)\ne 0\}\subset J 
$ and \vspace{.06in}
\item[-] $\ve: = (\eps_I)_{I\in \Ii_\Kk}$ is a set of positive constants such that $\ka\, \eps_I\le \eps_J$ whenever $I\subsetneq J$.
\end{itemlist}

Given $J\in \Ii_\Kk$, consider the set\footnote
{\, To begin with, readers should ignore the rather fussy conditions involving  the constants $\ve$; in this connection see \eqref{eq:bary} and   Corollary~\ref{cor:bary}  below.
Notice that we do need some such constants since the size of $\eps_J$ determines how thick the pieces $M_J$ will be, and to construct $\bM$ we need to 
embed (a covering of) $M_{IJ}$ into $M_J$ for all $I\subset J$.  
}
\begin{align}\label{eq:Y}
Y_J:&  = Y_{\Uu,J,\,\ve} \\ \notag
& = \Bigl\{(e,x; t)\in E_A \times U_J \times \De_J \ \Big|\  \begin{array} {ll} \; s_J(x) = t\cdot e,  & \|e\|< 
\ka\, \eps_{I(x)},\vspace{.05in} \\
 \|s_i(x)\|<\eps_{I(x)}& \forall i\in J\end{array}\Bigr\}.
\end{align}
\MS

 Here are some properties of this definition.
 \begin{itemlist}
 \item $\Ga_A$ acts on $Y_{\Uu,J,\,\ve}$ by
$$
\ga\cdot (e,x;t) = (\ga\cdot e, \ga\cdot x; t).
$$
\item
%Since $\Kk$ is tame, 
The condition $s_J(x) = t\cdot e$ implies that  
\begin{align}\label{eq:Ix}
I(x) : = \{j: s_j(x)\ne 0\} \;\subset \; I(t): = \{i:t_i>0\}.
\end{align}
In particular,  if $(e,x;t)\in Y_{\Uu, J,\,\ve}$ we must have 
\begin{align}\label{eq:Ixx}
x\in  s_J^{-1}(E_{I(x)}) =  \TU_{I(x)J} \subset  \TU_{I(t)J},
\end{align}
where the equality holds because $\Kk$ is tame (see \eqref{eq:cocyl}). 
Further,  the components of $e$ in $E_{I(t)}$ are determined by the pair $(x,t)$, while those in $E_{A\less I(t)}$ can vary freely.

\item   There  are three $\Ga_A$-equivariant projections of  $ Y_{\Uu,J,\,\ve}$ onto the factors of its domain.
\begin{itemlist}\item [-]
$
\pr_E: Y_{\Uu,J,\,\ve}\to E_A,\quad (e,x;t)\mapsto e.
$
For $I\subset A$, we  denote  by $e_I$ the elements of $E_I$, and 
denote by $\pr_{E_I}$ the projection to $E_I$.
%by $e|_I: = (e_i)_{i\in I}$ the projection of $e\in E_A$ onto $E_I$.
\item[-] The projection
$
\pr_U: (e,x;t)\mapsto x \in  U_J% \qquad \pr_\De: (e,x,t)\mapsto t\in \De_J.
$
has contractible fibers that vary with $x\in U_J$.
\item[-] 
 The fibers of $\pr_\De: Y_{\Uu,J,\,\ve}\to \De_J,\quad (e,x;t)\mapsto t $ also depend on the image $t\in \De_J$. 
  In particular, if  for some  $I\subsetneq J$  we have $t\in \intt \De_I: =  \De_I\less \p \De_I\subset \p \De_J$ then for any $(e,x; t)\in \pr_\De^{-1}(t)$, we must have $x\in \TU_{IJ}$ while the restriction $\pr_{E_{A\less I}}(e)$ can vary freely.
\end{itemlist} 
\item  For each element of the form $\bigl(e,x;\io_{IJ}(t)\bigr)\in Y_{\Uu,J,\,\ve} $ 
there is a corresponding element $(e,\rho_{IJ}(x); t)\in Y_{\Uu,I,\,\ve} $, where $\rho_{IJ}:\TU_{IJ}\to U_{IJ}$ 
is part of the atlas coordinate change.  Thus, if we define
\begin{align}\label{eq:deIJ}
\p_{J\less I} Y_J: = \pr_\De^{-1}(\p_{J\less I} Y_J): =\bigl \{(e,x;t)\in Y_J \ \big| \ t_j=0, j\notin I\bigr\},
\end{align}
 there is a 
$\Ga_A$-equivariant covering map
\begin{align}\label{eq:YIJ}
&\p_{J\less I} Y_J\to  Y_I\cap (E_A\times U_{IJ}\times \De_I)\subset Y_I. %\hspace{1in}\er
\end{align}
If the isotropy is trivial, we can therefore identify $\p_{J\less I} Y_J$ with an open subset of $Y_I$.

\item The relevance of the conditions involving the constants $\ve$ are explained by the following remark.
For each $x\in U_J$ such that  $ \|s_i(x)\|<\eps_{I(x)},  \forall i\in J$, and every $H$ satisfying
$I(x)\subset H\subset J$,
there is a corresponding element 
\begin{align}\label{eq:bary} \bigl(e,x;\io_{HJ}(b_H)\bigr) \;\in \;  Y_{\Uu,J,\,\ve},
\end{align}
 where $b_H$ is the barycenter of $\De_H$.  Indeed, if we take $e: = \bigl( |H| s_i(x)\bigr)_{i\in A}$, then $e_j = 0, j\notin I(x)$, by definition of $I(x)$, while  for $i\in I(x)$ we have
 $\|e_i\| = |H| \|s_i(x)\|< \ka \eps_{I(x)}$ as required by \eqref{eq:Y}.
% 
% 
% since $b_H$ has components $\frac 1{|H|} \ge \frac 1\ka$,  we 
%can take $e = \bigl( |H| s_i(x)\bigr)_{i\in A}$.
\end{itemlist}

\MS

The following result is proved in Corollary~\ref{cor:Y}.

\begin{prop}\label{prop:Y}  Let $\Uu^\Om$ be a family of chart domains for an atlas on $X$.  Without loss of generality, we may pass to a shrinking  $\Uu\sqsubset \Uu^\Om$  and choose  constants $\ve>0$
so that the following holds for all $J$.
\begin{enumerate}\item  $s_J(\ov {U}_J)\subset E_{J,\eps_J}$;
\item the space $Y_J: = Y_{\Uu, J,\, \ve }$ defined in \eqref{eq:Y} is a manifold  of dimension $D+|J|-1$ where $D: =  \dim E_A+d$;
\item $Y_J$ has  boundary given by 
\begin{align*}
\p Y_J: =  Y_J\cap \pr_\De^{-1}(\p \De_J) & ={\textstyle  \bigcup_{I\subsetneq J} \p_{J\less I}  Y_J} \\
 & =
{\textstyle  \bigcup_{I\subsetneq J} \bigl\{(e,x; t)\in  Y_J: x\in \TU_{IJ},} \; t\in \p_{J\less I}\De_J\bigr\}.
\end{align*}
\end{enumerate}
\end{prop}

\begin{cor}\label{cor:bary}  If Proposition~\ref{prop:Y} holds, then for all
  $I\subsetneq J$ there is an embedding $\io_{EU}: E_{A\less I,\eps_I}\times \TU_{IJ} \to Y_{\Uu, J, \,\ve}$ given by
\begin{align*}
\io_{EU}: \;\; (e_{A\less I},x) \mapsto (e_{A\less I} + b_I^{-1}\cdot s_I(x),x; b_I).
\end{align*}
\end{cor}
\begin{proof}  Since $s_J(\ov {U}_J)\subset E_{J,\eps_J}$ by (i), this holds by \eqref{eq:bary}.
\end{proof}
 
 Proposition~\ref{prop:Y} shows that the boundary of $Y_J$ lies over that of $\De_J$. 
It is well known that the  boundary of every topological manifold can be collared.  The next step is to show that we can construct this collar to have a special form, with control over the components  in $E_{A\less I} $ near the \lq corner' $\pr_\De^{-1}(\p_{J\less I} \De_J)$.  However, to establish this we need to pass to a  {\bf reduction} $\Vv = (V_I)_{I\in \Ii_\Kk}$ of the atlas (see \eqref{Vv}), since
this severely restricts  the overlaps $\pi_\Kk(V_I)\cap \pi_\Kk(V_J) $ in $|\Kk|$ of the different chart domains.
We define
 \begin{align}\label{eq:YV}
Y_{\Vv,J,\, \ve}: = Y_{\Uu,J,\,\ve}\cap (E_A \times V_J \times \De_J).  %, \quad Y_\Vv: = \bigsqcup_J Y_{\Vv,J,\,\ve}.
\end{align}
Since $Y_{\Vv,J,\,\ve}$ is an open subset of $Y_{\Uu,J,\, \ve}$, it is a manifold of dimension $d + \dim E_A + |J|-1$ with boundary
 $$
 \p Y_{\Vv,J,\,\ve} =  Y_{\Vv,J,\, \ve} \cap \p Y_{\Uu,J,\, \ve} \subset {\textstyle \bigcup_{I\subsetneq J}}\; E_A \times (V_J\cap \TU_{IJ}) \times \p_{J\less I} \De_J.
 $$
We denote by 
\begin{align}\label{eq:ioEV}
\io_{EV}: \;\;E_{A\less I,\eps_I}\times \p\TV_{IJ} \to Y_{\Vv, J, \,\ve}, \quad  (e_{A\less I},x) \mapsto (e_{A\less I} + b_I^{-1}\cdot s_I(x),x; b_I).
\end{align}
the restriction of the map $\io_{EU}$ in Corollary~\ref{cor:bary}, and will consider the projections
\begin{align*}%\label{eq:prV}
&\pr_V: Y_{\Vv,J,\, \ve} \to  V_J,\quad  (e,x;t)\mapsto x,\\ \notag
&\pr_{|V|}: Y_{\Vv,J,\, \ve} \to  |V_J|,\quad  (e,x;t)\mapsto |x|: = \pi_\Kk(x),
\end{align*}
where $\pi_\Kk$ is as in \eqref{eq:piK}.

There is a corresponding category with objects $\bigsqcup_{J \in \Ii_\Kk} Y_{\Vv,J,\, \ve}$ and morphisms given by the covering maps
\begin{align}\label{eq:rhoIJ*}
(\rho^Y_{IJ})_*: Y_{\Vv,J,\, \ve}\cap \bigl(E_A\times \TV_{IJ}\times \io_{IJ}(\De_I)\bigr) \to Y_{\Vv,I,\, \ve}, \quad \bigl(e,x; \io_{IJ}(t)\bigr)\mapsto (e,\rho_{IJ}(x); t).
\end{align}
This category has realization 
\begin{align*}%\label{eq:uY}
\uY_\Vv: ={\textstyle \bigcup_{J\in \Ii_\Kk}  Y_{\Vv, J,\, \ve}/\!\!\sim}
\end{align*}
where $(e, x; t)_I \sim (e', x'; t')_J $ for $|I|\le |J|$ if $I\subset J$, $e'=e$,\, $t' = \io_{IJ} (t)$, and
$\rho_{IJ}(x') =  x $.
Notice that the projections to $\De_J$  induce a map  
\begin{align*}%\label{eq:uY1}
 \pr_\De: \; \uY_\Vv\; \to\;  \De_\Kk = {\textstyle \bigcup_{J\in \Ii_\Kk}  \De_J/\!\!\sim}
\end{align*}
where  the simplicial complex $\De_\Kk$ (with  boundary identifications induced by the face inclusions $\io_{IJ}$) is the topological realization of the poset $\Ii_\Kk$.\footnote
{
\, The topological realization of  a topological category has one $k$-simplex for each length-$k$ composable string of morphisms, with the \lq obvious' boundary identifications.  Thus $\De_\Kk$ has one $k$-simplex for each $I\in \Ii_\Kk$ with $|I|=k+1$.
Observe that as the associated  footprint covering  $(F_I)_{i\in \Ii_\Kk}$ of the zero set $X$ is refined, the space $\De_\Kk$ gives better and better approximations to the topology of $X$:  indeed the \c{C}ech cohomology of   $\De_\Kk$ converges to  that of $X$.} 
There is also a projection 
\begin{align*} %\label{eq:prV2}
\pr_{|\Vv|}: \uY_\Vv\to |\Vv| \sqsubset |\Kk|,\qquad [e,x;t]\mapsto |x|.
\end{align*}

\begin{rmk}\rm  (i)  The projection $\pr_{|\Vv|} \times \pr_ \De$ induces a map   $$
\uY_\Vv \to \|\Vv\|'\;\; \subset\;\;  |\Vv|\times \De_\Kk,
$$
 whose image  $\|\Vv\|'$ is closely related to, but not the same as, the  topological realization 
$\|\bB_\Kk\big|_\Vv^{\less \Ga}\|$ of the category $\bB_\Kk\big|_\Vv^{\less \Ga}$ in \eqref{eq:TVIJ}.   For example, if $x\in V_J$ is such that its image $|x|: = \pi_\Kk(x)$ in $|\Kk|$ lies outside all the other sets $\pr_\Kk(V_I), I\ne J$,  then it gives rise to a single point in $\|\bB_\Kk\big|_\Vv^{\less \Ga}\|$ (since the only morphism involving $x$ is the identity morphism)  while it corresponds to a whole simplex $x\times \De_J$ in 
$\|\Vv\|'$.\footnote
{
\, If the isotropy is trivial, there is an embedding $\|\bB_\Kk\big|_\Vv^{\less \Ga}\|  \to \|\Vv\|'$ , whose image can be described  using versions of the  sets  $\ov\st^\De_J(|x|)$ in \eqref{eq:starx} below.
}
   The partial boundary $\p'\, Y_{\Vv,J,\, \ve}\subset \p Y_{\Vv,J,\, \ve}$ that we consider below could be understood in terms of an embedding of 
$\|\bB_\Kk\big|^{\less \Ga}_\Vv\|$  into  $\|\Vv\|'$.  However, we will take a more naive, geometric point of view.
\MS

\NI (ii) We saw in Remark~\ref{rmk:poly} that  in the polyfold setting one can use an sc-smooth partition of unity
to construct  a 
finite dimensional branched manifold $M$ with section $\Sss: M\to E_A$ that is a global chart for $X$.
One can think of the extra 
coordinates $t\in \De_J$ (with $\sum t_i = 1$) as a kind of \lq external' partition of unity that 
gives a more indirect way to patch the different coordinate charts together.   
\hfill$\er$   
\end{rmk}

\MS

 \NI {\bf The boundary collar.}
  We now consider lifts to $Y_{\Vv,J,\,\ve}$ of  the following  collar on $\p \De_J$
\begin{align}\label{eq:collDe}
 c_{J}^\De: \p \De_J\times [0,w) \to \De_J,\quad (t,r)\mapsto (1- r|J|)\,t + r |J|\,b_J,
\end{align}
 where $b_J = (\frac 1{|J|},\cdots, \frac 1{|J|})$ is the barycenter of $\De_J$ and $w< \frac1{4|J|}$;
see Figure~\ref{fig:3}.
Note that any $t\in \De_J$ with at least one component $t_i< w$ is in the image of this collar.  In order to
get maximal control over the collar 
we will not define it on all of $\p Y_{\Vv,J,\,\ve} $ since much of  $\p Y_{\Vv,J,\,\ve} $ is irrelevant to the task at hand.  Indeed,
we are only interested in boundary points $(e,x; t)$ with $x\in \TV_{IJ}$ for  $I\subsetneq J$ while, by Proposition~\ref{prop:Y},
 a general boundary point
has $$
x\in V_J\cap s_J^{-1} (E_I) = V_J\cap \TU_{IJ},
$$
 a set that  is usually strictly larger than the overlap $\TV_{IJ}$ (which is defined in \eqref{eq:VIJ}).
Although the submersion axiom \eqref{eq:subm0} implies that each $\TV_{IJ}$ is a submanifold in $V_J$ of codimension $\dim(E_{J\less I})$, 
we will make the following definition of the \lq boundary' of $V_J$: 
\begin{align}\label{eq:pVJ}
\p V_J: = {\textstyle \bigcup_{H\subsetneq J} }\TV_{HJ},
\end{align}
which lies over the  \lq boundary' $
 \p |V_J| =  \bigcup_{H\subsetneq J} |V_{HJ}| 
 %\;\subset \ol(|\Vv|) : =   {\textstyle \bigcup_{I\subsetneq L} } |V_{IL}|
 $ of $|V_J|$.

We will define the collar $$
c_J^Y: \p'\, Y_{\Vv,J,\,\ve} \times [0,w_J)\to Y_{\Vv,J,\,\ve},
$$
over a subset $\p' \,Y_{\Vv,J,\,\ve} $ of points $(e,x;t)\in  \p Y_{\Vv,J,\,\ve} $ such that  $x\in \p V_J$ and $t$ is restricted  
 to  lie in the set $\ov\st_J^\De(|x|)$ defined as follows.
 Recall that
for each  $x\in V_J$ the sets $H$ such that $|x|: = \pi_\Kk(x) \in \pi_\Kk(V_H)$  (where $\pi_\Kk: V_J\to |\Kk|$ is the projection 
\eqref{eq:piK})
form a chain
\begin{align}\label{eq:chainI}
I: = I_{\min}(|x|) = I_0(|x|)\subsetneq  I_1(|x|)\subsetneq \cdots \subsetneq  I_{m}(|x|) = I_{\max}(|x|)=:K.
\end{align}
If $  J = I_n(|x|),\;  n\le  m$  we will write
\begin{align}\label{eq:starx}
\ov\st_J^\De(|x|): & = \conv(b_{I_0}, b_{I_1}, \dots, b_{I_{n-1}}) \subset  \p_{J\less I_{n-1}(|x|)} \De_{J},
\end{align} 
for the convex hull of the barycenters of the simplices corresponding to the elements of this chain: see Figure~\ref{fig:fund1}.
Note that $\ov\st_J^\De(|x|)$ lies in the boundary of $\De_J$.

\begin{figure}[htbp] %  figure placement: here, top, bottom, or page
   \centering
  \includegraphics[width=5in]{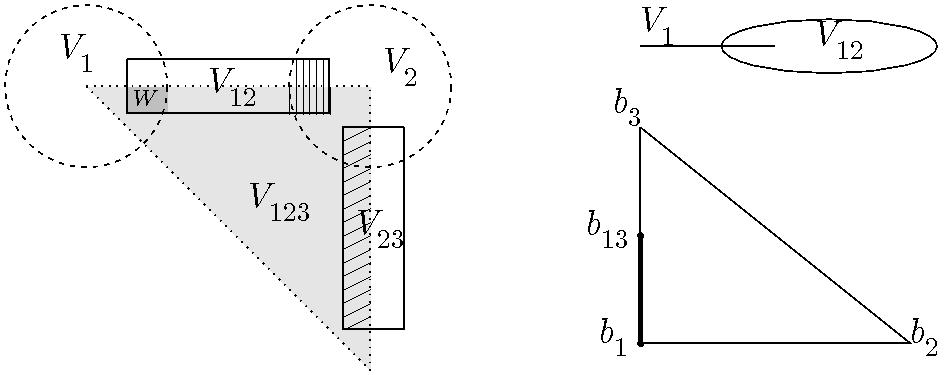} 
   \caption{The figure on the left is schematic, showing the sets $|V_I|$ rather than their (disjoint) lifts $V_I$; the sets $V_{2,12}\subset V_2$ and $V_{23,123}\subset V_{23}$ are hatched, while for $x$ in the  shaded set $W$, we have $ I_{\min}(|x|) = \{1\},  I_1(|x|) = \{1,2\}, I_{\max}(|x|) = \{1,2,3\}$.
 The top right illustrates the change in dimension from $V_1$ to $V_{12}$, while the bottom right shows $\ov\st_J^\De(|x|)$ for  $x \in \TV_{1,123}\cap \TV_{13,123}$.}
   \label{fig:fund1}
\end{figure}

The domain $\p'\, Y_{\Vv, J, \,\ve}\subset \p Y_{\Vv, J, \,\ve}$ of the collar map $c^Y_J$ contains all the points in the image of the injections $\io_{EV}$ in \eqref{eq:ioEV}, as well as the lifts to 
$Y_{\Vv, J, \,\ve}$ of all points in 
 $\im(c^Y_H)$
 where $I\subsetneq H \subsetneq J$.  
 To obtain points with more general $t$-coordinate 
we  consider the following {\bf rescaling operation}.
Suppose given $t\in \De_J$ and a tuple $\mu_J= (\mu_j)_{j\in A}$  such that $\mu_j=1$ if $ j\notin J$, $\mu_j>0,\;\forall j$, and $\mu_J\cdot t \in \De_J$.
Then for any element $(e,x;\,t)\in Y_{\Vv,J,\,\ve}$,
%For any map $T: \De_J\to \De_J$ and any element $(e,x; t)\in Y_{\Vv,J,\,\ve}$  such that $(T_i(t) = 0)\Longrightarrow (s_i(x) = 0)$ 
there is a commutative diagram
\begin{align}\label{eq:TDe}
\xymatrix {
(e,x;\, t) \ar@{|->}[d]_{\pr_{E_{A\less J}}\times \pr_V} \ar@{|->}[r]^{\mu_J\,\cdot \qquad\; } %^{T^Y\qquad\quad} 
&  \bigl((\mu_J)^{-1}\cdot e,x; \, \mu_J\cdot t\bigr)\ar@{|->}[d]^{\pr_{E_{A\less J}}\times \pr_V }\\
(e_{A\less J}, x)   \ar@{|->}[r]^{=} & (e_{A\less J}, x) ,
}
\end{align}
where we assume $\|(\mu_J)^{-1}\cdot e\| < \ka\, \eps_{I(x)}$ so that
the top arrow  has target in $Y_{\Vv, J,\,\ve}$.
\MS

The following result concerns a reduction $\Vv$ plus choice of constants $\ve$ that are {\bf compatible} in the sense of Definition~\ref{def:compatV}. In particular this means that property (i) in Proposition~\ref{prop:Y} holds, and that $(\Vv, \ve)$ is compatible with a fixed choice of local product structures as in \eqref{eq:subm0}. The proof is given in Lemma~\ref{le:col} below.

 \begin{prop}\label{prop:col}    Let  $(\Vv,\,\ve)$   be a  compatible reduction of an atlas $\Kk$. Then for each $J\in \Ii_\Kk$ there is an open subset
 $\p'\,  Y_{\Vv,J,\, \ve}\subset \p\,  Y_{\Vv,J,\, \ve}$, a constant $w_J>0$, and a $\Ga_A$-equivariant embedding 
\begin{align}\label{eq:coll1}
c_{J}^Y: \p'\,  Y_{\Vv,J,\, \ve}\times [0, w_J) \to Y_{\Vv,J,\, \ve},\quad
\bigl((e,x;t),r\bigr) \mapsto  (e',x'; c_J^\De(t,r)),
\end{align}
with the following properties:\footnote
 {\, The precise definition of $\p'\,  Y_{\Vv,J,\, \ve}$ may be found in \eqref{eq:dkY} and \eqref{eq:colYY}.
By slight abuse of language we will call $\p'\,  Y_{\Vv,J,\, \ve}$ the {\bf domain} of $c_J^Y$.
}
\begin{itemlist}\item  
$\p'\,  Y_{\Vv,J,\, \ve} \subset \bigl\{(e,x;t): \exists I\subsetneq J,  x^0\in \TV_{IJ}, \mbox{ s.t.}\;\;  x\approx x^0,\ 
 t\in \ov\st^\De_J(|x^0|)\bigr\}$,
\item 
$c_{J}^Y$  is {\bf compatible with the projections} to $E_{A\less \bullet}$  as follows:
 we have
 \begin{align}\label{eq:coll19}
\io_{EV}(E_{A\less I, \eps_I}\times \TV_{IJ}) \; \subset\; \p'\,  Y_{\Vv,J,\, \ve},\quad  \forall\,\, I\subsetneq J.
\end{align}
  Further, % and 
\begin{align}\label{eq:coll20}
& c_{J}^Y\bigl((e,x;t), 0\bigr) = (e,x;t),\quad \forall (e,x;t)\in \p'\,Y_{\Vv,J,\, \ve},\\ \notag
& \pr_{E_{J\less I}}(e) = 0\;\; \Longrightarrow\;\; c_{J}^Y\bigl((e,x;t), r\bigr)  =\bigl( e,x; c^\De_J(t,r)\bigr), \mbox{ and }
\end{align}
\begin{align}\label{eq:coll2}
&\pr_{E_{A\less I}} \circ c_{J}^Y\bigl(\io_{EV}(e,x), r\bigr) =  \pr_{E_{A\less I}} (e),\quad \forall (e,x)\in E_{A\less I, \eps_I}\times \TV_{IJ},\end{align}
\item the sets $\p'\,  Y_{\Vv,J,\, \ve}$ are {\bf compatible with covering maps} as follows:
 if $I\subsetneq H\subsetneq J$, then   the relevant part of the image of $c^Y_H$ lifts to the domain $\p'\, Y_{\Vv, J,\,\ve}$ of $c^Y_J$.  More precisely, 
if $(e,x;t)\in \p'\,  Y_{\Vv,J,\, \ve}$ has $x\in \rho_{HJ}^{-1}(\TV_{IH})\cap \TV_{HJ}$\footnote
{
\, By \eqref{eq:VIJ}, when $I\subsetneq H\subsetneq J$ any two of the sets $\TV_{IJ}, \TV_{HJ},  \rho_{HJ}^{-1}(\TV_{IH})$ determine the third.}
 and $t\in \p_{H\less I} \De_H$, then 
$(e,\rho_{HJ}(x);\,t)$ is in the domain $\p'\,  Y_{\Vv,H,\, \ve}$ of $c^Y_H$ and
for all $r\in [0,w_H)$ there is 
 $(e',x';\,t') \in \p'\,  Y_{\Vv,J,\, \ve}$ with $x'\in \TV_{HJ}$ such that 
 \begin{align}\label{eq:collH}
c^Y_H((e,x;t),r) = \bigl(e',\rho_{HJ}(x');\, t'\bigr)  \; \in \; Y_{\Vv,H,\ve}.
 \end{align}
 Further,  the restriction of $c^Y_H$ to   
 $ Y_{\Vv,H,\, \ve} \cap \pr_V^{-1}( \TV_{IH}\cap V_{HJ} )$  has a well defined lift 
 (also called $c^Y_H$) to $Y_{\Vv, J,\, \ve} $ such that for all $x\in  \TV_{IJ}\cap \TV_{HJ}$
 \begin{align}\label{eq:coll4}
 (\pr^Y_{HJ})_*\bigl( c^Y_H(e,x;t),r\bigr) = \bigl( c^Y_H(e,\rho_{HJ}(x),t),r\bigr) \; \in \; Y_{\Vv,H,\ve},  \quad r\in [0,w_H), 
 \end{align}
where $ (\pr^Y_{HJ})_*$ is as in \eqref{eq:rhoIJ*}.
\item each $\p'\,  Y_{\Vv,J,\, \ve} $  is {\bf invariant under rescaling} as follows:  
if  $(e,\, x;\, t)\in \p'\, Y_{\Vv,J,\, \ve}$ where $t\in \ov\st_H^\De(|x|)$
 then for all $\mu_H$ as in \eqref{eq:TDe} such that  $\mu_H\cdot t\in  \ov\st_H^\De(|x|)$ we have
 $$
 \mu_H\cdot \bigl(e,x; \, t \bigr): = \bigl(\mu_H^{-1}\cdot e,x; \, \mu_H\cdot t \bigr)\in \p'\, Y_{\Vv,J,\, \ve}
 $$ 
 and
\begin{align}\label{eq:rescal}
& \pr_{E_{A\less H} \times V}\circ c^Y_J\bigl((e,\, x;\, t), r\bigr) = \\ \notag
& \qquad  \pr_{E_{A\less H} \times V}\circ c^Y_J\bigl( (\mu_H^{-1}\cdot e,x; \, \mu_H\cdot t ), r\bigr) \in E_{A\less H}\times V_J;
\end{align}
\item  the collar maps $c^Y_J$ are {\bf compatible with shrinkings} as follows: if $(\Vv', \ve')\sqsubset (\Vv, \ve)$ is another compatible reduction, then there are constants $0<w_J'< w_J$ such that the restrictions of the maps $c^Y_J$ 
to $\p'\, Y_{\Vv', J, \,\ve'}: = \p Y_{\Vv', J, \,\ve'} \cap \p'\, Y_{\Vv, J, \,\ve}$ have all the above properties with respect to the constants $w_J'$.
\item  if $\Kk$ is oriented then the collar map $c^Y_J$ is  compatible with the natural induced orientation on its domain and range.
\end{itemlist}
\end{prop}

By Lemma~\ref{le:Veps}, any reduction $\Vv''$ has a shrinking  $\Vv\sqsubset \Vv''$ that is compatible with respect to some choice of constants $\ve$ and hence supports 
a collar $\bigl(c^{Y}_J\bigr)_{J\in \Ii_\Kk}$ as in  Proposition~\ref{prop:col}.  
Further,
we  show in Corollary~\ref{cor:col}  that  $(\Vv^\infty,\ve^\infty)$ has a further nested shrinking that is collar compatible 
 in the following sense.

\begin{defn}\label{def:collcompat} 
Let $(\Vv^\infty,\ve^\infty)$ be a compatible reduction, with collars $\bigl(c^{Y,\infty}_J\bigr)_{J\in \Ii_\Kk}$.  
We say that a shrinking $(\Vv,\ve)\sqsubset (\Vv^\infty,\ve^\infty)$ is {\bf collar compatible} if it is compatible as in Definition~\ref{def:compatV} and if
for all $J\in \Ii_\Kk$ the collar map $c^{Y,\infty}_J$ restricts to a collar $(c^Y_J)_J$ on $(\Vv,\ve)$ whose widths $w_J$ 
satisfy
$ \sqrt {\eps_I}< w_J$ for all $I\subsetneq J$.  
\end{defn}

\subsection{Construction of the category $\bM$ and functor  $\Sss:\bM\to \bE_A$}\label{ss:M}
   
%    We use these manifolds to define the structural maps in the category $\bM$ and functor  $\Sss:\bM\to \bE_A$ as follows.

In \eqref{eq:Mm}, the  component $M_J$ of $\Obj_\bM$ was defined as
 \begin{align}\label{eq:MJ}
M_J = E_{A\less J,\eps_J}\times V_J,
\end{align}
which is a manifold of dimension $d + \dim E_A$.
We take $M_{IJ}: = E_{A\less J,\eps_J}\times V_{IJ}$, and define the  map $\tau_{IJ}: \TM_{IJ}\to M_{IJ}$  that attaches $M_J$ to $M_I$
to have domain a suitable open subset $ \TM_{IJ}\subset M_J$ and to extend the atlas structural map
$$
\rho_{IJ}: \{0\}\times \TV_{IJ}\to \{0\}\times V_{IJ} \subset M_{IJ}\subset M_I.
$$
We require that $\tau_{IJ}$ is a $\Ga_A$-equivariant covering map, induced by a free action of $\Ga_{J\less I}$.
Further, to obtain a category, these maps must be   compatible with composition:
 i.e. for $I\subset H \subset J$ we need
\begin{align}\label{eq:taucomp}
\tau_{HJ}\circ \tau_{IH} = \tau_{IJ} \;\; \mbox{ on }\;\; \TM_{IJ}\cap \TM_{HJ} \cap \tau_{HJ}^{-1}(\TM_{IH}) = \TM_{IJ}\cap \TM_{HJ}.
\end{align}
(Note that by  \eqref{eq:VIJ} any two of the sets $\TM_{IJ}, \TM_{HJ}, \tau_{HJ}^{-1}(\TM_{IH})$ determine the third.)
For maximal elements $J$ of $\Ii_\Kk$, we then define $\Sss_J: M_J\to E_A$ as the projection
\begin{align*}%\label{eq:WsJ2}
\Sss_J: M_J\to E_A,\quad& (e_{A\less J},x)\mapsto  (e_{A\less J}, s_J(x)). 
\end{align*}
The above  should be considered as the default formula for $\Sss_J$, that holds at points $(e_{A\less J},x)\in M_J$ where $x$ is far from any overlap $V_{JK}$ with $J\subsetneq K$.  However, 
 in general it must be modified in ways explained in Example~\ref{ex:MJ} below.
\MS

Before giving  the general formulas for $\mu_J, \tau_{IJ}, \Sss_J$, we discuss an example. Part (i) shows the role of the collar in constructing $\tau_{IJ}$, 
and also how to achieve the closed graph condition in \eqref{eq:clgraph}, while 
part (ii) explains the relevance of the collar's compatibility with projections and  rescaling to the proof of the composition rule~\eqref{eq:taucomp}.  The usefulness of considering
multiple collar compatible shrinkings $(\Vv^n, \ve^n)$ will also become apparent.
We will use  cutoff 
functions $\bigl(\chi_{IJ}: V_I\to [0,1]\bigr)_{I\subsetneq J}$ of the following form:   if $\Vv\sqsubset \Vv'$  we  have 
\begin{align}\label{eq:be0}
{\textstyle \supp(\chi_{IJ})\subset \bigcup_{I\subsetneq H\subset  J}V_{IH}',\quad\mbox{ and } \;\;  
 \bigcup_{I\subsetneq H\subset  J} \ov V_{IH} \subset \intt (\chi_{IJ}^{-1}(1)).}
\end{align}
\MS

\begin{example}\label{ex:MJ}\rm {\bf (Attaching the $M_J$).}\, 
We begin by considering the case when the isotropy groups are trivial, so that $\tau_{IJ}: \TM_{IJ}\to M_{IJ}$ is a homeomorphism.
It is then easiest to define its inverse
$$
\al_{IJ}: = \tau^{-1}_{IJ}: M_{IJ}\to \TM_{IJ},
$$
since $M_{IJ}\subset M_I$ is defined to be the product $E_{A\less I,\de_I}\times V_{IJ}$ (where $V_{IJ}$ is defined in \eqref{eq:VIJ})
while $\TM_{IJ}$ will simply be defined as the image
$\al_{IJ}(M_{IJ})$. 
As in \cite{MW2}, we use the notation 
$\phi_{IJ}: = \rho_{IJ}^{-1}: V_{IJ}\to \TV_{IJ}$ for the inverse of the atlas structural map $\rho_{IJ}$.
\MS

\NI
(i)  Consider the case when 
there are two basic charts with labels $1,2$. 
Then $\bM$ has  three components:\footnote
{
\, Here we simplify notation by writing $M_{12}: = M_{\{1,2\}}, M_{1,12}: = M_{\{1\} \{1,2\}}$ and so on.
For an example of this construction, see \S\ref{ss:exam}.
}
 $$
M_1= E_{2,\de_2}\times V_1,\quad  
M_2 = E_{1,\de_1}\times V_2,\quad  M_{12}: = V_{12},
$$
where we assume 
$(\Vv, \vd)$ is collar compatible as in Definition~\ref{def:collcompat}. 
In particular, this means 
that for $i=1,2$ we have $\de_i < w_{12}^2$, where $w_{12}$ is the  width of the collar $c^Y_{12}$.
We first  define
 the attaching maps $\al_{1,12}$ and $\al_{2,12}$, then define  the sections  $\Sss_I$. and finally  prune the sets $M_{12}$ so as to satisfy the closed graph condition.

We define $\al_{1,12}$ as a composite $M_{1,12}: = E_{2,\de_2}\times V_{1,12}\; \to \;Y_{\Vv, 12,\vd}\; \to \;M_{12}$:
 \begin{align}\label{eq:al112} \notag
\al_{1,12}\bigl( (e_2,x)\bigr) &= \pr_V\Bigl(c_{12}^Y\bigl(\io_{EV}(e_2,x),\, r\bigr)\Bigr) \;\;  \mbox{ with } r: = \sqrt{ \|e_2\|} \\
& = \pr_V\Bigl(c_{12}^Y\bigl((s_1(x), e_2, \phi_{IJ}(x); b_1),r\bigr)\Bigr) \;\; \\ \notag
& = \pr_V\Bigl(\bigl(e_1',e_2, x'; (1-\ r, r)\bigr)\Bigr),\;\;\\ \notag
& = x'   \in V_{12} = M_{12}, 
\end{align}
where $\io_{EV}$ is the map in \eqref{eq:ioEV},
$b_1 = (1,0)$ is the barycenter of $\De_1$ considered as a point in  $\De_2$, 
 we have used formula \eqref{eq:collDe} for $c^\De_{12}$, and we have used the fact
from \eqref{eq:coll2} that $e_2$ is unchanged by $c_{12}^Y$.    We note the following.
\begin{itemlist}\item   Because  $(\Vv,\de)$ is collar compatible, Definition~\ref{def:collcompat} implies that
 the collar width satisfies $w_{12} > \sqrt{ \|\de_2\|}> r$.  Hence the 
 element $c_{12}^Y\bigl((s_1(x), e_2, \phi_{IJ}(x); b_1),r\bigr)$ is well defined for all $(e_2,x)\in M_{1,12}$.
\item Because the collar variable $r: = \sqrt{ \|e_2\|}$ vanishes 
 for the points   $(0,x)\in M_{1,12}$,  the map $ \al_{1,12}$
extends the inclusion
$\phi_{IJ}: V_{IJ}\to \TV_{IJ}$ by \eqref{eq:coll20}, as is required by  Proposition~\ref{prop:M}~(i).  Further, 
 for small enough $\de_i$ the closures of the  images of $ \al_{1,12}$ and the similarly defined map $ \al_{2,12}$ are  disjoint.
\item  Because  the points $(e,x;t)\in Y_{\Vv, J,\,\ve}$ satisfy $s_J(x) = t\cdot e_J$ and we chose $r = \sqrt{ \|e_2\|} $, 
we have  $$
r \,\|e_2\| =  (\|e_2\|)^{3/2} = \|s_2(x')\|,
$$ so that $r = \|s_2(x')\|^{1/3}$ is determined by $x'$.
% we can make a similar definition for 
\item To see that
$ \al_{1,12}$  
is injective, notice that because $c_{J}^Y$ is injective it suffices to check that the other elements, $e_1', e_2,r$ that appear in the tuple
$\bigl(e_1',e_2, x'; (1-r, r)\bigr) \in Y_{\Vv,12,\,\ve}$ are determined by $x'\in V_2$.   But we saw above that $r=  \| s_2(x')\|^{1/3}$, so that 
the equations $s_1(x) = (1-t)e_1', \ s_2(x) = te_2$ determine $e_1', e_2$.
\end{itemlist}

We now define $\Sss_{12}:= s_{12}: M_2=V_2\to E_{12}$, and define $\Sss_i $ on $\al_{i,12}^{-1}(\TM_{i,12})$ by pullback: thus on this set
$$
\Sss_i(e_j,x) = \bigl(\|e_j\|^{1/2} e_j,\, s_i(\al_{i,12} (e_j,x)\bigr),\quad i\ne j,
$$
has the form claimed in \eqref{eq:sJ}.
We then extend $\Sss_i$ to the rest of $M_i$ by patching it to the default map  $(e_j,x) \mapsto (e_j\, s_i(x))\in E_j\times E_i = E_{12}$
via  the
 cutoff  $\chi_{i}$  
 in  \eqref{eq:be0}:
\begin{align}\label{eq:Sss1}
\Sss_i(e_j,x) & = \chi_{i,12} (x) \bigl(\|e_j\|^{1/2} e_j,\, s_i(\al_{i,12} (e_j,x)\bigr)  + (1- \chi_{i,12} (x))\bigl(e_j,\, s_i(x)\bigr) \in E_{12}.
\end{align}
For this to be well defined, we need  $\al_{i,12}$ to extend to a neighborhood of $M_{i,12}$ in $M_i$.  But we can always assume that $\Vv$ is a shrinking of some other reduction $\Vv'$.  Then 
because the collar extends over $\Vv'$ 
we may extend $\al_{i,12}$ over the corresponding set $M_{i,12}'$ by using the above  formula \eqref{eq:al112}.
It is then clear that $\Sss_i^{-1}(0) = \{0\}\times s_i^{-1}(0)$. 
\MS

It remains to arrange that $\al_{i,12}$ has closed graph.  
Note that its restriction to $\{0\}\times V_{i,12}$ does have closed graph because $\Vv$ is a reduction of a good atlas $\Kk$, which among other things implies that the realization $|\Vv|\subset |\Kk|$ is Hausdorff; see the discussion around \eqref{eq:piK}, \eqref{Vv}.
Denote by 
\begin{align}\label{eq:Fr}
\Fr(M_{i,12}): = cl(M_{i,12})\less M_{i,12}
\end{align}
 the frontier of $M_{i,12}$ in $M_i$, where, as usual, $cl$ denotes the closure.
As above, we may assume that $\al_{i,12}$ extends to a homeomorphism $\al_{i,12}: cl(M_{i,12})\to V_{12}'$, which evidently has a closed graph.   Hence  it suffices  to arrange that
$V_{12}\cap \al_{i,12}(\Fr(M_{i,12}))  = \emptyset.  $  But $$
V_{12}\cap \, cl\bigl(\al_{i,12}(\Fr(M_{i,12}))\bigr) \subset cl(\TV_{i,12})\less \TV_{i,12}
$$
 is a closed subset of $V_{12}$ that is disjoint both from $cl(\TV_{j,12})$ (by the separation property of the sets $\TV_{1,12}, \TV_{2,12}$) and from the zero set
$s_{12}^{-1}(0)$ (because    $|\Vv|$ is Hausdorff).   Hence, as in Figure~\ref{fig:5}, if this set is nonempty we can simply remove it from $V_{12}$, i.e. we replace $V_{12}$ by 
\begin{align}\label{eq:prune}
V_{12}\less \textstyle{\bigcup_{i=1,2}} cl\bigl(\al_{i,12}(\Fr(M_{i,12}))\bigr).
\end{align}
 \begin{figure}[htbp] %  figure placement: here, top, bottom, or page
   \centering
   \includegraphics[width=4in]{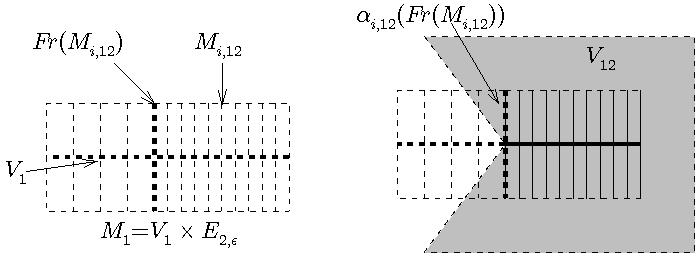} 
   \caption{Removing points from $V_{12}$ so that $\al_{i,12}$ has closed graph.  Since $V_{12}$ is open the point where the two heavy lines cross is not in $V_{12}$.  The set $\TM_{i,12}\;\cong\; \TV_{i,12}\times E_{i,\eps_i}\subset V_{12}$ is hatched.}
   \label{fig:5}
\end{figure}

\NI
(ii)  Now suppose that the atlas $\Kk$ has three basic charts with labels $1,2,3$, so that the sets $V_I$ in the reduction $\Vv$  intersect as in  Figure~\ref{fig:fund1}.
We assume that the isotropy is trivial and all $E_i\ne 0$, and again  explain how to choose the constants $\de_i$, and define the attaching maps $\al_{IJ}$ and sections $\Sss_I$ that involve the vertex $1$, 
namely those with labels $1,12,13,$ and $123$.  It is now convenient to assume that we have four nested 
collar compatible  shrinkings  $(\Vv^1,\ve^1) \sqsubset  (\Vv^2,\ve^2) \sqsubset  (\Vv^3,\ve^3) \sqsubset  (\Vv^4,\ve^4)$ of $\Vv'$.  
Correspondingly, with $I\subset \{1,2,3\}$ and $k\le \ell \le 4$ we define
$$
M_I^k: = E_{I, \eps_i^k}\times V^k_I,\quad M_{IH}^{k,\ell} =  E_{I, \eps_i^k}\times V^{k,\ell}_{IH}\; \subset\;  M_{I}^{k},
$$
where
$$
V^{k,\ell}_{IH} = V_I^k \cap V_{IH}^\ell = V_I^k  \cap \pi_\Kk^{-1}\bigl(\pi_\Kk(V_H^\ell)\bigr)
$$
We aim to define
a category with basic domains of the form  $M_I^{|I|}$ and compatible morphisms
$\al_{IH}: M_{IH}^{|I|,|H|} \to M_H^{|H|}$.
However, to make these continuous and to define the corresponding maps $\Sss_I$ we have to define transition functions 
on larger  sets such as $M_{IH}^{|I|+1,|H|}$.  As in (i), we will first define suitable maps $\al_{IH}$  and sections $\Sss_I$, and then will prune domains
to achieve the closed graph condition.\MS

If $|I|=1, |H|=2$ we define $\al_{IH}: M_{IH}^{1,2}\to M_H^2$ as in (i) above.  
These  methods also  easily adapt to define 
the maps $\al_{IH}$ for $|I|= 2$
and  
$\Sss_I: M_I\to E_A$ for $|I|\ge 2$. % and compatible transition maps $\al_{IH}$ for $|I|\ge 2$.
Indeed, if $J: = \{1,2,3\}$ then % $\Sss_J$ is determined by \eqref{eq:sJ}, while 
$$
\al_{1i, J}:  M_{1i,J}^{2,3} \to   M_J^3  %: = \mu_{J}\circ \al_{1k,J}
$$
can be defined much as in \eqref{eq:al112}.  The only new point is that because $\De_{1i}$ is a $1$-simplex,
we have to decide how to lift $V_{1i, J}$ to $\p Y_{\Vv, J,\,\ve}$ in order to use the collar.  For now, we use the 
 default choice given by the embedding $\io_{EV}$ in \eqref{eq:ioEV}, i.e. we embed it over the barycenter $b_{1i}$ of $\De_{1i}$ which we identify with the corresponding point $\io_{1i,J}(b_{1i})$ in $\De_{J}$.  Thus with $i\ne j,\  i,j \in \{2,3\}$, we define 
 \begin{align}\label{eq:al1k123}
\al_{1i,J}: &\; M_{1i,J}^{2,3}  \to  M_J^3:\; (e_j,x)\mapsto x' \;\;\mbox{ as follows: } \\ \notag
 E_{3, \eps^2_i}\times V_{1i}^{2,3}\ni &\; (e_j,x) \longmapsto % \\ % \begin{array}{l}  \\
 %c_{J}^Y\Bigl(\bigl( 2\cdot s_{1i}(x), e_j, \phi_{1i,J}(x); b_{1i}\bigr), r\Bigr),\quad r= \sqrt{ \|e_j\|}\\ \notag
 c_{J}^Y\Bigl(\bigl( \io_{EV}(e,x)\bigr), r\Bigr),\quad r= \sqrt{ \|e_j\|}\\ \notag
 & =   \bigl(e'_{1i}, e_j, x'; c^\De_J(b_{1i}, r)\bigr) \in Y_{\Vv^3, J,\,\ve^3} \\ \notag
 &  \longmapsto x' \in M^3_J. 
\end{align}
Since $r$ depends on $e_3$ and hence on $s_3(x')$ as above, it follows as before that 
 $\al_{1i,J}$ is injective. Notice also that if $x\in V_{1i, J}^{2, \ell }$ the point $\phi_{1i,J}(x)$ would lie in $\TV^\ell_{1i,J}$
% 
% $\p Y_{\Vv^\ell, J,\,\ve^3}$,
 as would its image $x'$ under the collar map since the collar maps preserve the shrinkings by  Proposition~\ref{prop:col}.
Taking $\ell = 4$ here, we may therefore 
 define $\Sss_{12}$ by pullback from $\Sss_{J}$ on $M_{12,J}^{2,3}$, tapering it off to the product
 $s_{12}\times \pr_{E_j}$ outside the larger set $E_{j,\eps_{1i}} \times V_{1i,J}^{2,4}$ by using the cutoff functions $\be_{1i,J}$ as in \eqref{eq:Sss1}.

  The main new task is to  define
$$
\al_{1,J}: M_{1,J}^{1,3}\to M_{J}^3,\mbox { so that } \al_{1,J}: = \al_{1i,J}\circ \al_{1,1i} \mbox{ in } M_{1,J}^{1,3}\cap M_{1,1i}^{1,2}.
$$
If $x\in V_{1,J}^{1,3}\less \bigcup_{i=2,3} V_{1,1i}^{1.3}$, (i.e. $x$ is \lq\lq far" from $V_{1,1i}^{1,2}$) then   we may define
\begin{align}\label{eq:al1123i}
\al_{1,J}(e_{23},x) = \pr_{E_3\times V}\Bigl(c^Y_{J}\bigl((s_1(x), e_{23}, \ \phi_{1,J}(x), \ b_1), r\bigr)\Bigr),\;\; r =\sqrt{\|e_{23}\|},
\end{align} 
as in \eqref{eq:al112}.
Hence the lift of $\al_{1,J}(e_{23},x)$ to $Y_{\Vv^3,  J, \,\ve^3}$  lies over the ray 
$c^\De_{J} \bigl(b_1 \times [0,w_0]\bigr)\subset \De_J$.
On the other hand, the composite $\al_{1i,J}\circ \al_{1,1i}$ first uses the collar $c^Y_{1i}$ for $b_1$ in $\De_{1i}$ and then the collar $c^Y_J$ of $b_{1i}$ in $\De_J$, and hence its natural lift to $Y_{\Vv^3,  J, \,\ve^3}$  is rather different.
  We  interpolate between these two maps 
as follows,  
where  we take $i=2$ for clarity, and use cut-off functions $\be_{1,12}$ as in \eqref{eq:be0}, with support in $V_{1,12}^{1,3}$ and that equal $1$ near the closed set $\ov V_{1,12}^{1,2}\sqsubset V_{1,12}^{1,3}$. Thus with $x\in V_{1,12}^{1,3}\cap V_{1,J}^{1,3}$, we define
\begin{align}\label{eq:al1J}\notag 
(e_{23},x) &\longmapsto\; c^Y_{12}\Bigl(\bigl(s_1(x),e_{23},\phi_{1,12}(x); b_{1}\bigr), \ r \Bigr) \mbox{ where }  r: = \be_{1,12}(x)\sqrt{\|e_2\|}\\ \notag
 &\hspace{1in}  =: (e_1', e_{23},x'; 1-r,r) \; \in \; Y_{\Vv^3,12,\,\ve^3} \\ 
& \longmapsto\; c^Y_{J}\bigl( (e_1',e_{23}, \phi_{12,J}(x');\, 1-r,r), r'\bigr) =:  \bigl(e_1'', e_{23}, x''; t''\bigr)\in Y_{\Vv^3,J,\,\ve^3} \\ \notag
&\hspace{1in}  \mbox{ where }\;\; %\begin{array}{l}  r: = \be_{1,12}(x)\sqrt{\|e_2\|}\\
r': = \max \bigl((1-\be_{1,12}(x))\sqrt{\|e_{2}\|}, \sqrt{\|e_3\|}\bigr)\\ \notag
& \longmapsto\;  x''=:\al_{1,J}\bigl((e_{23},x)\bigr) \; \in \; M_{J}^3 =  V_{J}^3 %\end{array}
\end{align}
Note the following.
\begin{itemlist}\item   Here (as in \eqref{eq:coll4}) we consider $c^Y_{12}$ to be the lift  to $\p'\, Y_{\Vv,J,\,\ve}$ of the collar 
for $\p'\, Y_{\Vv,12,\,\ve}$, and the  
  composite  $c^Y_J\circ c^Y_{12}$ is defined by \eqref{eq:collH}.
\item
The above map $(e_{23},x)\mapsto x''$ is continuous,  and equals that given in \eqref{eq:al1123i} when $\be_{1,12}(x) = 0$ because $\|e_{23} \| = \max\{ \|e_i\|: i=2,3 \}$ by definition. 
\item If
 $x\in V_{1,J}^{1,3}\cap V_{1,12}^{1,2}\subset V_{1,J}^{1,3}\cap (\be_{1,12}^{-1}(1))$,  then $$
 \al_{1,12}(e_{23},x)= \al_{12,J}\circ \al_{1,12}(e_{23},x).
 $$
Indeed, the invariance of the collar  under rescaling in \eqref{eq:rescal}  shows that applying the second collar map at $(1-r,r)$  with $r' = \sqrt{\|e_3\|}$
 and then projecting to $M_J^3$ gives the same result as rescaling, then applying the second collar at $b_{12}$ with the same $r'$, and then projecting to  $M_J^3$. Note that 
by \eqref{eq:coll2} this last claim holds even if $e_3 = 0$, so that the second collar map has $r'=0$ when $\be_{1,12}(x) = 1$.
\item  It remains to check that this map $(e_{23},x)\mapsto x''$ is injective. Since the first two maps in \eqref{eq:al1123i} are injective, it suffices to check that the projection
$(e_1'', e_{23}, x''; t'')\to x''$ is injective.
But  both collar maps preserve $e_2,e_3$ by the extended corner control in \eqref{eq:coll2}.
Hence, for $i=2,3$  we know  $\|e_i\|$ and therefore $t_i''$  from $s_i(x'') = t_i'' e_i$. Since  $\sum_i t_i''=1$, we therefore know $t''$
and hence also $e'' = (e_1'', e_{23})$.
\end{itemlist}

As before, we define $\Sss_1$ by pullback via $\al_{1,*}$ over $E_{23,\eps_1}\times \bigcup _{\{1\}\subsetneq  J} V_{1,J}$, extending to the rest of $M_1$ via a cutoff function $\be_{1,J}$.
However, to do this we need the pullback of $\Sss_1$ to be compatibly defined on a set that is larger than that on which we ultimately want $\Sss_1$ to equal the pullback.  But we can arrange that the identity 
$\al_{1,J} = \al_{12,J}\circ \al_{1,12}$ actually holds on a neighborhood of the closure of $V_{1,12}^{1,2}\cap V_{1,J}^{1,3}$, 
since in \eqref{eq:al1J} $\be_{1,12} = 1$  on a  neighborhood of $\ov V_{1,J}^{1,3}$, and we can always extend the domain of 
$\al_{1,12}$ to $V_{1,2}^{1,3}$.  Therefore we can imitate the formula in \eqref{eq:Sss1}.  

It remains to prune the domains $M_J$ so as to achieve the closed graph condition for all maps $\al_{IJ}$.  We will do this by downwards recursion on $I$. Thus, first taking $|I|=2$, we remove points from $M_{123}$ so that the maps $\al_{I,123}$ have closed graph, and then with $I = \{i\}$ remove points from all $M_J$ with $|J|\ge 2$ so that the maps $\al_{i,J}$ have closed graph.  At each stage we use the analog of formula \eqref{eq:prune}, removing
 from $V_J$ all points in $\cl_{V_J}\bigl(\al_{IJ}(\Fr(M_{IJ}))\bigr)$ where $\Fr(M_{IJ})$ is the frontier of $M_{IJ} = M_{IJ}^{|I|,|J|}$ in $M_I: = M_I^{|I|,|I|}$.
 Since  $\Fr(M_{IJ})\subset M_{IJ}^{|I|,|J|+1}$, the
 points removed   lie in the image of the extension of the collar over $Y_{J,\Vv^{|J|+1},\,\ve}$ but not in the image of the collar over 
 $Y_{J,\Vv^{|J|},\,\ve}$.  Hence, because the $\al_{IJ}$ are defined in terms of the collar map,  these points do not lie in  $\im \al_{HJ}$ for any $H\subsetneq I\subsetneq J$.  Thus the different steps do not interfere with each other.

\MS

\NI (iii)  If the isotropy is nontrivial, then we can still adopt the above approach, but now must interpret $\al_{IJ}$  as a local $\Ga_x$-invariant inverse to $\tau_{IJ}$ and then define $\TM_{IJ}$ to be the $\Ga_A$-orbit of its image. 
Further, we must make equivariant constructions, 
but this is possible since the collar is  equivariant, so that all the above formulas are appropriately equivariant.   In particular, the sets that must be removed 
in order to achieve the closed graph condition for the local inverse  $\al_{IJ}$ are $\Ga_x$-invariant, so that we can arrange that
$\tau_{IJ}$ has closed graph by removing its $\Ga_A$ orbit.  \hfill$\er$ 
 \end{example}

The next  result is essentially a restatement of Proposition~\ref{prop:M}, though it gives a little more information on the nature of the map $\tau_{IJ}$.
  Since the proof is rather complex, we describe the strategy here. 
%The role of the key ingredients  have already  b, 
As  in Example~\ref{ex:MJ}, we define the maps $\al_{IJ}$ by downwards recursion on the cardinality $|I|$ of the index set $I$, shrinking domains at each step.  In order to extend the interpolation formula for $\al_{IJ} = \tau_{IJ}^{-1}$ given in \eqref{eq:al1J} to a chain of inclusions $I_0\subsetneq I_1\cdots \subsetneq I_k$ of length $k>1$,
we apply an iterated sequence of collar maps over a family of paths $\Pp(e,x)$ in the simplex $\De_J$ as described in Step 2 below.  We then define the attaching maps $\tau_{IJ}$ and $\Sss_I$, and check that they have the needed properties.

\begin{prop}\label{prop:M1}  Suppose given a good atlas $\Kk$ on $X$.  Then  there is a reduction $\Vv$ and set of constants $\vd = (\de_I)_{I\in \Ii_\Kk} >0$, such that the following properties hold with
$$
M_I: =  E_{A\less I,\de_I}\times V_I,\quad M_{IJ}: =  E_{A\less I,\de_I}\times V_{IJ}.
$$
\vspace{-.15in}
\begin{itemlist} \item[{\rm (i)}]
For each $I\subset J$ there are open sets $\TM_{IJ}\subset M_J$ and $\Ga_A$-equivariant maps
$$
\tau_{IJ}: \TM_{IJ}\to M_{IJ} 
$$
that restrict to $\rho_{IJ}$ on $\{0\}\times \TV_{IJ}$ and are
such that \begin{itemlist}\item
 $\TM_{IJ}$ is a product $E_{A\less J, \de_I}\times \TM^0_{IJ}$ where 
 $\TV_{IJ}\subset \TM^0_{IJ}\subset V_J$
and $cl(\TM^0_{IJ})\cap cl(\TM^0_{HJ}) = \emptyset$ unless $I,H$ are nested, and
\item  $\tau_{IJ}= \id_E \times \tau_{IJ}^0$ where 
$\tau_{IJ}^0: \TM_{IJ}^0\to E_{J\less I, \de_I}\times V_{IJ}$ has the following properties:
\begin{itemize}\item[-]  $\tau_{IJ}^0(x) = \bigl(0,\rho_{IJ}(x)\bigr)$  for $x\in \TV_{IJ}$;
\item[-] $\tau_{IJ}^0$ has closed graph; and
\item[-]
$\tau_{IJ}^0$ quotients out by a free action of $\Ga_{J\less I}$ that extends to a free action on a neighborhood of $cl(\TM_{IJ}^0)$ in $V_J$.
\end{itemize}
%\item[-]  Moreover  
%\begin{align}\label{eq:taurestr} \tau_{IJ}^0|_{\TV_{IJ}} = \rho_{IJ}.
%\end{align} 
\end{itemlist} 
\item[{\rm (ii)}]   for $I\subsetneq J\subsetneq K$ we have 
\begin{align} \label{eq:tauIJK} \tau_{JK}(\TM_{IK}\cap \TM_{JK}) &= \TM_{IJ}\cap M_{JK},\quad \mbox{  and }\;\;
\tau_{IK}  = \tau_{IJ}\circ \tau_{JK}.
\end{align}
\item[{\rm (iii)}] For each $J$ there is $\Sss_J : M_J\to E_A$ such that for all $J\subset K,$ we have
\begin{align}\label{eq:sJ1}
&\Sss_{J}\circ \tau_{JK} = {\Sss_K}|_{\TM_{JK}} , \quad \Sss_J^{-1}(E_J) \subset \{0\}\times V_J  \; \mbox{ and }  \\  \notag
& \qquad \qquad %\Sss_J^{-1}(0) =\bigl \{(0,x)\ \big| \ s_J(x) = 0\bigr\}.
\Sss_J(0, x)  = (0,s_J(x)).
%e_{A\less J}, x)  = \bigl(e_{A\less J}, \Sss_{J|J}(e_{A\less J}, x)\bigr) \in E_{A\less J}\times E_J, \quad \mbox{ where } \\ \notag
\end{align}
\item[{\rm (iv)}]  If the initial atlas $\Kk$ is oriented, then so is the category $\bM$ defined by the above data as in \eqref{eq:Mm}.
\end{itemlist} 
\end{prop}
\begin{cor} Proposition~\ref{prop:M} holds.
\end{cor}
\begin{proof} 
If the category $\bM$ is defined as in \eqref{eq:Mm} using the above  data $M_I, \TM_{IJ}, \tau_{IJ} $, then all the properties of
Proposition~\ref{prop:M} hold.
\end{proof}

\begin{proof}[Proof of Proposition~\ref{prop:M1}]  {\bf Step 1:}  {\it The set-up and basic strategy of proof.}

Fix a shrinking $\Gg^0 = (G_I^0)_{I\in \Ii_\Kk}$ of the footprint cover. 
By Corollary~\ref{cor:col}
 we may choose a family of nested collar compatible shrinkings as above
$$
\psi^{-1}(\Gg^0) \sqsubset (\Vv^1,\,\ve^1) \sqsubset \cdots \sqsubset (\Vv^{\ka+1},\,\ve^{\ka+1})\sqsubset (\Vv^{\infty},\,\ve^{\infty})  \sqsubset \Uu^\infty,
$$
with collar widths that increase with $m$.  The projection  $\pi_\Kk: U_I^\infty \to |\Kk|$ quotients out by $\Ga_I$ and its restrictions to the $\Vv^m$ have the property that $$
\pi_\Kk(\ov{V^k_I}) \cap \pi_\Kk(\ov{V_J^\ell}) \ne \emptyset\;\Longleftrightarrow \; I\subset J \mbox{ or } 
J\subset I.
$$
  For $m\le \ell$ we denote $V^{m,\ell}_{IJ} : = V_I^m\cap \pi_\Kk^{-1}(\pi_\Kk(V^\ell_J))$, and for $m \le |I|,\; m\le \ell \le |J|$
define
\begin{align}\label{eq:mIJ0}
M_I^m: =  E_{A\less I,\eps_I^m}\times V_I^m, \qquad M_{IJ}^{m,\ell} = E_{A\less I,\eps_I^m}\times V_{IJ}^{m,\ell} .
\end{align}
For each $I\subsetneq J$ and $m\le |I|$ we will define  $\Sss_J: M_J^{|J|}\to E_A$,
a  subset $\TM_{IJ}^{m,\ell}\subset M_J^\ell$ and a $\Ga_A$-equivariant covering map
$$
\tau^{m,\ell}_{IJ}:  \TM_{IJ}^{m,\ell} \to M_{IJ}^{m,\ell}
$$ 
with the following properties:
\begin{itemlist}\item[{\rm (a)}]  $\tau^{m,\ell}_{IJ}$ has product form and closed graph as in (i), and quotients out by a free action of $\Ga_{J\less I}$
on $ (\TM_{IJ}^{m,\ell})^0$;
\item[{\rm (b)}]   for all $m\le m'\le |I|, \ell\le \ell'\le|J|$, $\TM_{IJ}^{m,\ell} \subset \TM_{IJ}^{m',\ell'}$  and  $\tau^{m',\ell'}_{IJ}|_{ \TM_{IJ}^{m,\ell}} = \tau^{m,\ell}_{IJ} $;
\item[{\rm (c)}]  if $I\subsetneq H\subsetneq J$  then $\tau_{IJ}^{|I|,|J|} = \tau_{HJ}^{|H|,|J|}\circ \tau_{IH}^{|I|,|H|}$ on their common domain; moreover this domain maps onto
$$
E_{A\less I, \eps_I} \times \Bigl(V_{IJ}^{|I|,|J|}\cap \rho_{IJ}(V_{HJ}^{|H|,|J|})\Bigr) \subset M_I^{|I|}.
$$
\item[{\rm (d)}]   if $I\subsetneq J$ then $\Sss_I\circ \tau_{IJ} = \Sss_J$ on $\TM_{IJ}^{|I|,|J|}$;
\item[{\rm (e)}]  $\Sss_J^{-1}(0) = \{0\}\times s_J^{-1}(0)\subset M_J^{|J|}$.
\end{itemlist}
In the end we will take 
\begin{align*}% \label{eq:MII}
M_I: = M_I^{|I|}, \quad  M_{IJ}: = M_{IJ}^{|I|,|J|}
\end{align*}
 with the corresponding sets $\TM_{IJ}^{|I|,|J|}$, and  the restrictions of the maps $\tau_{IJ}$ and  $\Sss_I$.  In particular, $\de_I = \eps_I^{|I|}$.

\MS

For simplicity, we first assume that the isotropy groups are trivial. 
As in Example~\ref{ex:MJ} (see in particular \eqref{eq:al1J})
for $I\subsetneq J$ 
  we will define a family of
 injective  maps
$$ 
\al_{IJ}: M_{IJ}^{|I|+1, |J|+1}\cap \{(e,x)\, \big|\; \|e_{J\less I}\|< \eps_I^{|I|}\} \;\; \to \;\; 
M_J^{|J|+1} , \quad  \la\ge 1
% M_J^{|J|+1} \cap \bigl(E_{A\less J}\times \TV_{IJ}^\infty\bigr), \quad  \la\ge 1
$$ %\end{align} 
(where $e_{J\less I}: = \pr_{E_{J\less I}}(e)$) with well defined restrictions
\begin{align}\label{eq:b}
&\al_{IJ}: = \al_{IJ}\big|_{M_{IJ}^{m,k}} : M_{IJ}^{m,k} \to M_J^k, \quad m\le |I|+1, k\le |J|+1, m \le k,
\end{align}
 such that
\begin{align}\label{eq:c}
& \al_{IJ} = \al_{HJ}\circ \al_{IH} \mbox{ on }  M_{IJ}^{|I|,|J|}\cap 
\al_{IH}^{-1}(M_{HJ}^{|H|,|J|}), \quad \forall \; I\subsetneq H \subsetneq J.
\end{align}
Then we define
$$
\TM_{IJ}^{m,\ell} = \al_{IJ} (M_{IJ}^{m,\ell}),\quad  \tau_{IJ} = \al_{IJ}^{-1}.
$$
With this, conditions (b), (c) will hold and $\tau_{IJ} = \al_{IJ}^{-1}$ has the required product form.  We will arrange the rest of (a) later.  
\MS

\NI {\bf Step 2:}  {\it Definition of  $\al_{IJ}$ via  the paths $\Pp(e,x)$.}

\begin{figure}[htbp] %  figure placement: here, top, bottom, or page
   \centering
 \includegraphics[width=2.5in]{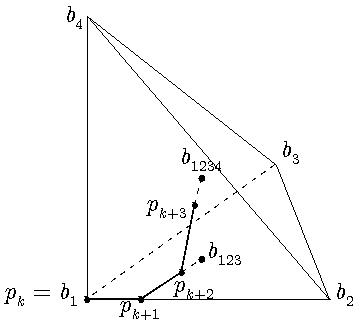} 
   \caption{The path $\Pp(e,x)$ with $I_k = \{1\},\dots, I_{k+3} = \{1,2,3,4\}$. }
   \label{fig:2}
\end{figure}

To define $\al_{IJ}(e,x)$ 
we consider the chain of length $m = m(|x|)$
formed by the sets $H$ such that $|x|\in |V_H|^{|H|+1}$ 
\begin{align}\label{eq:chain2}
I_{\min}(|x|) = I_0(|x|)\subsetneq  I_1(|x|)\subsetneq \cdots \subsetneq  I_{m}(|x|) = I_{\max}(|x|),
\end{align}
modifying the definition of $\ov\st^\De_J(|x|)$ from \eqref{eq:starx} accordingly.
Extending the procedure in \eqref{eq:al1J}, if $I = I_{k}(|x|)$ we  define $\al_{IJ}(e,x)$  by applying  collar maps in
$Y_{\Vv^{\ka+1}, I_m, \ve^{\ka+1}}$ a total of
$m- k$ times with initial points $p_{n-1}\in \pr_\De^{-1}(\De_{I_{n}})$ and collar lengths $r_n$ for $n=k+1,\dots, m=m(|x|)$.
% that   as in \eqref{eq:al1J} depend on cut-off functions  as well as the quantities  $a_n: = \sqrt{\|e_{I_{n+1}\less I_n}\|}$.
In fact, it is useful to think of applying the iterated collar map that lies over the path $\Pp(e,x)$ in  $\ov\st^\De_J(|x|)$
with the following vertices:
 \begin{align*}%\label{eq:p}
&p_k= b_{I_k}, \;\; %p_{m+1} = (1-r_1) b_{I_m} + r_1b_{I_{m+1}} = c_{I_m}^\De(p_m, r_1)., \;\;\cdots\;\; \\ \notag
  p_{n} = (1-r_n) p_{I_{n-1}} + r_nb_{I_{n}} = c_{I_{n}}^\De(p_{n-1}, r_n), \quad k<n\le m(|x|),
\end{align*}
(see Figure~\ref{fig:2}) where the $r_n$ are described below.
Note that by the collar compatibility with covering maps in \eqref{eq:coll4}  
it makes no difference whether at the $n$th step we apply the  collar map over the segment $[p_{n-1}, p_n]$ in $Y_{\Vv,I_n, \ve}$ 
(where $\Vv: = \Vv^\ka$)
 and then lift to  the next level $Y_{\Vv,I_{n+1}, \ve}$,  or whether we first lift all the way to $Y_{\Vv,I_{m}, \ve}$ (where $I_m = I_{\max}$), and then apply the collar maps.
We take the second approach, first   lifting
the initial point $(e_{A\less I_k}, x)$ to 
 \begin{align*}%\label{eq:phicol}
(e_{A\less I_k}+ b_{I_k}^{-1}\cdot s_{I_k}(x),\, \phi_{I_k I_m}(x); b_{I_k})\in 
\p_{I_{m}\less I_{k}} Y_{\Vv, I_{m}, \ve^{\ka+1}} \cap \p'_{I_{m}\less I_{0}} Y_{\Vv, I_{m}, \ve^{\ka+1}}
\end{align*}
%(see \eqref{eq:} and \eqref{eq:} for the definition of $\p'_{I_{m}\less I_{0}} Y_{\Vv, I_{m}, \ve^{\ka+1}}$) 
and then applying successive collar maps that remain in the boundary $\p'\, Y_{\Vv, I_{m}, \ve^{\ka+1}} $ until the very last step.
%The rather complicated definition of  $\p'_{J\less I} Y_{\Vv, J, \ve}$ may be found in \eqref{eq:dKY} and \eqref{eq:colYY}.  
Note that  by the collar compatibility with shrinkings we can work in $\Vv: =\Vv^\ka$ rather than in the different $\Vv^i$.

To complete this definition of $\al_{IJ}(e_{A\less I},x)$ it remains to define the lengths $r_n = r_n(x)$ for $k+1\le n \le m$.
To achieve consistency with coordinate changes, for each $I\in \Ii_\Kk$, we choose a cutoff function $\chi_I:|\Kk|\to [0,1]$ such that
\begin{align}\label{eq:chi}
& \supp(\chi_I)\subset \pi_\Kk(V_I^{|I|+1}),\qquad \chi_I^{-1}(1)\subset \pi_\Kk(V_I^{|I|}),
\end{align}
and for each $J$ denote its pullback to the set $V_J^{|J|+1}$ by the same letter.  Then, writing $a_n: = \sqrt{\|e_{I_n\less I_{n-1}}\|}$
and $\chi_i: = \chi_{I_i}$, we 
define
\begin{align*}%\label{eq:rn}
r_{m+1}(x): & = \chi_{{m+1}}(x) a_{m+1},\;\; r_{m+2}(x) = \chi_{{m+2}}(x) \max\bigl((1-\chi_{{m+1}}(x))a_{m+1}, a_{m+2}  \bigr), \dots
\\ \notag
 r_n(x) : &= \chi_n(x) \Bigl( \max_{m< j\le n} \la_j a_j\Bigr),\quad  \la_j:= \prod_{i=j}^{n-1} (1-\chi_i(x)),\;\;  j<n,\;\;  \la_n: = 1.
\end{align*}
To check that  $\al_{IJ}(e_{A\less I},x)$ is well defined we note the following.
\begin{itemlist}
\item The path $\Pp(e,x)$ depends both on the position of $|x|$  with respect to the sets $|V_H|^{|H|+1}$ in the chain \eqref{eq:chain2}, and on the relative sizes $a_k$ of the relevant components of $e$; cf. equations \eqref{eq:al1J}.
\item In order for the collar maps to be defined over $\Pp(e,x)$, we must have $r_n(x)< w_{I_n}$  for all $n$.  But $$
r_n\le \max_{m<j\le n} \sqrt{\|e_{I_n\less I_{n-1}}\|} < \sqrt{\|e_{A\less I}\|} < \sqrt{\eps_{I_m}}< w_{I_n}
$$
for all $m>n$ because $(\Vv,\ve)$ is collar compatible: see Definition~\ref{def:collcompat}.
\item Further at each stage we need the image of the iterated collar map to lie in the domain of the next collar map
It follows from \eqref{eq:coll19} that the initial point $$
(e_{A\less I_k}+ b_{I_k}^{-1}\cdot s_{I_k}(x),\, \phi_{I_k I_m}(x); b_{I_k})$$
 of $\Pp(e,x)$ does lie in
$\p'\, Y_{\Vv,J,\,\ve}$.  One then uses the fact that these domains  $\p'\, Y_{\Vv,J,\,\ve}$ are compatible with covering maps, as explained in 
\eqref{eq:collH},\eqref{eq:coll4}.
\item
To see that the path $\Pp(e,x)$ varies continuously with $x$, it suffices to check continuity for a sequence of points $x^\nu\to x^\infty$ for which just one of the functions $\chi$ ---  say $\chi_s$ --- changes from a  positive value to zero.
But in this case (assuming that $e$ is fixed) the functions $r_i(x)$ are continuous for $i<s$, while for $i\ge s$ we have
\begin{align*}
& a_i^\nu = a_i^\infty, i\ne s, s+1,  \quad a_s^\infty = \max (a_s^\nu, a_{s+1}^\nu),\\
&\lim_\nu r_{i}(x^\nu)  = r_i(x^\infty), i<s, \quad
\lim_\nu r_s(x^\nu) = 0, \quad \lim_\nu r_{i}(x^\nu)  = r_{i-1}(x^\infty), i>s.
\end{align*}
\item  If $x\in V_{IH}^{|I|+1, |H|}$ for $H = I_s$ where $m<s< n$, then $\chi_s(x) = 1$.  In this case,
we can divide $\Pp(e,x)$ into two independent segments at the point $p_s$, because the lengths   $r_n(x), n>s$ no longer depend on $a_i, i\le s$ since $\la_i = 0$ for $i\le s$.   Further, the second part of $\Pp(e,x)$ projects  to the path $\Pp(\phi_{IH}(x))$ under the natural projection $$
\bigl(\conv(b_{I_0},\dots, b_{I_{\max}(|x|)})\bigr)  \less \bigl(\conv(b_{I_0},\dots, b_{I_s}) \bigr)\;  \to \; \conv(b_{I_s},\dots, b_{I_{\max}}).
$$
\end{itemlist}

\NI {\bf Step 3:}\ {\it Definition of the maps $\al_{IJ}$ and sections $\Sss_I$ in the case of trivial isotropy.}

With these formulas in hand, we now define
 the maps $\al_{IJ}$ and sections $\Sss_I$ by downwards recursion on $|I|$.  
 For $|J|=\ka: = \max\{|J|: J\in \Ii_\Kk\}$, we define 
 \begin{align*}
  \Sss_J:& =  \Sss_J',\quad \mbox{ where }\\
    \Sss_J': & \; M_J\to E_A,\quad  (e_{A-J},x)\mapsto (e_{A\less J}, s_J(x)).
\end{align*}
 If
$|I| = \ka-1$,  for $x\in V_{IJ}^{|I|+1, |J|}$ the path $\Pp(|x|)$ has one segment of length $\chi_{I} a_k: =\chi_{I}  \sqrt{\|e_{J\less I}\|}$,
 and we define $\al_{IJ}:M_{IJ}^{|I|, |J|}\to M_J^{|J|}$ by applying the collar map as in \eqref{eq:al1k123}.
 For these values of $x$ we have $\chi_I(x) = 1$.  However the fact that  we have defined $\al_{IJ}$ over the larger set
 $V_{IJ}^{|I|+1, |J|}$ means that the function
\begin{align}\label{eq:SssH}
 \Sss_I: = \prod_{J: I\subsetneq J} (1-\chi_J) \Sss_I' + \sum_{J: I\subsetneq J} \chi_J  \al_{IJ}^*(\Sss_J): V_{I}^{|I|}\to E_I .
\end{align} 
% 
% \phi_{IJ}^*(\Sss_J):  \bigcup_{J: I\subsetneq J} V_{IJ}^{|I|+1,J}\to E_A
 is well defined and is compatible under pullback from  $V_J$.  

 Let us now suppose that maps $\al_{IJ}: V_{IJ}^{|I|+1,|J|+1}\to  V_{J}^{|J|+1}$, and functions $\Sss_I: V_I^{|I|}\to E_A$ have been defined for all $I\subsetneq J$ with $|I|>k$ so as to satisfy conditions~\eqref{eq:b},\eqref{eq:c}, and consider $I$ with $|I|=k$.  Because there are no transition functions $\al_{II'}$ between these sets $V_I$ we can work separately with each such $I$.  Then define $\al_{IJ}(x)$ for $x\in V_{IJ}^{|I|+1,|J|+1}$ by applying the collar maps $c^Y_{HJ}$  as described in Step 2 over the part, called $\Pp_{IJ}(x)$ below,  of the path 
$ \Pp(e,x)$ from $p_k = b_{I}$ (where $I = I_k(|x|)$) to $p_q$, where $J = I_q(|x|)$.

We check the properties of $\al_{IJ}$ as follows.

\begin{itemlist}\item  The map $\al_{IJ}$ depends continuously on $x$ because we saw above that the path $\Pp(e,x)$ depends continuously on $x$, and because by  \eqref{eq:coll2} the collar map along a path segment of length $0$ is the identity.
\item Both $\TM_{IJ}$ and $\al_{IJ}$ have the product form required by (a) because the collar map $c^Y_J$ does not change the components of $e_{A\less I}$ that lie in  $E_{A\less J}$; cf. \eqref{eq:coll2}.
\item  We repeatedly use the fact that the collar is compatible with all the shrinkings to show  that (b) holds.
\item  To prove the composition formula (c), we use the fact proved above that 
when  $x\in M_{IH}^{|I|+1,|H|}$, the path $\Pp_{IJ}(x) $ divides into two independent segments, the first of which is simply $\Pp_{IH}(x)$, while the second projects onto  $\Pp_{HJ}(\phi_{IJ}(x))$.  Now use the invariance of the collar map under rescaling \eqref{eq:rescal}.
\item To see that $\al_{IJ}$ is injective,
notice first that the path $\Pp_{IJ}(x)$ is determined by $x$.  Hence the collar maps
applied to the lift $(e',\phi_{IJ}(x); b_I)$ of $(e,x)\in M_{II}$ to $Y = Y_{\Vv,J,\ve}$ give a point in $Y$ that lies over a point $t_x\in \De_J$, that is determined by $\Pp(e,x)$ because the collar $c^Y_J$ lifts $c^\De_J$ by \eqref{eq:coll1}.  But 
 the collar maps are injective, as is the 
projection $Y_{\Vv,J,\ve}\cap \pr_\De^{-1}(t_x)$ to $M_J$.
\end{itemlist}
 
 Finally, we define
 $\Sss_I$ as in \eqref{eq:SssH}.  This clearly has the properties required in (iii).
\MS

\NI {\bf Step 4:}\ {\it Completion of the proof in the case of trivial isotropy.}
 The first claims  in (i), namely that   $\tau_{IJ}^0$ extends $\rho_{IJ}$ and that $\TM_{IJ}$ is a product of the form $E_{A\less J, \de_I}\times \TM_{IJ}^0$, are clear.  
 To establish the separation claim, namely that $cl(\TM^0_{IJ})\cap cl(\TM^0_{HJ}) = \emptyset$ unless $I,H$ are nested, notice that   the intersections 
 of $cl(\TM^0_{IJ}), cl(\TM^0_{HJ})$ with $\{0\}\times V_J$ certainly have this property by definition of a reduction.  Hence, starting with maximal $|I|$ as usual, we may, if necessary, shrink the constants $\de_I$  so that this property holds.
 Further, (iv) holds, because if $\bK$ is oriented, then so are all the manifolds $Y_{\Uu, J, \,\ve}$ and $M_I$.  Since the 
 structural maps in $\Kk$ preserve orientation by definition, and the
 collar maps $c^Y_J$   preserve orientation by Proposition~\ref{prop:col}, so do the maps $\tau_{IJ}$ constructed above. 

It remains to arrange that the maps $\al_{IJ}$ have closed graph.  We do this by the method described in Example~\ref{ex:MJ}.
Given $I\subsetneq J$, recall that $M_{IJ}: = M_{IJ}^{|I|,|J|}\subset M_{IJ}^{|I|,|J|+1}$  and define 
%Thus for $I\subsetneq J$, we define
  $\Fr(M_{IJ}): = \cl(M_{IJ})\less M_{IJ}$ where we take the closure in $M_{IJ}^{|I|,|J|+1}$.  Then  the maps $\al_{IJ}$ extend to give a compatible family of embeddings defined over  $\cl(M_{IJ})$.  The images $\al_{IJ}(\cl(M_{IJ})), \al_{HJ}(\cl(M_{HJ}))$ are disjoint unless $I,H$ are nested.  Moreover,
   if $H, I$ are nested, the intersection $\al_{IJ}(M_{IJ})\cap  \al_{HJ}(\Fr(M_{HJ}))$ is empty because $\al_{IJ}(M_{IJ})\subset M_J = M_J^{|J|}$ while 
 $   \al_{HJ}(\Fr(M_{HJ}))\subset \Fr(M_J) \subset M_J^{|J|+1}\less M_J^{|J|}$.  
 Hence, if  we define 
\begin{align}\label{eq:MKprune}
 M_J': = M_J\less {\textstyle \bigcup_{I\subsetneq J} } cl(\al_{IJ}(\Fr(M_{IJ})),
\end{align}
 the maps $\al_{IJ}$  for $I\subsetneq J$ have image in $M_J'$ and closed graph.    Moreover, if we have already arranged that
 all the maps $\al_{JK}: M_{JK}\to M_K$ for $J\subsetneq K$ have closed graph and satisfy the compatibility conditions \eqref{eq:c} for all $I\subset J\subset K$,
then if we replace the domain $M_{JK}$ by $M_J'\cap M_{JK}$ the maps  $\al_{JK}: M_{JK}'\to M_K$ will still have these properties.  
 Hence we may arrange that all the maps $\al_{IJ}$ have closed graph by applying these two steps for each $J$, starting with $J$ such that $|J|$ is maximal and then working down.
 
This completes the proof if the isotropy groups are trivial.
\MS

\NI {\bf Step 5:}\ {\it The case of nontrivial isotropy.}

To construct the maps $\tau_{IJ}$ in general, we argue as above, taking $\phi_{IJ}(x)$ to be the local inverse to the covering map $\rho_{IJ}$ at $x\in \TV_{IJ}$, and then defining $\tau_{IJ}$ to be the $\Ga_A$-equivariant extension of  $\al_{IJ}^{-1}$ to a neighborhood of the orbit of $(0,x)$ in $\TM_{IJ}$.  
To see that this definition is consistent and independent of the choice of $x\in \rho_{IJ}^{-1}(\rho_{IJ}(x_0)$, note that
 the collar map is equivariant and, once the shrinkings $(\Vv^k, \ve^k)$ are chosen,  the only other choice in 
the above construction is that of the cutoff functions $\chi_I$ in \eqref{eq:chi} whose pullbacks to the sets $V_I$ are also equivariant.
Hence the local inverse  $\phi_{IJ}(x_0)$  is invariant under the stabilizer group $\Ga_x$, and so the extension is well defined.

\MS

Since all the previous arguments apply without essential change, it
  remains to check that  $\tau_{IJ}^0$ quotients out by a free action of $\Ga_{J\less I}$
on $\TM_{IJ}^0$ that extends to a neighborhood of $\cl(\TM_{IJ}^0)$.  
 To establish this, we must define an appropriate action of $\Ga_{J\less I}$ on $\TM_{IJ}^0$.  
If $\Ga_{J\less I}$ acts trivially on $E_{J\less I}$, then this action is simply the restriction of the given action of $\Ga_{J\less I}$ on  $V_J$.  However, in general  this is not the case, and the new action $$
\Ga_{J\less I}\times \TM_{IJ}^0\to \TM_{IJ}^0,\quad (\ga,x)\mapsto \ga*x  
$$
is described as follows.  
Notice first that because the collar $c_J^Y$ is $\Ga_A$-equivariant and injective,  each point 
 $x_0\in \TV_{IJ}\subset \TM_{IJ}^0$ %\subset (E_{A\less J}\times  V_J)$ 
 with $\tau_{IJ}^0(x_0) = 
(0, x_0')\in E_{J\less I, \de_I}\times V_{IJ}$
has a neighborhood $\Nn(x_0)$ on which $\tau_{IJ}^0$ is injective and has  image $\Nn'$ of product form, namely $\Nn' = 
 E_{J\less I, \de_I}\times \Oo'\subset  E_{J\less I, \de_I}\times V_{IJ}$.  Further,  $\Ga_{J\less I}$ acts on $\Nn'$ via its action on 
 $ E_{J\less I, \de_I}$, since it fixes the points of $V_{IJ}\subset V_I$.  
If $\tau_{IJ,x_0}^{-1}:\Nn'\to \Nn(x_0)$ is the local inverse to $\tau_{IJ}$ at $x_0$, we now define
\begin{align}\label{eq:ga*}
\ga * x: = \ga\cdot_J \tau_{IJ,x_0}^{-1}(\ga^{-1}\cdot_I  \tau_{IJ}(x)), \quad x\in \Nn(x_0).
\end{align} 
where for clarity we have written $x\mapsto \ga\cdot_I x$ (resp. $x\mapsto \ga\cdot_J x$) for the standard action of $\ga\in \Ga_{J\less I}$ on $E_{J\less I, \de_I}\times V_{IJ}\subset  E_{J\less I, \de_I}\times V_{I}$\; (resp. on $\TM_{IJ}^0\subset V_J$).  
Then
\begin{align*}
\tau_{IJ}(\ga * x)& = \tau_{IJ}\bigl( \ga\cdot_J \tau_{IJ,x_0}^{-1}(\ga^{-1}\cdot_I  \tau_{IJ}(x))\bigr)\\
&= \ga\cdot_I\bigl(\tau_{IJ}\circ \tau_{IJ,x_0}^{-1}(\ga^{-1}\cdot_I  \tau_{IJ}(x)) \; = \; \tau_{IJ}(x),
\end{align*}
where the second equality uses the equivariance of $\tau_{IJ}$ with respect to the actions $\cdot_J, \cdot_I$.  
Now extend this action over the whole orbit by setting $\de*(\ga*x): = (\de\ga) * x$.
This new action $x\mapsto \ga*x$  is free, since $\Ga_{J\less I}$ acts freely on $\TV_{IJ}$.  Further,  
this action extends to a free action on a neighborhood of the closure of $\TM_{IJ}^0$ since it is determined by $\tau_{IJ}$, and hence by the collar, both of which can be extended. 
\end{proof}

\begin{lemma}\label{le:compat}  The  action $x\mapsto \ga*x$ of $\Ga_{J\less I}$ on %(a small neighborhood of)
 $\TM_{IJ}$ has the following properties:
\begin{enumerate}
 \item if  $H,I\subset J$, then the action of $\Ga_{J\less I}$  on $\TM_{IJ}$ preserves the subset $\TM_{IJ}\cap \TM_{HJ}$;
 \item if  $H\subset I\subset J$ then the restriction to  $\TM_{IJ}\cap \TM_{HJ}$ of the action of $\Ga_{J\less I}$  on $\TM_{IJ}$ agrees 
 with that obtained by considering $\Ga_{J\less I}$  as a subgroup of  $\Ga_{J\less H}$ and restricting the corresponding action from  $\TM_{HJ}$ to $\TM_{IJ}\cap \TM_{HJ}$.  
 \item if $H\subset I\subset J$ and $y\in \TM_{IJ}\cap \TM_{HJ}$,   then $$
 \tau_{IJ}(\ga_{J\less H} *y) = (\ga_{J\less H}|_{I\less H})* \tau_{IJ}(y)
 $$
 where $\ga_{J\less H}|_{I\less H}$ is the image of $\ga_{J\less H}$ under the projection $\Ga_{J\less H}\to \Ga_{I\less H}$.
 \item Properties (i) and (ii) continue to hold for the extension of the action to the closure $\cl(\TM_{IJ})$ of
 $\TM_{IJ}$ in $M_J$.
\end{enumerate}
\end{lemma}
\begin{proof}   (i) follows from \eqref{eq:ga*} because the action $x\mapsto \ga\cdot_J x$ preserves the  sets $\TM_{HJ}$ for all $H\subset J$.
(ii) also follows immediately from \eqref{eq:ga*} and the fact that $\tau_{HJ} = \tau_{HI}\circ \tau_{IJ}$ on $\TM_{IJ}\cap \TM_{HJ}$.  (iii) holds because the maps $\tau_{IJ}$ are equivariant with respect to the projection $\Ga_J\to \Ga_I$ and take $\TM_{IJ}\cap \TM_{HJ}$ to $M_{IJ}\cap \TM_{HI}$ by \eqref{eq:tauIJK}.  Finally (iv) holds because the extended action is  defined by extending the domain of the maps in \eqref{eq:ga*}.
\end{proof}

\begin{rmk}\rm \label{rmk:sm2}{\bf (The smooth case.)}  
Note first that if we apply the above construction to a smooth atlas (i.e. one that satisfies the smooth submersions condition in \eqref{eq:prod2}), then the 
charts used in \eqref{eq:phiYI} to give $Y_{\Vv,J,\,\ve}$ the structure of a  topological  manifold do not have differentiable inverses.  
A related problem may  also be seen in Example!\ref{ex:MJ}: 
%not smooth because its defining equation $s_J(x) = t\cdot e$ is not smooth at points $t\in \p \De$.  
%Further
 the attaching map $\al_{1,12}$ in \eqref{eq:al112} is given by the collar, which by \eqref{eq:coll00} has the form $(e_2,x)\mapsto x': = \phi(\|e_2\|^{1/2} e_2, x)$,
where $\phi$ is the local product structure along $\TV_{IJ}$  in \eqref{eq:subm0}.  Thus, even if $\phi$ were
 a diffeomorphism, $\al_{1,12}$ would not have a smooth inverse along the submanifold  $e_2 = 0$.  
Thus, just as in standard blow-up constructions, in order to obtain a smooth category $\bM$ from a smooth atlas one needs to choose a smoothing of $Y_{\Vv,J,\,\ve}$ along its boundary.

Alternatively, one could use a different construction that avoids 
%The construction in the smooth case is similar but somewhat
%easier since there is no need to 
introducing the manifold $Y$.  Instead, one can construct the all important collar structure used to define the maps $\tau_{IJ}$  
 by using the exponential map with respect to a suitable family of metrics on the sets $V_J$.  Indeed, recall that 
by the smooth tangent bundle condition \eqref{eq:prod2} the derivative  $\rd s_{J\less I}$ induces an isomorphism from
the normal bundle $T^\perp(\TV_{IJ})$ of $\TV_{IJ}$ in $V_J$ to the product $E_{J\less I, \eps_I}\times \TV_{IJ}$.  
%For example, in the situation considered in (ii) above, 
To explain the idea, let us suppose for simplicity that that the cover $\Vv$ is refined so that the group $\Ga_{J\less I}$ acts freely on the components on $\TV_{IJ}$, so that the restriction of $\tau_{IJ}$ to each component is  a diffeomorphism onto $V_{IJ}$ .
Then we can think of $V_{IJ}$ as a subset of $\TV_{IJ}$ and the task is to define 
a consistent family of injections $\al_{IJ}:  E_{J\less I, \eps_I}\times V_{IJ}\to V_J$.
To this end, choose a family of $\Ga_I$-invariant Riemannian metrics  $g_I$ on $V_I$ and constants $\eps_I$ that are compatible in the following sense:
\begin{itemlist} \item for each $I\subsetneq J$, $\TV_{IJ}$ is a totally geodesic submanifold of $(V_J, g_J)$ and 
$$
(\rho_{IJ})_*(g_J|_{\TV_{IJ}}) = g_I|_{V_{IJ}}.
$$
%$
%g_J|_{\TV_{IJ}} = (\phi_{IJ})_* (g_I|_{V_{IJ}})$;
\item $0<\eps_I<\eps_J$ if $I\subsetneq J$;
\item  for each $I\subsetneq J$, the $g_J$-exponential map  along directions perpendicular to
$\TV_{IJ}$  defines an embedding $\al_{IJ}: E_{J\less I, \eps_I}\times V_{IJ}\to V_J$;
\item the corners are locally flat, i.e. if $x\in \TV_{IJ}\cap \TV_{HJ}$ for $I\subsetneq H\subsetneq J$ then
$$
\al_{IJ}\bigl(e_{J\less H}+ e_{H\less I}, \ x\bigr) = \al_{HJ}\bigl(e_{J\less H}, \ \al_{IH}(e_{H\less I}, \ x)\bigr).
$$
\end{itemlist}
The last condition means that the composition rule holds directly, without having to introduce analogs of the paths $\Pp(e,x)$.
Of course, the choice of the $g_I,\eps_I$ requires some attention to detail as in
the proof of Lemma~\ref{le:Veps} 
 below;  see also  the construction of the perturbation section in \cite[\S7.3]{MW2}.  Thus one begins with a family of shrinkings $\Vv^\ka\sqsubset \cdots \sqsubset \Vv^1 \sqsubset \Vv^0$ of an initial reduction $\Vv^0$, where $\ka : = \max \{|J| \ | \ J\in \Ii_\Kk\}$ and then chooses metrics $g_J$ on $V_J^{|J|}$,
starting with $J$ of length $|J| = 1$, that satisfy the above conditions for the submanifolds
$\TV_{IJ}^{|J|}$ of $V_J^{|J|}$   for some constant $\eps_I'>0$.  Finally, once $g_J$ is defined on $V_J^\ka$ for all $J$, one chooses suitable constants $\eps_J$, now starting with maximal $|J|$ and working down.    Further details are left to the reader, as is the proof that   the resulting branched manifold is cobordant to the topological one constructed in detail above.  For this last step one would need to adapt the proof of uniqueness in Step 2 of the proof of Theorem~B in \S\ref{ss:topol}.

Besides  obtaining  a smooth rather than topological branched manifold, there are no real advantages to this construction unless one wants to work with the virtual fundamental class on the chain level using de Rham cochains.   Another point is that
 by \cite{Mbr} we can construct  the branched manifold $M$ to be a simplicial complex, so that we could simplify the proof of Lemma~\ref{le:fundM} by using (locally finite) singular homology instead of \c{C}ech homology. However, because we know nothing about $X$ except that it is compact and Hausdorff, the VFC has still to be considered as an element in \c{C}ech homology.  For further discussions of the smooth case, see Step 3 of the proof of Theorem~B.
 \hfill$\er$
\end{rmk}

\subsection{Proof of Theorems~A and~B}\label{ss:topol}

To prove Theorem~A, we must show that  the category $\bM$ constructed in \S\ref{ss:M}   
has a unique completion to a weighted branched groupoid, and then analyze the structure of this groupoid and the associated 
weighted branched manifold $(M,\La)$.  
The arguments needed here are very similar, but not identical, to those in  \cite[Proposition~2.3] {Morb} (which considers 
the case of the category $\bB_\Kk$ defined by an atlas with trivial obstruction spaces)  
and in \cite[\S3.3]{MWiso} (which analyzes the zero set of a transverse perturbation section).
Theorem~B has two parts. It first states  that if $\Kk$ is oriented the weighted branched manifold $(M,\La)$ carries a natural fundamental class $[X]^{vir}_\Kk$, a result that
 was proven in \cite{Mbr}  in the case when $\bM$ is smooth and  compact, with or without boundary.  
Although smoothness is assumed throughout \cite{Mbr}, the only place where this condition is essential is in the construction of the fundamental class in the proof of \cite[Proposition~3.25]{Mbr}. In this case, we may  replace $\bM$ by an equivalent wnb groupoid that is  tame in the sense that its branching loci are piecewise smooth and hence triangulable, which allows us to work with singular homology in the proof of Lemma~\ref{le:fundM}.
  In the present case,
we must use rational \c{C}ech cohomology, and the appropriate dual homology theory for noncompact manifolds as described in \S\ref{ss:app}.
The second and  more substantial part of  the proof of Theorem~B explains why  $[X]^{vir}_\Kk$ is independent of all choices made in its construction, and why, in the smooth case,  the new definition is consistent with the previous definition via perturbation sections.  
\MS

We begin with a lemma about groupoid completions of \'etale categories; for definitions, see \S\ref{ss:defs}, \S\ref{ss:br}.   As usual we denote by $\Ii$ a collection of subsets of a finite set $A$, and say that $I,H\in \Ii$ are nested if $I\subset H$ or $H\subset I$.  
We state the condition {\it identity} below for completeness: it follows immediately from the fact that every category has identity morphisms.

\begin{lemma}\label{le:complet}  Let $\Ii$ be a collection of subsets of a finite set $A$, and  $\bM$ be an \'etale category with
\begin{align*}%\label{eq:Mm1}
\Obj_{\bM} &= {\textstyle  \bigsqcup_{I\in \Ii} } M_J, \qquad 
\Mor_{\bM} ={\textstyle \bigsqcup_{I\subset J,\, I,J \in \Ii} \TM_{IJ} } \\ \notag
s\times t: &\;\;\Mor_{\bM}\to \Obj_{\bM}\times \Obj_{\bM},\quad (I,J,y)\mapsto \Bigl(\bigl(I,\tau_{IJ}(y)\bigr), \bigl(J,y\bigr)\Bigr),
\end{align*}
where  $\TM_{IJ}\subset M_J$ is an open 
subset  %whose closure $cl(\TM_{IJ})$ is disjoint from $cl(\TM_{HJ})$ unless $I,H$ are nested,
and  the maps  
$\tau_{IJ}: \TM_{IJ} \to M_{IJ}$ satisfy the following conditions:
\begin{itemlist} \item {\bf (identity)} for all $I\in \Ii$,  $\tau_{II} = \id$ on  $\TM_{II} = M_{II} = M_I$;
\item {\bf (composition)} for all $H\subset I\subset J$,   $\tau_{IJ}^{-1}(\TM_{HI}\cap M_{IJ}) = \TM_{HJ}\cap \TM_{IJ}$ and 
$\tau_{HJ} = \tau_{HI}\circ \tau_{IJ}$  on $ \TM_{HJ}\cap \TM_{IJ}$; hence, if  $z\in \TM_{IJ}\cap \TM_{HJ}$ where $H\subset I\subset J$
we have $(H,I,y)\circ (I,J,z) = (H,J,z).$
\item {\bf (separation)}  $\cl(\TM_{IJ})\cap \cl(\TM_{HJ} ) = \emptyset$ unless $I,H$ are nested;
\item {\bf (group actions)}  for each $i\in A$ there is a finite group $\Ga_i$ such that,
for all $I\subset J$, $\tau_{IJ}$ quotients out by the restriction to $\TM_{IJ}$ of a free action of $\Ga_{J\less I}$ on $cl(\TM_{IJ})\subset M_J$, where 
$\Ga_{J\less I}: =  \prod_{i\in J\less I} \Ga_i$.  Moreover, these actions $x\mapsto \ga*x$ satisfy the compatibility conditions listed in Lemma~\ref{le:compat}.
\end{itemlist}
Then, there is a unique nonsingular groupoid $\Hat\bM$ with the same object space and realization as $\bM$.
 Its morphism spaces for $I\subset J$ are 
\begin{align}\label{eq:MorHM}
\Mor_{\Hat\bM}(M_I,M_J) &:=
 {\textstyle \bigcup_{\emptyset\ne F\subset I} } \; \bigl(\TM_{IJ}\cap \TM_{FJ}\bigr)\times \Ga_{I\less F},\\ \notag
 &= \bigl\{ (y,\ga)\in M_J\times \Ga_J\ \big| \ y\in \TM_{IJ}, \ga\in \Ga_{I\less H_{y}}\bigr\},
 %\quad\mbox{where} \;\; F_x: = \min \{F: x\in \TM_{FJ}\}\\
\end{align}
where $H_y: = \min \{H: y\in \TM_{HJ}\}$, with 
\begin{align}\label{eq:MorHM1}
 s\times t\,(I,J,y,\ga_{I\less H_y})& = \bigl(\,(I,\, \ga_{I\less H_y}^{-1} * \tau_{IJ}(y)), \; (J,y)\bigr)\\ \notag
 (I,I,y,\ga_{I\less H_y})^{-1} &= (I,I,\,\ga_{I\less H_y}^{-1}*y,\,\ga_{I\less H_y}^{-1})
\end{align}
In particular, $\Hat\bM$ is \'etale, and  there is an injective functor $\bM\to \Hat\bM$.
\end{lemma}
\begin{proof}  
%Since the  situation considered here  is very similar (though  not identical)  to that in \cite[Lemmas~3.2.9]{MWiso}, we will  only sketch the main points.  
Observe first that because  a nonsingular category has at most one morphism between any two objects,  $\Hat\bM$ must have precisely one morphism between any two objects $(I,x), (J,y)$ that are equivalent under the equivalence relation $\sim_\bM$ on $\Obj_{\bM}$ generated by $\Mor_{\bM}$.  Hence, 
there is precisely one nonsingular groupoid with $\Obj_{\Hat\bM} = \Obj_{\bM}$ and $|\Hat\bM| = |\bM|$.   Since there is an injection
$$
\Mor_\bM(M_I,M_J) = \TM_{IJ}\; \hookrightarrow\;
 {\textstyle \bigcup_{\emptyset\ne F\subset I} } \; \bigl(\TM_{IJ}\cap \TM_{FJ}\bigr)\times \Ga_{I\less F}
 $$
  when $I\subset J$, and the structural maps described in \eqref{eq:MorHM1} are \'etale,  it remains to check that the
 formula \ref{eq:MorHM} does describe $\Mor_{\Hat\bM}(M_I,M_J)$.

The separation property implies that for each $y\in M_J$
 the set of $F$ such that $|y|\in |\TM_{FJ}|\subset |\bM|$ is nested. Let $H_y$ be the minimal such element.  By Lemma~\ref{le:compat},  for all $H_y\subset I\subset J$ the group $\Ga_{I\less H_y}$ acts freely on $\TM_{H_y I}$.  Hence each element in 
 $\bigcup_{\emptyset\ne F\subset I}  \; \bigl(\TM_{IJ}\cap \TM_{FJ}\bigr)\times \Ga_{I\less F}$
has a unique description of the form $(y,\ga)$ where $y\in \TM_{IJ}, \ga\in \Ga_{I\less H_y}$.  Further, given such $(y,\ga)$ 
it follows from Lemma~\ref{le:compat}~(iii) that $\tau_{H_y J}(y) = \tau_{H_y I} (\ga^{-1}* \tau_{IJ}(y))$ for any $I$ with $y\in \TM_{IJ}$, so that
$$
\bigl(J,y\bigr)\sim_\bM \bigl(H_y, \tau_{H_y J}(y)\bigr) = \bigl(H_y, \tau_{H_y J}(\ga^{-1}* \tau_{IJ}(y))\bigr) \sim_\bM \bigl(I,\tau_{I J}(y)\bigr).
$$
Hence for each $I$ with  $y\in \TM_{IJ}$ and each $\ga\in \Ga_{I\less H_y}$ there must be a morphism $m$ in $\Hat\bM$  from $(I,\ga^{-1}* \tau_{IJ}(y))$ to $(J,y)$.
 To see that these are the only morphisms in $\Hat\bM$ it remains to observe that each equivalence class $[(J,y)]$ contains a unique element of the form $(H_y,z)$ where $z\in M_{H_y}$; further if $H_y\subset I\subset J$ then $[(J,y)]\cap M_I = \{x\in \rho_{H_yI}^{-1}(z)\} $ consists of the $\Ga_{I\less H_y}$-orbit of $\tau_{IJ}(y)$. Since each morphism is uniquely specified by its source and target, there is no need to write out the composition rule in $\Hat\bM$ explicitly.\end{proof}

\begin{lemma}\label{le:Hcomplet}   Suppose in the situation of Lemma~\ref{le:complet} that for each $I\subset J$ the map $\tau_{IJ}$ has closed graph.
Then \begin{itemlist} \item[{\rm (i)}] the  maximal Hausdorff quotient $|\bM|_\Hh$ %is given by closing the equivalence relation on $\Obj_{\bM}$, and hence is
is  the realization of a nonsingular groupoid $\Hat\bM_\Hh$ with objects $\Obj_{\bM}$ and morphisms from $M_I$ to $M_J$ with $I\subset J$ given by
$$
\Mor_{\Hat\bM_\Hh}(M_I,M_J) :=
{\textstyle \bigcup_{\emptyset\ne F\subset I} } \; \bigl(\TM_{IJ}\cap cl(\TM_{FJ})\bigr)\times \Ga_{I\less F}.
$$
\item[{\rm (ii)}] For each $I$, the map $\pi_I^\Hh: M_I\to |\Hat\bM_\Hh|$ is a local homeomorphism with open image, and in particular is a proper map onto its image.
 \item[{\rm (iii)}]  The space $M: = |\bM|_\Hh = |\Hat\bM_\Hh|$ can be given the structure of a  weighted nonsingular branched  manifold
with weighting function $\La_M:M \to \Q^+ = \Q\cap (0,\infty)$  given for $p\in |M_I|_\Hh$ by 
$$
\La_M(p) := \tfrac 1{|\Ga_I|}\,  \# \bigl\{ y\in M_I \,\big|\, \pi_{\bM}^\Hh(y) =p \bigr\} = \tfrac {|\Ga_{I\less F_y}|}{|\Ga_I|}
$$
where $F_y: = \min \{F: y\in \cl(\TM_{FI})\} = \min\{F:  \pi_{\bM}^\Hh(y)\in \cl(\pi_{\bM^\Hh}(M_F)\}$.
Moreover the wnb manifold $M$ is oriented if $\bM$ is.
\end{itemlist}
\end{lemma}
\begin{proof}
Denote by $\approx_{\bM}$  the equivalence relation on $\Obj_{\bM}$  corresponding to the 	quotient map $\Obj_{\bM} \to |\bM|_\Hh$.   Its graph is the closure in $\Obj_{\bM}\times \Obj_{\bM}$  of  the graph of $\sim_{\bM}$.  
First consider the component in $\Mor_{\Hat\bM}$ consisting of morphisms from $M_I$ to $M_J$ for $I\subset J$ with  $\ga = \id$.  This set can be identified with $\TM_{IJ}$ and has closed graph by hypothesis.  Next consider the set of morphisms in $\Hat \bM$ from $M_J$ to $M_J$:
\begin{align*}
\Mor_{\Hat\bM}(M_J,M_J)& :=
 {\textstyle \bigcup_{\emptyset\ne F\subset J} } \; \bigl(\TM_{FJ}\bigr)\times \Ga_{J\less F},\\
 s\times t\,(J,J,y,\ga)& = \bigl(\,(J, \ga^{-1}*  y), \; (J,y)\bigr).
\end{align*}
These morphisms form a group  with closure
\begin{align*}
\Mor_{\Hat\bM_\Hh}(M_J,M_J) &:=
 {\textstyle \bigcup_{\emptyset\ne F\subset J} } \; \bigl(M_J\cap cl(\TM_{FJ})\bigr)\times \Ga_{J\less F},\\
 & = \bigl\{(y, \ga)\ \big| \ y\in M_J, \ \ga \in \Ga_{J\less F_y}\bigr\}\quad\mbox{ where }\\
s\times t\,(J,J,y,\ga)& = \bigl(\,(J, \ga^{-1}*  y), \; (J,y)\bigr).
\end{align*}
Next observe that 
 every morphism $(I,J, y,\ga)\in \Mor_{\Hat\bM}(M_I,M_J)$, where $I\subset J$, may be written as the composite 
$$
(J,J,y,\ga)\circ (I,J, \ga^{-1}*y, \id)
$$
 of a morphism of the second type followed by one of the first type. 
 Therefore, because 
 the action of $\Ga_{I\less F}$ on $\TM_{FI}\cap \TM_{IJ}\subset M_J$ (where $F\subset I\subset J$ ) extends to an action
on $cl(\TM_{FJ})\cap\TM_{IJ}$ by Lemma~\ref{le:compat}~(iii), 
the limit of a convergent sequence of morphisms also is such a composite.
Claim (i) then follows easily. 

Since $|\Hat\bM_\Hh|$ has the quotient topology, to establish (ii) it suffices to show that  
$(\pi_J^\Hh)^{-1} (\pi_I^\Hh(M_I))$ is open in $M_J$ for all $J$.   
This set  is empty unless $I\subset J$ or $J\subset I$.  In the former case, %$U\cap M_{IJ}$ is open in $M_I$ and 
$(\pi_J^\Hh)^{-1} \pi_I^\Hh(M_I) = \tau_{IJ}^{-1}(M_{IJ}) = \TM_{IJ}$, which is open. 
In the latter case, %$U\cap \TM_{JI}$ is open in $M_I$ and 
$(\pi_J^\Hh)^{-1} \pi_I^\Hh(M_I) = \tau_{IJ}(\TM_{IJ})$, which is also open.    This proves (ii).

To prove (iii), note that by (ii) we may define the local branches at $p=\pi_I^\Hh(y) \in  |M_I|_\Hh$ to be the image under $\Ga_{I\less F_p}$ of an open neighborhood $U\subset M_I$ of $y$ that is  disjoint from  $\cl(\TM_{FI})$ unless $F_y\subset F$ and is also 
disjoint from its images under $\Ga_{I\less F_y}$. 
Each such local branch is given weight $\frac 1{|\Ga_I|}$.  
It  then follows easily from Lemma~\ref{le:compat}  that $\La_M$ is well defined and has the required properties. For more details, see \cite[Lemma~3.2.10]{MWiso}.  Finally, the statement about orientations is clear.
\end{proof}

\begin{rmk}\label{rmk:Hcomplet}\rm  As in \cite[Lemma~3.2.10]{MWiso}, it follows from part (ii) of Lemma \ref{le:Hcomplet} that the topology on $|\Hat\bM_\Hh|$ is second countable, locally compact,
and metrizable.
\end{rmk}  
\MS

With these preliminaries in hand, it is easy to show that in the oriented case $M = |\Hat\bM_\Hh|$ has a fundamental class.

\begin{lemma}\label{le:fundM}  Let $\bM$ be oriented with corresponding oriented wnb groupoid $\Hat\bM_\Hh$ constructed as in Lemma~\ref{le:Hcomplet}, and let $M: = |\Hat\bM_\Hh|$.  Then
there is a class $\mu_M\in  \check{H}^\infty_N(M)$ with the following property:  if $U : = \pi_I^\Hh(M_I)$ for some $I\in \Ii_\Kk$, then
\begin{align}\label{eq:fundc}
\rho_{M,U}(\mu_M) = \tfrac 1{|\Ga_I|}\ (\pi_I^\Hh)_* (\mu_I) \; \in \;  \check{H}^\infty_N(U),
\end{align}
where $\mu_I\in  \check{H}^\infty_N(M_I)$ is  the fundamental class  in \eqref{eq:ffcl} and $\rho_{M,U}$ is the restriction map on homology in \eqref{eq:rhoYU}.
\end{lemma}
\begin{proof}
It follows from Lemma~\ref{le:Hcomplet} that the statement of the lemma makes sense: the class $\mu_I$ exists by property (a$'$) in \S\ref{ss:app}, 
the restriction exists by (b$'$) because $U$ is open, and   the pushforward exists by (c$'$) because the map $\pi_I^\Hh:M_I\to |M|$ is proper.  
We prove the lemma by showing that for $k = 1,2,\dots,$ there is a class $\mu_k$ on $W_k: = 
\bigcup_{I: |I|\le k}\pi_I^\Hh(M_I)$ such that
$$
\qquad \qquad 
\mu_k|_{\pi_I^\Hh(M_I)} \; = \; (\pi_I^\Hh)_*(\tfrac 1{|\Ga_I|} \mu_I),  \qquad \quad \forall I, |I|\le k.
$$
When $k=1$, $W_1$ is a disjoint union of sets $\pi_I^\Hh(M_I)$, where $|I|=1$, and we simply define  $\mu_1$ to be the given pushforward.    Let us suppose that $\mu_k$ is constructed, and consider the definition of $\mu_{k+1}$.  Since the sets 
$\bigl(\pi_J^\Hh(M_J)\bigr)_{|J|=k+1}$  are disjoint, it follows from (e$'$) that we can consider each of them separately.
Further, by applying  Mayer--Vietoris  with $U = W_k, V=\pi_I^\Hh(M_J)$ it suffices to show that
the classes $\mu_k\in \check{H}^\infty_N(W_k)$ and $(\pi_J^\Hh)_*(\frac 1{|\Ga_J|} \mu_J)\in \check{H}^\infty_N(V)$ have the same restriction to $W_k\cap V = \cup_{I\subsetneq J}\pi_I^\Hh(M_{IJ})$.  
But because  restriction commutes with pushforward by (d$'$), it suffices to prove the corresponding statement for the fundamental classes of the spaces $M_J$.   Namely, we must check that 
$$
\tfrac 1{|\Ga_J|}(\tau_{IJ})_*(\mu_J|_{\TM_{IJ}}) = \tfrac 1{|\Ga_I|}\ (\mu_I|_{M_{IJ}}).
$$
But on manifolds the homology theory $\check{H}^\infty$ agrees with the usual locally compact singular homology.  Hence the above property holds because, by hypothesis,  
the maps $\tau_{IJ}: \TM_{IJ}\to M_{IJ}$ are orientation preserving covering maps of degree $|\Ga_J|/|\Ga_I|$.
\end{proof}  

We are now in a position to prove the main theorems.  We begin with the proof of Theorem~\ref{thm:M}, which as already noted in  \S\ref{ss:br}
immediately implies Theorem~A.
\MS

\NI \begin{proof}[Proof of Theorem~\ref{thm:M}]\;  Given the oriented atlas $\Kk$ we construct the category $\bM$ as in Proposition~\ref{prop:M1}.  We saw in  Lemma~\ref{le:Hcomplet}~(i) that this category has a unique  Hausdorff groupoid completion $\Hat\bM_\Hh$, which proves (i).  Part (ii) follows immediately from Lemma~\ref{le:Hcomplet}~(iii).  Further, the action of the group $\Ga_A$  on $\bM$ and $E_A$ induces an action on $\Hat\bM_\Hh$ and $\bE_A$, and the functor $\Sss: \bM\to \bE_A$ extends to a $\Ga_A$-equivariant functor 
$\Hat\bM_\Hh\to \bE_A$.  Therefore it remains to check that the induced map $\Sss_M: M\to E_A$ on the realizations has compact zero set $\Sss_M^{-1}(0)$ and that there is a map $\psi:\Sss_M^{-1}(0)\to X$ that induces a homeomorphism $\Sss_M^{-1}(0)/\Ga_A\stackrel{\cong}{\to} X$.

Since $M_I\cap \Sss_I^{-1}(0) = \{0\}\times (V_I\cap s_I^{-1}(0))$ by \eqref{eq:sJ}, 
the full subcategory 
 of $\bM$ with objects $\bigsqcup_I M_I\cap \Sss_I^{-1}(0)$ includes into the full subcategory of $\bB_\Kk$ with objects 
  $\bigsqcup_I V_I\cap s_I^{-1}(0)$.    Hence there is an induced map on the realizations
  $$
  |\bM|\cap  \Sss^{-1}(0) \to |\Kk| \cap |\s|^{-1}(0)\cong X.
  $$  
  This is continuous and surjective, but not injective because we have not yet quotiented out by the group actions.  Nevertheless, because $X$ is Hausdorff, the universal property  satisfied by  the Hausdorff quotient implies that it factors through a map  $\psi: |\Hat{\bM}_\Hh|\cap  \Sss^{-1}(0)= \Sss_M^{-1}(0) \to X$.     
Further, because $\psi_I$ induces a homeomorphism $V_I\cap s_I^{-1}(0)/\Ga_I \to G_I\subset  X$ (where $G_I$ is the footprint of the reduced chart domain $V_I$) and $\Ga_{A\less I}$ acts trivially on $\Sss_I^{-1}(0)=V_I\cap s_I^{-1}(0)$, this map $\psi: \Sss_M^{-1}(0)\to X$ does factor through a bijective and continuous map $\Sss_M^{-1}(0)/\Ga_A\to X$.  To see that it is a homeomorphism, it suffices to check that
$\Sss_M^{-1}(0)$ is compact.

To this end, notice first that because
 the topology on $M$  is metrizable by Remark~\ref{rmk:Hcomplet}, we need only check that $\Sss_M^{-1}(0)$ is sequentially compact.  Thus, consider a sequence of points $p_k\in \Sss_M^{-1}(0)$.  Because $M$ is the union of the finite number of sets $|M_I|_\Hh$ we may suppose   that  that $p_k\in |M_I|_\Hh$ 
for all $k$. Choose a sequence $y_k\in M_I\cap \Sss_I^{-1}(0)$ such that $\pi_I^\Hh(y_k) = p_k$.  Then 
$y_k = (0,z_k)\in E_{A\less I,\eps_I}\times (V_I\cap s_I^{-1}(0))$, and $\psi(y_k) = 
\psi_I(z_k)\in X$.   By passing to a subsequence, we may suppose that the sequence $\psi_I(z_k)$ converges to $x_\infty\in X$.  Since the footprints 
$G_J: = \psi_J(s_J^{-1}(0))$ of the reduced charts form an open covering of  $X$,  we may further suppose that there is $J$ such that $\psi_I(z_k)\in G_J$ for all $k$ and that
this sequence has limit $x_\infty = \psi_J(z_\infty)\in G_J$.   
Because 
 $G_I\cap G_J \ne \emptyset$, the sets $I,J$ are nested, and the original sequence $p_k\in |M_I|_\Hh$ must lie in the intersection
 $p_k\in |M_I|_\Hh \cap |M_J|_\Hh$.
 Therefore the $p_k$ also have lifts $y_k'=(0,z_k')\in E_{A\less J,\eps_J}\times (V_J\cap s_J^{-1}(0)) $, and now it follows from the fact that the map
 $V_J\cap s_J^{-1}(0)\to G_J\subset X$ is finite to one that some subsequence of the $z_k'$ must converge to a some point  $z_\infty'$ in the finite set 
 $\bigl(V_J\cap s_J^{-1}(0)\bigr) \cap \psi_J^{-1}(x_\infty)$.    Hence $(p_k)$ has a subsequence that converges to $\pi^\Hh_J(0, z_\infty')\in |M_J|_\Hh\subset M$.  This completes the proof of Theorem~A.
\end{proof}
\MS

The proof of Theorem~B is somewhat longer, and hence we restate it for the convenience of the reader.  Here we assume that $\Kk$ and $X$ are as in Theorem~A. 
\MS

\NI {\bf Theorem B:} {\it If  $\Kk$ is oriented, there is a unique element $[X]^{vir}_\Kk\in \check{H}_d(X;\Q)$ that is defined as follows.  
 For $b \in \check{H}^d(X;\Q)$ and $D = d+\dim E_A$, we have
\begin{align*} 
&\langle [X]^{vir}_\Kk,\, b \rangle : = (\Sss_M)_*( \Hat{b})\in \check{H}^c_{\dim E_A}(E_A, E_A\less \{0\};\Q) \cong \Q,
\end{align*}
where $\Hat{b}$ is the image of $b$ under the composite 
\begin{align*} %\label{eq:capDD}
 \check{H}^d(X;\Q)\stackrel{\psi^*}\longrightarrow  \check{H}^d(\Sss_M^{-1}(0);\Q)\stackrel{\Dd}\longrightarrow \check{H}^c_{\dim E_A} (M, M\less \Sss_M^{-1}(0); \Q),
\end{align*}
and $\Dd$ is given by cap product with the fundamental class $\mu_M\in H^{d+\dim E_A}(M)$ constructed in Lemma~\ref{le:fundM}.  Moreover, $[X]^{vir}_\Kk$ depends only on the oriented  concordance class of $\Kk$, and in the smooth case agrees with the class defined in \cite{MWiso}.}
\MS

\begin{proof} {\bf Step 1:}  {\it Definition of $[X]^{vir}_\Kk$.}

Since the fundamental class $\mu_M$ exists by Lemma~\ref{le:fundM}, and an appropriate  cap product exists by point (f$'$) in the appendix, in order to see that $\langle  [X]^{vir}_\Kk,\, b \rangle$ is well defined it remains to note that the map
$$
(\Sss_M)_*: \check{H}^c_{\dim E_A} (M, M\less \Sss_M^{-1}(0); \Q) \to \check{H}^c_{\dim E_A}(E_A, E_A\less \{0\};\Q) \;\cong\; \Q
$$
is well defined. Further, it takes values in $\Q$, because $E_A$ is oriented by the definition in Remark~\ref{rmk:submer}~(iii) and the theory  $\check{H}_*^c$ coincides with singular homology theory on  simplicial spaces.
\hfill$\er$
\MS

\NI {\bf Step 2:}  {\it Proof of uniqueness.}

To prove the uniqueness of $ [X]^{vir}_\Kk$, one must state and prove the analog of Proposition~\ref{prop:M} for cobordism atlases, and also prove that all choices made in the construction  are unique modulo oriented cobordism. 
For the constructions that involve atlases, such results are proved in~\cite{MW1,MW2,MWiso}: see \cite[Prop.~4.2.3]{MW1} for 
different choices of metrics and  \cite[Thm.~4.2.7]{MW1} for different choices of tame shrinkings;  \cite[Theorem~5.1.6]{MW1} for a discussion of reductions, \cite[\S8]{MW2} for orientations (in particular \cite[Thm.~8.1.12]{MW2}) and
\cite[Appendix]{MWiso} for weighted branched cobordisms.   The present construction also requires a choice of local product structures (as in 
\eqref{eq:subm0}) and partition of unity (as in \eqref{eq:rhokk}) in order to define the collar of the manifolds $Y_{\Vv,J,\,\ve}$.
However, in distinction to the smooth case, it is not necessary to arrange that
cobordism atlases  have specified collars (i.e. local product structures) near the two boundary components because the VFC $[X]^{vir}_\Kk$ is now defined via diagram \eqref{eq:commcap} which involves restriction to the boundary rather than via a perturbation section that must be extended from the boundary to the interior.
  
Thus we define a  cobordism atlas $\Kk^{01}$ over $[0,1]\times X$ between two $d$-dimensional atlases $\Kk^0, \Kk^1$ on $X$ to be an atlas 
  $\Kk^{01}$ over $[0,1]\times X$  of dimension $d+1$ such that 
    \begin{itemize}\item[(i)]
  the charts whose footprints intersect $\p([0,1]\times X) = \sqcup_\al \al\times X$ are manifolds with boundary;
  \item [(ii)]
 for $\al = 0,1$ there are 
  functorial inclusions  
   $$
  \io_\al:  \Kk^\al \to \Kk^{01}, \quad    \io_\al^\Ii: \Ii_{\Kk^\al}\to \Ii_{\Kk^{01}}, \qquad \al = 0,1
   $$
that (for simplicity) we assume to have disjoint images, and  for each $I\in \Ii_{\Kk^\al}$ take the chart domain $U_I^\al$ onto the boundary $\p U^{01}_{I'}$ of the corresponding chart in ${\Kk^{01}}$, where $I': =\io_\al^\Ii(I)$, preserving orientation for $\al = 1$ and reversing it for $\al=0$;
\item [(iii)] we further require that the local product structures in  \eqref{eq:subm0} for the chart domains  in $\Kk^\al$ extend to local product structures near the 
boundary points of the corresponding chart domains in $\Kk^{01}$; 
\end{itemize}
We show in \cite[Thm.~7.1.5]{MW2} that any pair of reductions $\Vv^\al$ of $\Kk^\al$ may be extended to a reduction $\Vv^{01}$ of $\Kk^{01}$  
such that there are natural inclusions
$ \io_\al^V: |\Vv^\al|\to |\Vv^{01}|$ 
    that are homeomorphisms to their image.  Further, if
   $J\in  \Ii_{\Kk^{\al}}$  for $\al = 0,1$,
   then for suitable small $\ve^\al>0$   there is a commutative diagram
   \[
   \xymatrix
{
 E_{A^{01}\less A^\al,\,\ve^\al}\times  Y_{\Vv^\al,J, \,\ve}  \ar@{->}[d]_{ \pr_V}\ar@{->}[r]^{ \qquad \io_\al^Y} &  Y_{\Vv^{01},\io^\al(J), \,\ve} \ar@{->}[d]_{\pr_V}
\\
 V^\al_J \ar@{->}[r]^{\io_\al^V}& V^{01}_{\io^\al(J)}.
}
\]
Notice here that we take the product of $ Y_{\Vv^\al,J, \,\ve} $ with the extra obstruction spaces 
    $E_{A^{01}\less A^\al, \,\ve^\al}$ in order to increase its  dimension  to that of $  Y_{\Vv^{01},\io^\al(J), \,\ve}$.
  Because the maps \eqref{eq:subm0} in the submersion axiom for $\Vv^{01}$ extend those for $\Vv^{\al}$, we can choose the covering and partition of unity in   Step 2 of the proof of Lemma~\ref{le:col}  for $\Vv^{01}$ to extend those already chosen for $\Vv^\al$.  Therefore,
    we can construct the collars on $Y_{\Vv^{01},\io^\al(J), \,\ve}$  to extend already constructed collars on  the sets $Y_{\Vv^{\al}, J, \,\ve}$.
Hence, after possibly shrinking $ \ve>0$,
  we can arrange that for small $\ve^\al>0$ there are embeddings
\begin{align*}%\label{eq:EM}
\io_\al^M:  E_{A^{01}\less A^\al, \,\ve^\al}\times  M^\al  \to M^{01},\quad s.t. \;\;  \sqcup_\al \im(\io_\al^M) = \p M^{01};
\end{align*}
 and also that the map $
\Sss_M^{01}: M^{01}\to E_A$  satisfies 
\begin{align}\label{eq:EM1}
 \Sss_M^{01}\circ \io^M_\al = \Sss_M^\al\circ \pr_M^\al:\;\;  E_{A^{01}\less A^\al, \,\ve^\al} \times M^\al\to E_A,
\end{align}
where $\pr_M^\al:   E_{A^{01}\less A^\al, \,\ve^\al} \times M^\al\to M^\al$ is the projection.

Because $M^{01}$ is constructed from an atlas for the product 
$[0,1]\times X$,  the natural projection $(\Sss_M^{01})^{-1}(0) \to [0,1]\times X$ factors through a homeomorphism.
$$
(\Sss_M^{01})^{-1}(0)/\Ga^{01}\stackrel{\cong}\longrightarrow  [0,1]\times X 
$$ 
Notice here that  for $\al = 0,1$, the group $\Ga_{01}$ decomposes as a product that we will write $\Ga'_{01\less \al}\times \Ga_\al$, where 
$\Ga'_{01\less \al}$ acts trivially on $(\Sss_M^{01})^{-1}(0)\cap (\im  \io^M_\al)$.  Therefore
there are natural identifications
$$
\bigl((\Sss_M^{01})^{-1}(0) \cap (\im \io^\al_M)\bigr)/\Ga^{01}\;\cong \; \bigl((\Sss_M^{\al})^{-1}(0)\bigr)/\Ga^\al \cong \; \{\al\} \times X\; \subset\; [0,1]\times X.
$$

Thus,  $M^{01}$ is an oriented  branched manifold of dimension $N^{01} + 1$, where $N^{01} = d + \dim E_{A^{01}}$, with boundary that decomposes as a union
\begin{align}\label{eq:EM2}
\p M^{01} =  \sqcup_{\al=0,1} EM^\al   \quad \mbox{ where } \;\;  EM^\al: = \io^\al_M(E_{A^{01}\less A_\al, \,\ve^\al} \times M^\al).
\end{align}
For $\al=0,1$, the branched manifold $EM^\al$  carries a fundamental class $$
\mu_{EM^\al}: = \mu_{E_{A^{01}\less A_\al}} \times \mu_{M^\al}.
$$
Because the isotropy group of the boundary chart labelled $I_\al$ in $\Kk_\al$ equals that of the corresponding cobordism chart in $\Kk^{01}$, the equation \eqref{eq:fundc} is consistent with  the boundary map in the long exact sequence \eqref{eq:homles} for the pair $(M^{01}, \p M^{01})$.
Hence the proof of Lemma~\ref{le:fundM}  adapts to show that the interior of $M^{01}$ also carries a fundamental class 
\begin{align}\label{eq:fundW}
\mu_{M^{01}}\in \check{H}^\infty_{N^{01}+1}(M^{01}\less \p M^{01})
\end{align}
 such that
\begin{align*}%\label{eq:fundW1}
\p \bigl(\mu_{M^{01}}\bigr) = (\mu_{EM^1}, - \mu_{EM^0}) \in   \check{H}^\infty_{N^{01}}(EM^{0})\oplus  \check{H}^\infty_{N^{01}}(EM^{1}) \cong
 \check{H}_{N^{01}}(\p M^{01}),
\end{align*}
where $\p$ is the boundary map in the long exact sequence in \eqref{eq:homles}.

We now apply the cap product in \eqref{eq:commcap}  with 
\begin{align*} & Y = M^{01}, \quad 
 U =  (\Sss_M^{01})^{-1}(E_{A^{01}}\less \{0\}) \;\subset\; M^{01},   %\\  \notag
\quad  A = \sqcup_{\al=0,1} EM^\al. %\\  \notag
\end{align*}
Then $Y\less U = (\Sss_M^{01})^{-1}(0) $ is compact with a natural projection to $[0,1]\times X$ and hence to $X$. Since these maps are proper, any class $b\in \check{H}^d(X)$ pulls back to a class $b_Y\in \check{H}^d(Y\less U)$ such that $\io^*(b_Y) = b_A$ where $\io:A\to Y$ is the inclusion, and $b_A = (b_{0}, b_1)$, where $b_\al$ can be identified with the pullback of $b$ to $(\Sss_M^\al)^{-1}(0) \subset M_\al$.  Hence the cap product $$
(\p \mu_{M^{01}})\cap b_A \in \check{H}^c_{N^{01}}(A, U\cap A)
$$
is in the image of the map $\p'$ in \eqref{eq:commcap} and hence vanishes when pushed forward to $\check{H}^c_{N^{01}}(Y, U)$.  But there is a commutative diagram
\[
 \xymatrix
 {
 (\p\mu_{M^{01}}) \cap b_A \ar@{|->}[d]_{\io_*} \hspace{-.3in}&\in &\hspace{-.3in}  \check{H}^c_{N^{01}} (A, U\cap A) \ar@{->}[d]^{\io_*}  \ar@{->}[r]^{ \Sss_M\;\;\;\;\qquad} & \check{H}^c_{N^{01}} ( E_{A^{01}},  E_{A^{01}}\less \{0\})  \ar@{->}[d]^{=}\\
 0\hspace{-.1in}&\in & \hspace{-.3in}    \check{H}^c_{N^{01}} (Y, U) \ar@{->}[r]^{ \Sss_M\;\;\;\qquad\;\;} & \check{H}^c_{N^{01}} ( E_{A^{01}},  E_{A^{01}}\less \{0\}) .
}
\]
Hence $(\Sss_M)_*\bigl(  (\p \mu_{M^{01}})\cap b_A\bigr) = 0$.  
Since $(\p\mu_{M^{01}}) \cap b_A $ measures the difference between the two classes 
$\mu_{EM^{\al}}\cap b_\al $,  these classes  have the same image in 
$\check{H}^c_{N^{01}} ( E_{A^{01}},  E_{A^{01}}\less \{0\})$, as claimed.
\MS

\NI {\bf Step 3:} {\it Agreement with previous definition in the smooth case.}
It remains to show that in the smooth case the class $[X]^{vir}_\Kk$ constructed here agrees with that constructed in \cite[\S3]{MWiso}. 
The idea there was to construct a small smooth perturbation functor\footnote
{
\, for notation see  \eqref{eq:TVIJ}}
$$
\nu = (\nu_I): \bB_\Kk|_\Vv^{\less \Ga}\to \bE_\Kk|_\Vv^{\less \Ga}
$$
 such that $s_I + \nu_I$ is transverse to zero for all $I$, and then assemble the resulting zero sets $Z^\nu_I: = (s_1+\nu_I)^{-1}(0)\subset V_I$ into a weighted branched manifold 
$Z^\nu: = |\Hat Z_\Hh^\nu|$.  
Note that $Z^\nu$ is oriented and has a weight-preserving natural inclusion into $M$, i.e. each branch of    $Z^\nu$ is a submanifold of  a branch of $M$ with the same weight.
Now choose a sequence $\nu_k$ of perturbation sections with  $\|\nu_k\|\to 0$. There is a corresponding nested sequence of  neighborhoods $B_{\eps_k}(\io_\Kk(X))$ of the zero set $X\cong \io_\Kk(X) \subset |\Vv|$  with intersection equal to $\io_\Kk(X)$. 
Then the zero sets $Z^{\nu_k}$ map to $B_{\eps_k}(\io_\Kk(X))\subset |\Vv|$, and we showed in \cite[Thm.~3.3.5]{MWiso} that for  all $\ell>k$  
the two branched manifolds $Z^{\nu_\ell},  Z^{\nu_k}$ are cobordant in  $B_{\eps_k}(\io_\Kk(X))$ and hence represent the same homology class in $B_{\eps_k}(\io_\Kk(X))$.  It  follows
from the tautness property of rational \c{C}ech homology (see \S\ref{ss:app}(h\,$'$)) that the inverse limit of this sequence of classes in $B_{\eps_k}(\io_\Kk(X))$ determines a unique element 
of $\check{H}_d (\io_\Kk(X);\Q) \cong \check{H}_d (X;\Q)$ that we called  $[X]^{vir}_\Kk$ and showed to be  independent of all choices.

We now interpret this construction in the current setting. As above, fix  a compact neighborhood\footnote
{
\, One important difference between $|\Vv|$ and $M$ is that the zero set $|\s|^{-1}(0)$ does {\it not} have compact neighborhoods in $|\Vv|$
by \cite[Ex.~6.1.11]{MW2}, while it does in the branched manifold  $M$.}
 $\ov{\Nn_0}$ of   $\Sss_M^{-1}(0)$, so that
$$
 \de_0: =   \inf \bigl\{\|\Sss_M(x)\|: x\in \Fr{\Nn_0}: = \ov{\Nn_0}\less \Nn_0\bigr\} > 0,
$$
and choose a nested sequence $\ov{\Nn_k}$ of compact neighborhoods of $\Sss_M^{-1}(0)$ such that
\begin{align*} 
&  \textstyle{ \bigcap _k}\; \ov{\Nn_k}  = \Sss_M^{-1}(0), 
\quad  \Sss_M(\ov \Nn_k) \subset E_{A, \de_k}, \;\;\mbox { where } \de_{k+1}<  \de_k< \de_0.
  \end{align*}
Choose  a corresponding sequence of transverse perturbation sections
$\nu_k = (\nu_{k,I})$ 
such that the perturbed zero set $(s_I + \nu_{k,I})^{-1}(0)$ is contained in $V_I\cap \pi_\Kk^{-1}(\Nn_k)$ for all $I$,
and  for each $k$, consider the  map 
$$
\Hat{\nu}_k: M\to E_A,  \quad \Hat{\nu}_k \bigl(\pi_I(e_{A\less I},x)\bigr)= \nu_k(x) \in E_I\subset E_A.
$$
This is well defined because $\nu_k: \bB_\Kk|_\Vv^{\less \Ga}\to \bE_\Kk|_\Vv^{\less \Ga}$ is a functor.
Then
\begin{align*}
\pr_{E_{A\less I}} \bigl( (\Sss_M + \Hat{\nu}_k) (\pi_I(e_{A\less I}, x))\bigr) &  = \pr_{E_{A\less I}} \bigl( \Sss_M  (\pi_I(e_{A\less I}, x))\bigr)  \ne 0 \mbox{ if }e_{A\less I}\ne 0, \;\;\mbox{ while}\\
\pr_{E_{I}}\bigl(  (\Sss_M + \Hat{\nu}_k) (\pi_I(e_{A\less I}, x))\bigr) & = (s_I + \nu_{k,I}) (x).
\end{align*}
Therefore we may identify the weighted branched manifold $Z^{\nu_k}$ with the perturbed zero set
$$
(\Sss_M + \Hat{\nu}_k)^{-1}(0)\;\subset\; \Nn_k\;\subset\;  \Sss_M^{-1}(E_{A,\de_k}).
$$
Given $b \in \check{H}^d(X;\Q)$,  choose a sequence $b_k\in \check{H}^d(\ov{\Nn_k};\Q)$  such that
  $
  \underset{{\leftarrow}}\lim\; b_k = \psi^*(b),
  $
  where $\psi: \Sss_M^{-1}(0) \to X$ is the footprint map, and let $\io_k: Z^{\nu^k}\to \Nn_k$ be the inclusion.  We must show that
  $$
{\textstyle \lim_k}\; \; \langle  \mu_{Z^{\nu_k}},\, \io_k^*(b_k) \rangle =  (\Sss_M)_* (\mu_M\cap \psi^*(b)) \in \Q.
 $$
  Consider the diagram below, in which the top and bottom square commute while the middle homotopy commutes:
  \[
 \xymatrix
 {
 (M, M\less \Sss_M^{-1}(0))  \ar@{->}[r]^{\; \Sss_M} & (E_A, E_A\less \{0\}) \\
 (M, M\less \ov \Nn_k) \ar@{ ->}[u]^{\io} \ar@{<->}[d]_{=}  \ar@{->}[r]^{\; \Sss_M } & (E_A, E_A\less \{0\}) \ar@{<->}[u]_{=}\ar@{<->}[d]^{=}\\
  (M, M\less \ov \Nn_k)  \ar@{->}[d]_{\io} \ar@{->}[r]^{ \Sss_M + \nu_k\;} & (E_A, E_A\less \{0\})\ar@{<->}[d]^{=} \\
 (M,M\less Z^{\nu_k})  \ar@{->}[r]^{\Sss_M + \nu_k\;} & (E_A, E_A\less \{0\}).
}
\]
Because $Z^{\nu_k}$ is a weighted branched smooth submanifold of $M$ with orientation and weights compatible  with that of $E_A$ and $M$, its fundamental class 
$\mu_{Z^{\nu_k}}$ satisfies
\begin{align*}%\label{eq:muk}
 \mu_{Z^{\nu_k}}& = \mu_M\cap \bigl(( \Sss_M + \nu_k)^*(\oo_E)\bigr)  \in H_d(Z^{\nu_k}, \Q),
\end{align*}
where $\oo_E\in H^{\dim E_A}(E_A, E_A\less \{0\})$ is the natural generator.\footnote
{\, Note that we can use singular homology since we can assume that $Z^{\nu_k}$ and $M$ are simplicial complexes by \cite{Mbr}.}
This immediately implies that  $$
 \big\langle  \mu_{Z^{\nu_k}},\, \io_k^*(b_k) \big\rangle =  \big\langle ( \Sss_M + \nu_k)_*( \mu_M \cap \io_k^*(b_k)), \oo_E  \big\rangle\; \in \Q .  
  $$
Now note that the commutativity of  the above diagram implies  that   $$
  \underset{{\leftarrow}}\lim\; ( \Sss_M + \nu_k)_*( \mu_M \cap \io_k^*(b_k)) = 
( \Sss_M)_*( \mu_M \cap \psi^*(b))\; \in \; H_{\dim E_A}(M, M\less \Sss_M^{-1}(0)).
$$
The result follows. \end{proof}

With a little more work, we can prove that our construction extends to atlases for compact pairs $(W,X)$ as in
\cite[Lemma~5.2.4]{P}.  The following lemma defines  %define $ [W]^{vir}_\Kk$ as an element of 
$$
[W]^{vir}_\Kk\;\;  \in\;\;  \check{H}_{d+1}^\infty(W\less X) 
 = \Hom (\check{H}^{d+1}(W\less X); \Q),
 $$
 Note that $\check{H}_{d+1}^\infty(W\less X)  = \check{H}_{d+1}^c(W, X) $ by \S\ref{ss:app}~property (g$'$).
% which   is the theory dual to the rational \c{C}ech cohomology
%of the \lq interior' $W\less X$ of $W$. 

\begin{lemma}\label{le:boundvir}
Given an oriented $(d+1)$-dimensional Kuranishi atlas $\Kk^W$ with boundary on  a compact pair $(W, X: = \p W)$, there is 
an associated virtual fundamental class $[W]^{vir}_\Kk \in \check{H}_{d+1}^\infty(W\less X)$ such that 
\begin{align} \label{eq:virb}
\p ( [W]^{vir}_\Kk) = [X]^{vir}_\Kk\;   \in \; \check{H}_{d}^\infty(X) = \check{H}_{d}^c(X).
\end{align}
 where $\p$ is the differential in the long exact sequence \eqref{eq:homles}. In particular, the image of $[X]^{vir}_\Kk$ in 
 $\check{H}_{d}^\infty(W) = \check{H}_{d}^c(W)$ is zero.
 %$\check{H}_{d+1}^\infty(W\less \p W): = \Hom (\check{H}^{d+1}(W\less \p W); \Q)$  is the theory dual to \c{C}ech cohomology.
\end{lemma}
\begin{proof}
We define the notion of  an oriented $(d+1)$-dimensional Kuranishi atlas 
$(\Kk,\p \Kk)$ 
for the pair $(W, \p W)$ by replacing $[0,1]\times X$ by $W$ in the above definition of a cobordism atlas.
Thus we take $\Kk^{01}=:\Kk^W$ to be an atlas  for $W$, $\Kk^1=: \Kk^X$ an atlas for $X$ and $\Kk^0$ to be empty, and assume the obvious analogs of (i) --(iii) above.
Then, given a branched manifold  $(M^X, \La^X)$ constructed from $\Kk^X$,  we may construct a branched manifold $(M^W,\La^W)$ with boundary  
$$
\p(M^W) =  E_{A^{W}\less A^X} \times M^X,
$$
 and extend $\id\times \Sss_{X}$ from $\p(M^W)$ to a map $\Sss_{W}: M^{W}\to E_{A^{W}}$ that satisfies the analogs of equations \eqref{eq:EM1} and  \eqref{eq:EM2} above. Further,  using the fundamental class 
$\mu_M^{W}\in H_{N^W}^\infty(M^W\less \p M^W)$
defined as in \eqref{eq:fundW}, %defines 
% satisfies the analog of 
  we define an element  
%  $ [W]^{vir}_\Kk \in \check{H}^c_{d+1}(W, \p W)$
%by applying the 
\begin{align*}
& [W]^{vir}_\Kk \in \check{H}_{d+1}^\infty(W\less X) \quad\mbox{ by setting }  \\ \notag
&\langle  [W]^{vir}_\Kk,\, b \rangle : = (\Sss_{M^W})_*( \Hat{b})\in \check{H}_{\dim E_{A^{W}}}(E_{A^{W}}, E_{A^{W}}\less \{0\};\Q) \cong \Q,
\end{align*}
where $\Hat b$
%\in \check{H}_{\dim E_{A^{W}}}(M^W \less \p M^W, \Sss_{M^W}^{-1}(E_{A^{W}}\less \{0\});\Q)$
 is defined as follows.    Let
$$
Y = M^W, \quad A = \p M^W, \quad U = \Sss_W^{-1}(E_{A^W}\less \{0\}).
$$
%Y = M^W\less \p M^W, \;\;   U = M^W \less \bigl(\p M^W \cup \Sss_{W}^{-1}(0)\bigr), \;\; 
% Y\less U = \Sss_{M^W}^{-1}(0) \less \p M^W.
% $$  
 Then
the pullback $\psi^* b\in \check{H}^{d+1} (Y\less (U\cup A); \Q)$ %(\Sss_{M^W}^{-1}(0) \less \p M^W);\Q)$ 
of  $b \in \check{H}^{d+1}(W\less X;\Q)$  
determines  $$
\Hat{b} : = \mu_{M}^{W}\cap \psi^*b \; \in \check{H}^c_{\dim E_{A^{W}}}\bigl(Y\less A, U\less A;\Q\bigr)
$$
where 
$
\cap:  \check{H}^\infty_{p+q}(Y\less A) \otimes \check{H}^p(Y\less U) \to  \check{H}^c_{q}(Y, U\cup A)$
is as in \eqref{eq:cap} with $A = \emptyset.$

To prove \eqref{eq:virb}, note that in the following  diagram
(with the same $Y,U,A$)
 \begin{align*} %\label{eq:commcap2}
 \xymatrix
 {
 \check{H}^\infty_{p+q + 1}(Y\less A)  \ar@{->}[d]_{\p} \hspace{-.4in}& 
\otimes& \hspace{-.4in}\check{H}^{p+1}(Y\less (A\cup U))  %\ar@{->}[d]_{\p\otimes i^*}
 \ar@{->}[r]^{\cap\;\;\;}   &  \check{H}^c_{q}(Y\less A, U\less A) 
 \ar@{->}[r]^{\;\;\;\io_*}  
 & \check{H}^c_{q}(Y, U) %\ar@{->}[d]^{\io_*}
 \\
 \check{H}^\infty_{p+q}(A)   \hspace{-.3in} & \otimes& \hspace{-.3in} \check{H}^p(A \less U)  \ar@{->}[u] ^{\de} \ar@{->}[r]^{ \cap\;\;\;\;\;}  &
  % \check{H}^c_{q}(A, A\cap U) \ar@{->}^= 
%\hspace{.4in}
 \check{H}^c_{q}(A, A\cap U) \ar@{->}[ur]^{j_*} & %X
}
\end{align*}
%where now $Y = M^W, A = \p M^W, U = \Sss_W^{-1}(E_{A^W}\less \{0\}) $, 
we have
$$
j_*\bigl( (\p \mu_M^W)\cap b' \bigr) = \io_*\bigl(\mu_M^W \cap (\de b')\bigr) \;\;  \in\;\;  \check{H}^c_{q}(Y, U),
$$
for all $\mu \in  \check{H}^\infty_{p+q + 1}(Y\less A) $ and $b'\in  \check{H}^p(A \less U)$, where $\de$ is as in \eqref{eq:Ales}.\footnote{
\, This extension to property (B5)  on \cite[p.336]{Ma} holds by combining Properties (B4) and (B6).}\;
Since  $ \psi^*$ commutes with $\de$, this implies
$$
\bigl\langle \p ( [W]^{vir}_\Kk), b\bigr\rangle  = \bigl\langle  [W]^{vir}_\Kk, \de b\bigr\rangle, \quad\forall\; b\in   \check{H}^{d}(X).
$$
The result follows.
%If $\be' = \psi^*(\be)$ for $\be\in \check{H}^{d}(X)$, then  $\p'\be' = \psi^*(\p'\be)\in \check{H}^{d+1}(W\less X)$ by the naturality of the long exact sequence \eqref{eq:homles}.  Hence $\p ( [W]^{vir}_\Kk) = [X]^{vir}_\Kk$ as claimed.
\end{proof}

\section{Further details and constructions}\label{s:det}

In \S\ref{ss:Y} we first define the notion of a compatible shrinking $(\Uu,\ve)$ and 
prove Proposition~\ref{prop:Y}.  We then introduce the more intricate notion of a  compatible reduction $(\Vv,\ve)$,
which involves not only the compatibility of $\Vv$ with a set of constants $\ve$ but also its compatibility with a suitable cover 
of the set of overlaps in $|\Vv|$, properties that are essential 
 for the proof in  \S\ref{ss:col} 
 that  $Y_{\Vv, J,\, \ve}$ has a collar that satisfies  the conditions listed in
Proposition~\ref{prop:col}.  

\subsection{Shrinkings and the manifold $Y$}\label{ss:Y}

We  assume given an ambient preshrunk tame\footnote
{
\, For terminology see \S\ref{ss:defs}. }
 atlas $\Kk^\Om$ with chart domains $\Uu^\Om$, together with
a tame shrinking $\Uu^\infty \sqsubset \Uu^\Om$, 
and then choose a further shrinking $\Ff^0$ of the footprints $\Ff^\infty$ of $\Uu^\infty$.
 For short we write $\psi^{-1}(\Ff^0)\sqsubset \Uu^\infty\sqsubset \Uu^\Om$.  By the submersion axiom and the
 precompactness of $\TU_{IK}^\infty$ in $\TU_{IK}^\Om$ 
for each $I\subsetneq K$, we may  choose a finite set of points ${z_\al} \in \TU_{IK}^\Om$, constants $\eps_\al>0$, and 
 $\Ga_{z_\al}$-equivariant local homeomorphisms
\begin{align}\label{eq:dom0}
\phi^E_{IK,z_\al}: E_{K\less I, \eps_\al}\times \TW_{IK,z_\al}\to  U_K^\Om,\qquad  1\le \al\le A_{IK},
\end{align} 
where $ \TW_{IK,z_\al}\subset \TU_{IK}^\Om$ is open, such that
\begin{align}\label{eq:dom1}
&s_{K\less I}\circ \phi^E_{IK,z_\al}(e,y) = e, \quad \mbox{ and } \\ \notag
&\TU^\infty_{IK} \;\; \subset\;  \bigcup_{1\le \al \le A_{IK}}  \TW_{IK,z_\al} \subset\; \TU_{IK}^\Om,\qquad \forall\; I\subsetneq K.
\end{align} 
We may and will assume that each $\phi^E_{IK,z_\al}$ is $\Ga_K$-equivariant.   (To do this, 
first shrink the $\TW_{IK,z_\al}$ so that they have disjoint images under the group $\Ga_K/\Ga_{z_\al}$, and then replace them by their orbit under $\Ga_K/\Ga_{z_\al}$.)  
\MS

Our first task is to make a preliminary choice of shrinking so that the space $Y_{\Uu,J,\,\ve}$  is a manifold with boundary.
In the following definition,  condition  (b)  ensures that the charts are compatible with a fixed shrinking of the zero sets, while conditions (a) and (c) have already been used in the proof of Corollary~\ref{cor:bary}.  
Some version of condition (d) is an essential ingredient in the proof that $Y_{\Uu,J,\,\ve}$ is a manifold with boundary; see \eqref{eq:phiYI} below.  
As we saw in \eqref{eq:deIJ} and \eqref{eq:YIJ}, 
the  elements $(e,x;t)\in \p_{J\less I} Y_{\Uu,J,\,\ve}$ have $x\in \TU_{IJ}$ and  $\|e\|< \ka \eps_{I(x)}$ where 
$I(x)\subset I \subsetneq J$.  Therefore in order for the domain of the local chart in  \eqref{eq:phiYI}  to include all the boundary points of $Y_{\Uu,J,\,\ve}$,  we do need  the map $\phi^E_{IK,\eps}$ to be defined using a constant $\eps$ that is 
$>\ka\eps_I$, and we have chosen to use $(\ka+1)\eps_I$ for convenience and precision.  Note also that we do not insist that the image of the map 
$\phi^E_{IK}$ in \eqref{eq:phiE01} below is contained in $\Uu$ or even in $\Uu^\infty$.  Such a requirement comes later; see \eqref{eq:phiIHH}.

 \begin{defn}\label{def:compat}  Given $\psi^{-1}(\Ff_0)\sqsubset \Uu^\infty\sqsubset  \Uu^\Om$ as above,  we will say that a shrinking $\Uu$ and set of positive constants $\ve: = (\eps_K)_{K\in \Ii_\Kk}$ are {\bf $(\Gg_0, \Uu^\infty)$-compatible} 
 if the following holds:
 \begin{enumerate}\item[\rm (a)]   $0<\ka\, \eps_I < \eps_K$ if $I\subsetneq K$  (see \eqref{eq:ve});\vspace{.05in}
   \item[\rm (b)]    $
 \psi^{-1}(\Ff^0)\sqsubset  \Uu\sqsubset \Uu^\infty$; \vspace{.05in}
 \item[\rm (c)] 
 $s_I(\ov{U_I})\subset E_{I,\eps_I}$ for all $I$;\vspace{.05in}
    \item[\rm (d)]  for all $I\subsetneq K$, each $ z\in \TU_{IK}\subset U_K$ has a neighborhood $\TOo_{IK}\subset \TU_{IK}$ such that one of the homeomorphisms $\phi^E_{IK,z_\al}$ in \eqref{eq:dom0} restricts to give a 
 map
\begin{align}\label{eq:phiE01}
  \phi^E_{IK}:& E_{K\less I,(\ka+1)\eps_I}\times \TOo_{IK}\to U^\Om_K
\end{align} 
  that is a homeomorphism to its image, where $\ka: = \max \{|K|: K\in \Ii_\Kk\}$.
  \end{enumerate}
    For simplicity we call  the  pair $(\Uu,\,\ve)$ a {\bf compatible shrinking}. 
\end{defn}
    
    \begin{lemma}\label{le:Ueps} Suppose given  $\psi^{-1}(\Gg_0)\sqsubset \psi^{-1}(\Ff_0)\sqsubset \Uu^\infty\sqsubset  \Uu^\Om$ as above.  Then there is  an
 $(\Ff_0, \Uu^\infty)$-compatible shrinking $(\Uu,\,\ve)$. 
  \end{lemma}
 
 \begin{proof} First choose any tame shrinking $\Uu'$ such that  $\psi^{-1}(\Ff^0)\sqsubset  \Uu'\sqsubset \Uu^\infty$, which is possible by \cite[Prop.~3.3.5]{MW1}.  Then each set $U_I'$ is covered by a finite number of the sets $\TW_{IK,z_\al}$ in \eqref{eq:dom0} and we  choose 
 any set of constants $\ve$ satisfying (a), and also so that $\eps_I<\frac {\eps_\al}{\ka+1}$ for all relevant $\al$.  Then, if we define
 $U_I: = U_I'\cap s_I^{-1}(E_{I, \eps_I})$, property (d) holds.  
Further, $\Uu: = (U_I)$ is  a tame shrinking of $\Uu^\infty$  because
the coordinate changes commute with the section maps $s_I$ and preserve the norms $\|\cdot \|$ on $E_A$.  (More precisely
\begin{align*}
\Hat\phi_{IK}\circ s_I \circ \rho_{IK} & =  s_K: \TU_{IK}\to E_K, 
\end{align*}
where the canonical inclusion
$\Hat\phi_{IK}: E_I\to E_K $  preserves $\|\cdot \|$, i.e. $\|\Hat\phi_{IK}(e)\| = \|e\|$.)   Hence $\Uu$ satisfies (c) and (b), as required. 
\end{proof}

From now on, we fix $(\Ff^0, \Uu^\infty)$, and hence cease to refer to them explicitly.

\begin{lemma}\label{le:Y}  
If $(\Uu, \ve)$  is  compatible, 
then for each $J$, $Y_{\Uu,J,\, \ve }$ is a manifold  of dimension $D: =  d+ \dim E_A+|J|-1$, with boundary equal to
\begin{align*}
 Y_{\Uu, J,\, \ve }\cap \pr_\De^{-1}(\p \De)& ={\textstyle \bigcup_{I\subsetneq J} \p_{J\less I}  Y_{\Uu, J,\, \ve} }\\
 &  = 
{\textstyle  \bigcup_{I\subsetneq J} \bigl\{(e,x;t): x\in \TU_{IJ}, t\in \p_{J\less I}\De_J\bigr\}.}
\end{align*}
 \end{lemma}
 
\begin{proof}   We  show that each point $(e,x;t)\in Y_{\Uu,J,\, \ve }$ has a neighborhood homeomorphic to an open subset of $(\R_{\ge 0})^k\times \R^{D-k}$, where  
 $k = \# \{j\in J\ \big|\ t_j = 0\}$.  Thus, the projection $\pr_\De:  Y_{\Uu, J,\, \ve }\to \De_J$ is compatible with
 the boundary structure of $\De_J$.

First consider a point $(e,x;t)\in Y_{\Uu,J,\,\ve}$ with $t_i\ne 0$ for all $i\in J$.  Then the coordinates $e_j, j\in J,$ are determined by $(x,t)$ 
via the requirement $s_J(x) = t\cdot e|_J$ while the components of $e|_{A\less J}: = (e_i)_{ i\in A\less J}$ can vary freely.    Hence
the tuple $(e,x;t)$ is uniquely determined by the point $(e|_{A\less J}, x;t) \in E_{A\less J} \times U_J\times \intt\De_J$, and so has a manifold neighborhood of dimension $D$.

It remains to define boundary charts at the points $(e^0,x^0,t^0)\in Y_{\Uu,J,\eps}$ with
$$
t^0\in \p \De_J = {\textstyle \bigcup_{I\subsetneq J} \p _{J\less I}\De_J =:
\bigcup_{I\subsetneq J} \intt \De_I.}
$$
Suppose first
   that $$
I(x^0): =    \{i: s_i(x^0)\ne 0\} =  \{i: t^0_i>0\} =:  I(t^0) =: I,
   $$ 
   so that  $x^0\in \TU_{IJ}$.  
   By \eqref{eq:dom1}, there is a neighborhood $\TOo$ of $x^0$ in $\TU_{IJ}$ 
   that is contained in one of the sets $\TW_{IJ,z_\al}$ in 
  \eqref{eq:dom0}, and below we denote by $\phi$ the associated map $\phi_{IJ, z_\al}^E$.
     There is a corresponding 
  neighborhood of $(e^0,x^0,t^0)$ in $$
  \p_{J\less I} Y_{\Uu, J,\, \ve} \cap \{(e,x;t): e|_{J\less I} = 0\}$$
given by
\begin{align*}
& \TOo'_{I,J, \, \ve}: = \big\{(e_{A\less J}+ t_I^{-1}\cdot s_I(x),x; t_I) \ \big| \ x\in \TOo, t_I\approx t_I^0, \|e_{A\less J}\|<\ka \eps_I\bigr\} 
 %& \hspace{2in}
 \subset \;
 \p_{J\less I} Y_{\Uu, J,\, \ve}. %\cap \{(e,x;t).
\end{align*}
Now consider the map
 \begin{align}\label{eq:phiYI}  
 & \psi:  E_{J\less I,(\ka+1) \eps_I}\times [0,\de)^{|J\less I|}\times \TOo'_{I,J,\, \ve} 
 %\cap ( E_{A\less  (J\less I)}\times \TOo \times \intt \De_I)\bigr)
  \;\longrightarrow\; Y_{\Uu^\Om,J,\, \ve}, %\quad \mbox{ by }
 \\ \notag
 &\qquad \bigl( e_{J\less I}, r_{J\less I}, (e_{A\less J}+ t_I^{-1}\cdot s_I(x),x; t_I)\bigr)\longmapsto  \\ \notag
&\qquad \qquad \bigl(e_{A\less J}+ e_{J\less I}+ (\la t_I)^{-1} \cdot s_{I}(x') ,\, x'; \,  \la t_I + r_{J\less I}\bigr),\\ \notag
&\qquad \qquad  \mbox { where } x': =  \phi\bigl( r_{J\less I}\cdot e_{J\less I}, x) \mbox{ for } \phi : = \phi_{IJ, z_\al}^E,
\\ \notag &\qquad \qquad  \mbox{ and }  \la: = 1-|r_{J\less I}| = 1-{\textstyle \sum_{j\in J\less I} r_j.}
\end{align}
To see that $\psi$ does have image in $Y_{\Uu^\Om,J,\, \ve}$ for sufficiently small $\de>0$ and $\TOo$, we check the  conditions in \eqref{eq:Y} as follows.
\begin{itemlist}\item
By \eqref{eq:dom1} $$
r_{J\less I}\cdot e_{J\less I} = s_{J\less I}\circ  \phi\bigl( r_{J\less I}\cdot e_{J\less I}, x) = s_{J\less I} (x'),
$$
so that the image $(e,x;t)$ of $\psi$ does satisfy the equation $s_J(x) = t\cdot e$
 if $x'\approx x^0$ and $\de>0$ is sufficiently small.
 \item  
 Next, we  check that 
 $ s_J(x') \in  E_{A,\eps_{I(x')}}$.   To this end, note first that because we started by assuming
 $I(x^0) = I$, the definition of $Y_{\Uu,J,\, \ve }$ implies that  $\|s_I(x^0)\|< \eps_I$.
Second, because  $s_I(x')\approx s_I(x^0)$, 
we have $s_I(x') < \eps_I$  for sufficiently small $\de,\TOo$.  But if 
  $r_{J\less I}\ne 0$  we have
 $I(x')\supsetneq I(x^0)$ so that  
 $\eps_I < \frac 1 \ka \eps_{I(x')}$  by \eqref{eq:ve}.  Therefore
because 
$\la\approx 1$ and
we use the sup norm on the product $E_A$, we have 
$$
 s_J(x') = 
e_{J\less I}+ (\la t_I)^{-1} \cdot s_{J\less I}\circ \phi\bigl( r_{J\less I}\cdot e_{J\less I},x) \in  E_{A,\eps_{I(x')}}
$$ 
  for sufficiently small $\de>0$.
\item Since elements in the domain of  $\psi$ have $\|e_{A\less J}\|<\ka \eps_I < \eps_{I(x')}$, %\in  E_{J\less I,(\ka+1) \eps_I}
elements in its image also satisfy this condition.
\end{itemlist}
%Thus the image is in $Y_{\Uu^\Om,J,\,\ve}$.   
It is now easy to check that $\psi$ is a local homeomorphism that equals  
 the identity map when  $r_{J\less I}  = 0$ since  $\phi(0,x) = x$.  Hence its restriction to a suitable open subset
  of its domain provides a
 local boundary  chart for $Y_{\Uu,J,\,\ve}$ at $(e^0,x^0,t^0)$.

 It remains to consider the case when $I = I(x^0)\subsetneq H = I(t^0)$.  In this case, write
 $t^0 = t^0_I +t^0_{H\less I}$.% where $H =  I(x^0) \subset I = I(t^0)$.
 Then the above formula for $\psi$ must be modified as follows:
Denote the elements of $I(t^0)$ by $t_H' =  t_I'  +t_{H\less I}'$.  Then, for $r_{J\less I} \approx 0$, we define
\begin{align}\label{eq:coord}
& \psi\Bigl( e_{J\less I}, r_{J\less I}, \bigl(e_{A\less J}+ (t_H')^{-1}\cdot s_I(x),x; \ t_H'\bigr)\Bigr)
 = \\ \notag
 & \qquad  \qquad  \quad 
 \Bigl(e_{A\less J}+ e_{J\less I} + (t_I'')^{-1}\cdot s_I(x''),\; x''; \; t_J ''\Bigr)
\end{align}
 where 
 \begin{itemize}
 \item $x$ varies in a neighborhood $\TOo\subset \TU_{HJ}$ of $x^0$; 
 \item  $\la < 1$ is chosen so that $t_J'': = \bigl((t_i)''\bigr)_{i\in J}$  
 has $|t_J''|: = \sum_{i\in J} t_i'' = 1$, where 
 \begin{align*} 
 & t''_i =   \la t_i',\mbox { if }  i\in I, \qquad t''_h =  \la t_h' + r_h \mbox { if }  i\in H\less I,  \qquad 
 t''_j = r_j  \mbox { if }  j\in J\less H
 %\end{array}\right.
 \end{align*}
 \item  $x''=\phi\bigl( t_{J\less I}''\cdot e_{J\less I}, x) \in V_J$.
% \item  and $ e_{J\less I}'' = (t_{J\less I}'')^{-1}\cdot $
 \end{itemize} 
%  and for suitable constant $\la >1$ define $(t^0_I)' = \la' t^0_I \in \De_I$, and $J': = J\less H$.  
 Then one can check as above that $\im\, \psi $ is a neighborhood of $(e^0,x^0,t^0)$ in $Y_{\Uu, J,\,\ve}$.
 This completes the proof.
% of (i).
% 
%Since $Y_{\Vv, J, \,\ve}$ is open in $Y_{\Uu, J, \,\ve}$, claim (ii) is an immediate consequence of (i).
\end{proof}

\begin{cor}\label{cor:Y} Proposition~\ref{prop:Y} holds.
\end{cor}
\begin{proof} Combine Lemmas~\ref{le:Ueps} and \ref{le:Y}.
\end{proof}

\begin{rmk}\label{rmk:colprod}\rm 
Notice that in  \eqref{eq:coord}  the coordinates $r_{H\less I}\in \R^{H\less I}$ parametrize directions %normal to 
tangent to   $\p_{J\less H} Y_{\Uu, J,\,\ve}$, while the coordinates $r_{J\less H}\in \R^{J\less H}$ parametrize the directions normal to the codimension $|J\less H|$-face $\p_{J\less H} Y_{\Uu, J,\,\ve}$. \hfill$\er$
\end{rmk}
\MS

We now  define and construct {\bf compatible reductions}  $(\Vv,\,\ve)$.   In order to prove Proposition~\ref{prop:col}, it turns out that we need more control over the sets $\Oo_{IK}$ in Definition~\ref{def:compat}~(d).  
Indeed, because of the consistency requirements on the collar, 
it  is not sufficient to choose the  $\Oo_{IK}$  separately for each pair $I\subsetneq K$; rather they must be chosen consistently for all pairs as we now describe. Further, because the collar has fixed width and image in $Y_{\Vv,J,\, \ve}$,  the product maps in (d) must have image in $V_K$ rather than in $V_K^\Om$.  Then, as we will see in the first step of the proof of Lemma~\ref{le:col} below, they can be used to provide local collars along the boundary of
$Y_{\Vv,J,\, \ve}$

%
%  This notion may seem overcomplicated, but it is crucial to the proof that the local collars  that we will construct over certain thickenings of the sets $|W_\al|$ below  can be assembled into global collars $c^Y_J$ with the consistency properties listed in Proposition~\ref{prop:col}.

Note first that because the local product structures 
\begin{align}\label{eq:phiHJ0}
\phi^E_{IK,z_\al}: E_{K\less I, \eps_\al}\times \TW_{IK,z_\al}\to  U_K^\Om,\qquad  1\le \al\le A_{IK},
\end{align} 
in \eqref{eq:dom0} are equivariant and satisfy $s_{K\less I}\circ \phi^E_{IK,z_\al}(e,y) = e$,
they descend via $\rho_{HK}$ whenever $I\subsetneq H \subsetneq K$.  More precisely, for such $H$
$$
\phi^E_{IK,z_\al}: E_{H\less I, \eps_\al}\times (\TU_{HK}\cap \TW_{IK,z_\al})\to  \TU_{HK}^\Om = s_K^{-1}(E_H)
$$
is the lift of a well-defined map
\begin{align}\label{eq:phiHJ}
& \phi^E_{IH, \rho_{HK}(z_\al)}: E_{H\less I, \eps_\al}\times \rho_{HK}(\TU_{HK}\cap \TW_{IK,z_\al})\to  U_{H}^\Om.
\end{align}
Before defining the notion of compatible reduction, we describe certain covers of the set $\ol( |\Vv|)$
of \lq overlaps' in $|\Vv|$, which is the image in $|\Vv|$ of the relevant part of the boundary of $\bigcup_J Y_{\Vv,J,\,\ve}$. See Figure~\ref{fig:W}.

\begin{figure}[htbp] %  figure placement: here, top, bottom, or page
   \centering
  \includegraphics[width=2.5in]{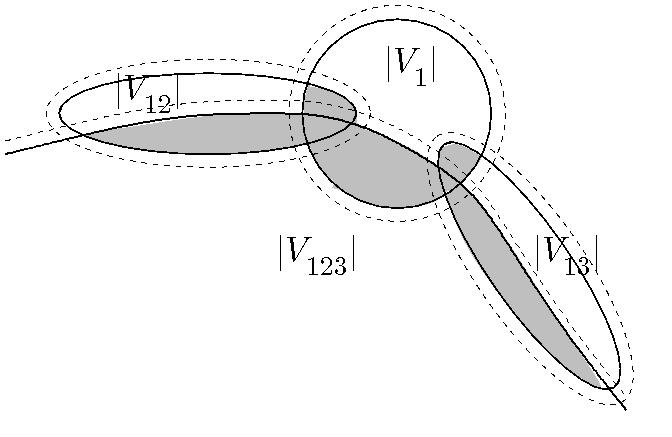} 
   \caption{Here $\ol( |\Vv|)$ is shaded, and the sets $V_I^\Om$ are given by dotted lines; note that  sets $|V_I^\Om|$ and $|V_J^\Om|$ are disjoint unless $I,J$ are nested.}
   \label{fig:W}
\end{figure}

\begin{defn}\label{def:lift} Given a subset $|W|\subset |\Vv^\Om|$ we say that $W\subset V_I^\Om$ is a {\bf lift}  of $|W|$ if
$$
\pi_\Kk(W) = |W|, \quad W = V_I^\Om\cap \pi_{\Kk}^{-1}(|W|),
$$
i.e. $W$ is a \lq full' inverse image of $|W|$ in $\Vv_I^\Om$.
\end{defn}

\begin{lemma}\label{le:W}
If  $(\Uu,\ve )$ is compatible, and $\Vv\sqsubset \Vv^\Om\sqsubset  \Uu $ is any nested reduction, denote by
\begin{align}\label{eq:ol}
\ol( |\Vv|):  = {\textstyle \bigcup_{I\subsetneq K}} |\ov V_{IK}| \;\subset \; |\Vv^\Om|,
\end{align}
the closure of the set of overlaps in $|\Vv|$.  Then we may cover $\ol(|\Vv|)$
 by a finite number of sets $(|W_\al|)_{1\le \al\le N}$, where for each $\al$ there is $\TW_{IK,z_\al}$  as  in \eqref{eq:phiHJ0} such that
$$
W_\al: = \TV_{IK}^\Om \cap \pi_\Kk^{-1}(|W_\al|)\; \subset\; \TW_{IK,z_\al}
$$  
is a lift of $|W_\al|$.
Moreover, we require that $I$ is minimal and $K$ is maximal in the sense that
\begin{enumerate}\item $W_\al$ is an open subset of $\TV_{IK}^\Om$,
\item 
$|V_H|\cap |W_{\al}| \ne \emptyset  \Longrightarrow  I\subset H\subset  K$.
\end{enumerate}
In this situation, we say that $\Vv$ is {\bf adapted to the cover $(|W_\al|)_{1\le \al\le N}$.}
\end{lemma}
\begin{proof}  Choose compatible shrinkings $\Vv\sqsubset \Vv^1 \sqsubset \cdots \sqsubset  \Vv^\ka\sqsubset  \Vv^\Om \sqsubset \Uu $.
Work by downwards induction on $|I| = \ell\le \ka -1 $ so that at the $\ell$th stage we have a covering $(|W^\ell_\al|)_{\al\in B_{\ell}}$ of
$$
{\textstyle  \bigcup_{I\subsetneq K, \ell\le |I|}}\; |V_{IK}|
$$
with lifts $W_\al^\ell$ satisfying (i) and such that (ii) holds if $|H|\ge \ell$. 
When $\ell = \ka-1$, the existence of the finite covering holds by  the precompactness of $|\Vv|$ in $|\Uu|$ while (ii) is easy to arrange  because the sets $|V_H|$ with $|H|=\ell$ are    disjoint.
Now let us suppose that this holds for $\ell+1$ with the sets $(|W_\al^{\ell+1}|)_{\al\in B_{\ell+1}}$ and consider the statement for $\ell$.

The covering $(|W_\al^{\ell}|)$ will consist of sets of  two kinds:
\begin{itemlist}\item  If $|W_\al^{\ell+1}|$ lifts to  $W_\al^{\ell+1}\subset V_{I}^{\Om}$ where $|I|\ge \ell+1$ is as in (i) then we take the set $|W_\al'|$ where 
$$
W_\al': = W_\al^{\ell+1}\less {\textstyle \bigcup_{|H|=\ell} } \, \cl (\TV_{HI}^\ell).
$$
This is open in $V_{I}^{\Om}$ since we have removed a closed set, and satisfies (ii) for $\ell$.
These sets cover 
$$
\bigl({\textstyle  \bigcup_{I\subsetneq K,  \ell+1\le |I|}}\;  |V_{IK}^\ell|\bigr)\; \less \;
\bigl({\textstyle  \bigcup_{H\subsetneq K, |H|=\ell }}\; \cl(|V_{HK}^\ell|)\bigr).
$$
\item  Next add a finite cover of the compact set ${\textstyle  \bigcup_{H\subsetneq K, |H|=\ell }}\; \cl(|V_{HK}^\ell|)$
by sets $|W_\al|$ whose lifts lie in $V_{HK}^\Om $ where  $|H|=\ell$.    These obviously satisfy (ii).
\end{itemlist}
This completes the proof. \end{proof}

 \begin{rmk}\label{rmk:precomp0} \rm  (i)   If $\Vv$ is  adapted to the cover $(|W_\al|)_{1\le \al\le N}$, and $\Vv'\sqsubset \Vv$ is any shrinking, then 
$ \Vv'$ is  also adapted to the cover $(|W_\al|)_{1\le \al\le N}$.

\NI (ii)
 If $I\subsetneq H$ then in general $\TV_{IH}$ is not closed in $V_H$.  Therefore, in order to cover $\ol(|\Vv|)$ by sets $|W_\al|$ that satisfy condition (ii) in Lemma~\ref{le:W} one cannot insist that each set $|W_\al|$ lift to an open subset of some $V_I$, but rather as in Lemma~\ref{le:W}~(i) that it have a lift to an open subset of some $V_I^\Om\sqsupset V_I$.\hfill$\er$
 \end{rmk}

   \begin{defn}\label{def:compatV}  
 Suppose that  $\psi^{-1} \Gg^0 \sqsubset  \Vv^\Om \sqsubset \Uu  $ where $\Gg^0 \sqsubset \Ff$ is a reduction 
 of the footprint cover   (i.e. $G^0_I\sqsubset F  _I,\, \forall I$ and $\bigcup_I G^0_I = X$), and choose 
 a shrinking $\Vv^\infty\sqsubset \Vv^\om$  that  is adapted to the cover
 $(|W_\al|)_{1\le \al\le N}$ where $|W_\al|\subset |\Vv^\Om|$.   
 With these choices fixed, we then say that the pair $(\Vv, \ve)$ is  {\bf precompatible} if  the following conditions hold.
  \begin{enumerate} 
      \item[\rm (a$'$)]    $0<\ka \eps_I < \eps_J$ for all $I\subsetneq J$, 
 \item[\rm (b$'$)]   $\psi^{-1}(\Gg^0)\sqsubset  \Vv\sqsubset \Vv^\infty$; \vspace{.05in}
        \item[\rm (c$'$)]   $s_I(\ov{V_I})\subset E_{I,\eps_I}$  for all $I$;\vspace{.05in}
   \item[\rm (d$'$)]  for all $\al$ with $W_\al\subset \TV_{IK}^\Om$ and $I\subsetneq H\subset K$
  \begin{align}\label{eq:phiIHH}
     \phi^E_{IH,\al} \bigl(E_{H\less I,(\ka+1) \eps_I}\times (\TV_{IH}\cap \rho_{HK}(W_{\al}\cap \TV_{HK})\bigr) \subset V_H.
   \end{align}
  \end{enumerate}
   Further, we say that $(\Vv, \ve)$ is  {\bf compatible} if  it is precompatible and if 
  $(\Vv,\, \ve)\sqsubset(\Vv',\, \ve') $ where  $(\Vv',\, \ve') $ is also  precompatible and
 $\ve\le \ve'$, i.e. $\eps_J\le \eps_J'$ for all $J\in \Ii_\Kk$.
    \end{defn}
 
\begin{rmk}\label{rmk:precomp} \rm  
If $(\Vv,\ve)$ is compatible, so that it is a shrinking of the precompatible $(\Vv',\, \ve') $, then  
we may assume that each set $|W_\al|$ of the associated covering of $|\Vv|$ lifts to some subset $\TV_{IK}'$.  In other words, we can equivalently define  $(\Vv,\ve)$ to be compatible  if
  $(\Vv,\ve)\sqsubset \Vv^\infty$ is precompatible as above, for some  reduction $\Vv^\infty$ that is provided with constants $\ve^\infty\ge \ve$ such that (a$'$) and (c$'$) hold. \hfill$\er$
\end{rmk}
 
The next lemma shows that the hypothesis in Proposition~\ref{prop:col} can be satisfied.

 \begin{lemma}\label{le:Veps} Suppose given  $\psi^{-1}(\Gg_0)\sqsubset \psi^{-1}(\Ff_0)\sqsubset \Vv^\infty \sqsubset  \Vv^\Om$
such that  $\Vv^\infty$ is adapted to the covering $(|W_\al|)_{1\le \al\le N}$, where $|W_\al| \subset |\Vv^\Om|$.
  Then:
 \begin{enumerate}\item
 There is a precompatible shrinking $(\Vv,\, \ve)\sqsubset  \Vv^\infty$. 
 \item  Any precompatible $(\Vv',\, \ve')$ has a compatible shrinking $(\Vv,\, \ve)\sqsubset (\Vv',\, \ve')$.
\end{enumerate}
 \end{lemma}
 
 \begin{proof} The proof of (i) is somewhat similar to that of Lemmas~\ref{le:Ueps} and~\ref{le:W}, 
except that now we have to make sure that (d$'$) holds, i.e. that we can choose $\Vv$  so that
 the image of $  \phi^E_{IH,\al}$  lies in $V_H$  for all $I\subsetneq H \subset K$ rather than just in the fixed ambient space  
 $U_J^\Om$ as in \eqref{eq:phiE01}.   Claim (ii) then follows by the same argument, with  $\Vv^\infty$ replaced by  $\Vv'$.
 \MS
 
 To prove (i), we  first choose any reduction $\Vv^\ka$ of $\Uu$, where $(\Uu,\ve^\ka)$ is compatible, so that 
  $(\Vv,\,\ve)$ satisfies (a$'$),(b$'$),(c$'$).
 We then  work by
downwards induction on $\ell: = |J|$, so that after the $\ell$th stage we have chosen a reduction $(\Vv^\ell,\,\ve^\ell)$ 
with 
$$
\psi^{-1}(\Gg^0)\sqsubset \Vv^\ell\sqsubset \Vv^\ka, \quad \ve^\ell \le \ve^\ka
$$
that satisfies  (a$'$), (b$'$), (c$'$) for all $I,K$, 
and satisfies (d$'$) for all $I$ with $|I|\ge \ell$.
   Since (d$'$) is vacuous when $\ell = \ka$, it suffices to  suppose that 
we have found suitable $(\Vv^{\ell+1},\,\ve^{\ell+1})$ for some $1<\ell + 1 \le \ka$, and consider the construction of
 $(\Vv^\ell,\,\ve^\ell)$.   Our method gives $\ve^\ell$ where $\eps_J^\ell = \eps_J^{\ell+1} $ if $|J|>\ell$ and 
  $\eps_I^\ell \le \eps_I^{\ell+1} $ if $|I|\le\ell$.  Further, for $|J|>\ell$ we construct $V_J^\ell$ by removing some points in $\TV_{IJ}^{\ell+1}$ from $V_{J}^{\ell+1}$ for $|I|=\ell$. Note that removing these points does not affect the validity of (d$'$) for pairs  $I\subsetneq K$ with $|I|\ge \ell+1$.

Choose an intermediate  reduction $\Vv'$ such that $\Vv^0\sqsubset \Vv'\sqsubset \Vv^{\ell+1}$.
Because the subsets $\pi_\Kk(V_I^\infty)\subset |\Kk|$  with $|I| = \ell$ are disjoint, we may work separately with each such $I$.  
Given  $x\in V_I$ with $I \subsetneq K =  I_{\max}(|x|)$
the set 
$\TV_{IK}' = V_K'\cap \pi_\Kk^{-1}(\pi_\Kk(V_I'))$ is precompact in $\TV_{IK}^{\ell+1}=V_K^{\ell+1}\cap \pi_\Kk^{-1}(\pi_\Kk(V_I^{\ell+1}))$ and 
hence there is
 $0<\eps_I^\ell\le \eps_I^{\ell+1}$ so that for each $\al$ with $W_\al\subset V^\Om_I$ and each $I\subsetneq H \subset K$ we have
\begin{align}\label{eq:Wa}
  \phi^E_{IJ}\Bigl(E_{H\less I,(\ka+1)\eps_I}\times ( \TV_{IH}'\cap  \rho_{IH}^{-1}(W_\al)) \Bigr) \subset V^{\ell+1}_H.
\end{align}
For $J$ with $|J|>\ell$ we now define
$$
V_J^\ell : = V_J^{\ell + 1} \less {\textstyle  \bigcup_{I\subset J, |I|=\ell} \bigl(s_J^{-1}(E_I)\cap (V_J^{\ell+1}\less V_J')\bigr)} .
 $$
 Then $V_J^\ell$ is an open subset of $V_J^{\ell + 1} $, since we have removed a closed subset.
 Now choose $\eps_J^\ell$ for $|J|<\ell$ so as to satisfy (a$'$) and then define
  $$
 V_J^\ell: = \bigl\{x\in V_J^{\ell+1} \ | \ s_H(x) < \tfrac 12\eps_J^\ell\bigr\},\quad |J|\le \ell.
 $$
 Then  (c$'$) holds, and (b$'$) still holds for $J$ with $|J|>\ell$ because  it holds for $\Vv'$, and it holds when 
 $|J|\le\ell$ because 
 we did not change the zero sets $s_J^{-1}(0)$.   Moreover (d$'$) holds 
 because when $|J|>\ell$ the only points in $V_J^{\ell+1}$ that were removed to form $V_J^\ell$ 
 lie in $s_J^{-1}(E_I)$ for $I=\ell$.  But this   does not affect the validity of  \eqref{eq:Wa} (and hence \eqref{eq:phiIHH}) because 
 $$
   \phi^E_{IJ}\bigl((E_{J\less I,(\ka+1)\eps_I}\less \{0\})\times \{z\}\bigr) \cap s_J^{-1}(E_I) = \emptyset
   $$
   by the first equation in \eqref{eq:dom1}.
This completes the proof. \end{proof}

\subsection{Construction of the boundary collar}\label{ss:col}
 
It remains to  establish the existence of a collar with the properties stated in Proposition~\ref{prop:col}.
 Recall from \eqref{eq:collDe} that $\De_J$ has a collar of the following form\footnote
 {
 \, Here for the sake of clarity we write $t^\p$ for the coordinate of a general point in $\p \De_J$, while $t$ could be any point in $\De_J$.}
 \begin{align}\label{eq:tde}
 c_{J}^\De: \p \De_J\times [0,\de] \to \De_J,\quad (t^\p,r)\mapsto (1- r|J|)\,t^\p + r |J|\,b_J,
\end{align}
 where $b_J$ is the barycenter of $\De_J$ and $0<\de < \frac 14$: see Figure~\ref{fig:3}.   It is convenient to write
\begin{align*}%\label{eq:nbhdcol}
& \Nn_\de^\De(\p_{J\less I}\De): = 
 \{t\in \De_J \ \big| \  t_j< \de, \  \forall j\in J\less I\}. 
\end{align*}
Notice that 
\begin{align}\label{eq:nbhdcol2}
 c_{J}^\De\Bigl( (\p\De \cap \Nn_\de^\De(\p_{J\less I}\De))\times [0,\de)\Bigr) \subset \Nn_{2\de}(\p_{J\less I}\De);
 \end{align}
i.e.  the width-$\de$ collar of  the corner $\p\De \cap \Nn_\de^\De(\p_{J\less I}\De)$ lies in $\Nn_{2\de}^\De(\p_{J\less I}\De)$.
 We now show that for each $J$ this collar lifts to a (partial) collar for $\p Y_{\Vv, J\, \ve}$ with the properties stated in 
 Proposition~\ref{prop:col}.

 \begin{lemma}\label{le:col} Suppose that $(\Vv,\,\ve)$ is a compatible reduction.  Then for each $J\in \Ii_\Kk$ there is a constant 
 $w_J> 0$, subset $\p'\, Y_{\Vv,J,\,\ve}\subset \p Y_{\Vv, J,\,\ve}$ and map  $c^Y_J$  as in \eqref{eq:coll1}
 with the  properties detailed in  Proposition~\ref{prop:col}.
%\eqref{eq:coll2}, \eqref{eq:collH}, \eqref{eq:coll4}, and \eqref{eq:rescal}.
\end{lemma}

\begin{proof} 

The proof has three steps.  
\MS

\NI {\bf Step I:}  {\it Construction of local collars.}
As in Remark~\ref{rmk:precomp} we will assume that $(\Vv,\ve)\sqsubset (\Vv^\infty,\ve^\infty)$ is precompatible, where  each set  $|W_\al|$ lifts to some $\TV_{IK}^\infty$.  %When needed for clarity, we denote these $I,K$ by $I_\al, K_\al$.
In this step, we fix $\al, I=I_\al,$ and $K=K_\al$, and
 define a local collar of width $w_\al$ over a subset $\Oo^\infty_{K,\al}$ of $\p Y_{\Vv^\infty, K, \ve^\infty}$.
 This subset  is determined by  the set $W_\al\subset \TV_{IK}^\infty$, and is the inverse image of an open subset $|\Oo^\infty_{K,\al}|$ of the set of overlaps 
 $\ol(|\Vv^\infty|)$ in \eqref{eq:ol}.

To this end,  consider  the coordinate chart for $Y_{\Vv^\infty,K,\,\ve^\infty}$ given much as in \eqref{eq:phiYI} by
\begin{align}\label{eq:coll00}
&\psi: 
\; E_{A\less I, (\ka+1)\eps_I} \times W_\al  %\TOo_{IK}^V\times  \Oo^\De_I
 \times [0,\de_\al]^{|K\less I|} \; \longrightarrow\;  Y_{\Vv^\infty,K,\,\ve^\infty},\\ \notag
&\; \bigl( e_{A\less I} , x,  r_{K\less I}\bigr)\longmapsto \bigl(e_{A\less I}+ (\la b_I)^{-1} \cdot s_I(x'),\ x' ;\ \la b_I + r_{K\less I}\bigr),\;\; \mbox{ where }\\ \notag
&\;  \hspace{.5in}  x'= \phi_{IK,z_\al}^E(r_{K\less I}\cdot e_{K\less I},x), \;\; \la: = 1-|r_{K\less I}| =: 1-{\textstyle \sum_{j\in K\less I} r_j}. 
\end{align}
%where $ \TOo'_{IK}\subset \{(e,x)\in E_{I} \times \TV_{IK} \ \big| \ s_I(x) = b_I^{-1}\cdot e_I\}$ is open.
For each $x \in W_\al: = W_\al\cap \TV_{IK}^\infty$, restrict to those $r^\p_{K\less I}$ such that  $$
\la^\p b_I + r^\p_{K\less I}\in \ov\st^\De_K(|x|) \subset \p \De_K,
$$ 
where the superscript $\p$  indicates that the corresponding point lies  in the boundary.
The above map provides coordinates 
\begin{align}\label{eq:Cde}
\Cc^\de: %\bigl((e_{A\less I}, y), r^\p_{K\less I}\bigr) \mapsto \psi
\bigl(e_{A\less I}, x, r^\p_{K\less I}\bigr) \mapsto \psi\bigl(e_{A\less I}, x, r^\p_{K\less I}\bigr) = (e_{A\less I} + e_I'', x''; t^\p)
\end{align}
for an open subset 
\begin{align}\label{eq:OKal}
\Oo_{K,\al}^\infty \subset \bigl\{(e,x;t^\p):  t^\p\in \ov\st^\De_K(|x|),\  t^\p\approx 0\bigr\} % \Bigl(\subset \p Y_{\Vv^\infty,K,\,\ve}\Bigr),
\end{align}
of the boundary $\p Y_{\Vv^\infty,K,\,\ve}$.  We will assume, as we may, that 
$\Oo_{K,\al}^\infty = \pr_V^{-1}(|\Oo_{K,\al}^\infty| )$, where $|\Oo_{K,\al}^\infty| $  is open in 
%which we assume to be the inverse image of an open subset $|\Oo_{K,\al}^\infty| $ of 
$\ol(|\Vv^\infty|)\subset |\Vv|$.
%\vspace{-.2in}

\begin{figure}[htbp] %  figure placement: here, top, bottom, or page
   \centering
   \includegraphics[width=3in]{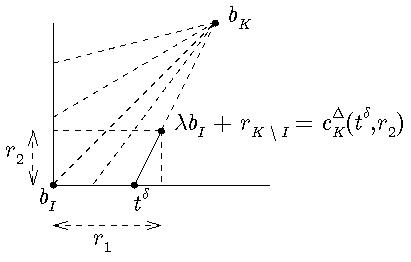} 
   \vspace{-.15in}
   \caption{Here  $K = I\cup\{1,2\}$ and $t^\de$ lies on the boundary with $t_2=0$. Hence $r_{K\less I} = (r_1,r_2)$ where
$r_2$ is the collar coordinate along the ray from $t^\de$ to $b_K$, while $t^\de = c^\De_K(b_I, r)$ for $r = (t^\de)_1$.
   }
   \label{fig:3}
\end{figure}
We  now define a collar over $\Oo_{K,\al}^\infty$  of width $w_\al < \frac 12\de_\al$ (see \eqref{eq:nbhdcol2}) as follows.
Given $$
(t^\p,r)\in \ov\st^\De_K(|x|)\times [0,\de),\; \mbox{ where } \; (t^\p,r)\approx (b_I, 0),
$$
  choose $r_{K\less I}^\p,r_{K\less I}$ (both $\approx 0$) so that
%, with $\la: = 1-|r_{K-1}|,\; \la': = 1-|r_{K-1}'|$ we have
\begin{align}\label{eq:coll5}
t^\p: = \la^\p b_I +  r_{K\less I}^\p,\quad c^\De_K(t^\p, r) =  \la\, b_I + r_{K\less I}, \;\;
\end{align} 
where 
$ \la^\p: = 1-|r^\p_{K-1}|,\; \la: = 1-|r_{K-1}|$; see Figure~\ref{fig:3}.
Then, with $\Cc^\de$ as in \eqref{eq:Cde},  define
\begin{align}\label{eq:coll4$}
&c^{Y}_{K,\al}\; :  \;\; \Oo^\infty_{K,\al} \times [0,w_\al) 
\to  Y_{\Vv^\infty,K,\,\ve},\\ \notag
&
 \bigl( (e_{A\less I} + e_I'', x''; t^\p), r\bigr)\stackrel{(\Cc^\de)^{-1}\times \id}\mapsto \bigl(\, (e_{A\less I}, x, r_{K\less I}^\p),\, r\bigr)
 %= \bigl( e_{A\less I}, \io_{EV}(0,x),\, r^\de_{K\less I}\bigr), \, r\bigr)
% \\
% & \qquad 
\; %\stackrel{\psi}
\mapsto \;  \psi \bigl(e_{A\less I}, x , r_{K\less I}\bigr),
\end{align}
where %$y = \io_{EV}(0,x)$  and 
$r_{K\less I} \in [0,\de)^{K\less I}$ is the function of $r_{K\less I}^\de$ and $\de$  defined in \eqref{eq:coll5}.
In particular, if $|K\less I| = 1$ then $r_{K\less I}$ has only one component, and so is the same as the collar variable $r$, while $t^\de = b_I$.  
Therefore the collar is simply given by $\psi$:
\begin{align}\label{eq:coll7}
&c^{Y}_{I\cup \{j\},\al}\; :  \;\; \Oo^\infty_{I\cup \{j\},\al} \times [0,w_\al) 
\to  Y_{\Vv^\infty,I\cup \{j\},\,\ve},\\ \notag
&
 \bigl( (e_{A\less I} + e_I, x; b_I), r\bigr)
 %= \bigl( e_{A\less I}, \io_{EV}(0,x),\, r^\de_{K\less I}\bigr), \, r\bigr)
% \\
% & \qquad %\stackrel{\psi}
\mapsto \;  \psi \bigl(e_{A\less I}, x , r\bigr).
\end{align}

The next task is to extend the domain of  this collar 
to
\begin{align}\label{eq:stOoK}
\ov\st\bigl(\Oo^\infty_{K,\al}\bigr): = \bigl\{(\mu_H\cdot  (e,x;t) \ |  (e,x;t)\in \Oo_\al,\;  \mu_H\cdot t \in \ov\st^\De_K(|x|)\bigr\}
\end{align}
by rescaling as follows.
%Suppose that the tuples $(t,r)$ and $(1-|r_{K\less I}|) b_I + r_{K\less I}$ are related as in \eqref{eq:coll5}.
Consider a tuple $\mu_H$ (as in \eqref{eq:rescal}),where $I\subset H\subsetneq K$,  and point $t^\de\in  \ov\st^\De_K(|x|)\cap (\{b_I\}\times [0,\de]^{|K\less I|})$
such that $$
\mu_H\cdot t^\p\in \ov\st^\De_K(|x|) \cap  (\{b_I\}\times [0,\de]^{|K\less I|}),
$$
 and let $\mu'_H\cdot $ with $(\mu_H')_i = 1$ for $i\notin H$ give the corresponding rescaling 
in the coordinates $\De_I\times [0,\de_\al]^{|K\less I|}$.   Thus if 
$c^\De(t^\p,r) = (1-|r_{K\less I}|)\, b_I + r_{K\less I}$ 
%we assume that 
%$(t,r)$ and $(1-|r_{K\less I}|)\, b_I + r_{K\less I}$ are related 
as in \eqref{eq:coll5}, we have
\begin{align}\label{eq:rescal5}
c^\De_K(\mu_H\cdot t^\p, r) =  \mu_H'\cdot(\la\, b_I + r_{K\less I}).
\end{align}
Note that this rescaling in the boundary $\p_{K\less H}\De_K$ does not affect the collar variable $r$ along this part of the boundary.
Then the
following diagram commutes, where we write
 $e_I' = (t_I)^{-1}\cdot s_I(x')$, $y: = (e_I, x;\, t_I)\in \p Y$:
\begin{align}\label{eq:rescal1}
\xymatrix
{
%\psi_{I,K,y} \bigl( \la^{-1}(e), x;  \la' t_I +  \la t_{H\less I} ),  t_{K\less H}\bigr)
\bigl( e_{A\less I}, y, r_{K\less I}\bigr)  \ar@{|->}[d]_{\mu_H'\cdot } \ar@{|->}[r]^{\psi\qquad\qquad\qquad } &
 \bigl( e_{A\less I} + e_I',\, x'= \phi(r_{K\less I}\cdot e_{K\less I}, x);\,\la b_I+ r_{K\less I}\bigr) \ar@{|->}[d]_{\mu_H'\cdot }\\
 \bigl((\mu_H')^{-1}\cdot e_{A\less I}, y,  \mu_H'\cdot r_{K\less I}\bigr)  \ar@{|->}[r]^{\psi\qquad\qquad} & \bigl( (\mu_H')^{-1}\cdot (e_{A\less I} + e_I'),\, x' \; ;\, \mu_H'\cdot(\la b_I+   r_{K\less I})\bigr),
}
\end{align}
because the rescaling on the left does not affect the image point $x'= \phi(r_{K\less I}\cdot e_{K\less I}, x)\in V_K$ on the right.
Therefore, because $c^Y_{K,\al}$ is a composite of $\psi^{-1}$ (at $r=0$) with $\psi$, and because rescaling does not affect the collar variable $r$,    the following diagram commutes:
\begin{align}\label{eq:rescal2}
\xymatrix
{
\bigl( ( e_{A\less I}+e_I'', x'';\, t^\p), r \bigr)  \ar@{|->}[d]_{{\mu_H\cdot }} \ar@{|->}[r]^{\;\;\;\;c^Y_{K,\al}} & 
(e', x'; t') \ar@{|->}[d]_{\mu_H\cdot}\\
\bigl( ( \mu_H^{-1}\cdot (e_{A\less I}+ e_I''), x'';\, \mu_H\cdot t^\p), r\bigr)\ar@{|->}[r]^{\qquad\; c^Y_{K,\al}} & 
%\psi\bigr(\mu^{-1}\cdot e_{A\less I}, (e_I,x; t_I), \, \mu\cdot t_{K\less I}\bigr) =
 \bigr( (\mu_H)^{-1}\cdot e', x',  \mu_H\cdot t'\bigr).
}
\end{align}
In other words, if we apply the collar and then rescale (a little) by $\mu_H$, we get the same result as  rescaling by $\mu_H$ and then applying the collar.
It follows that we can unambiguously extend the domain of the local collar to $\ov\st(\Oo^\infty_{K,\al})$ by 
defining 
$$
c^Y_{K,\al}\bigl( ( e_{A\less I}+ e_I'', x'';\, t), r \bigr) : = \mu_H^{-1}\cdot c^Y_{K,\al} \bigr( \mu_H\cdot (e', x',  t')\bigr),
$$
where $\mu_H$ is chosen so that $\mu_H\cdot (e', x', t')$ lies in the domain of the map in \eqref{eq:coll4$}.
Note that $c^Y_{K,\al}$ is equivariant because the maps in 
 \eqref{eq:phiHJ0}  and \eqref{eq:coll00} used to construct it are equivariant.
 \MS
 
Although we assumed in the above construction that $K$ was maximal, so that $W_\al\subset V_{IK}^\infty$ this condition was not used in any essential way in the above construction.
Thus for any $J$ such that $I\subsetneq J \subset K$, by using the map in \eqref{eq:phiHJ} instead of  \eqref{eq:phiHJ0} 
we can define a collar
 $c^Y_{J,\al}$  over 
\begin{align}\label{eq:colY}
c^Y_{J,\al}:\, &  \ov\st(\Oo^\infty_{J,\al})\times [0,w_\al) \to Y_{\Vv^\infty, J,\,\ve^\infty} \;\;\mbox{ where }
\\ \notag
\ov\st(\Oo^\infty_{J,\al}): & = \bigl\{ (e,\rho_{JK}(x),t^\p)\in \p Y_{\Vv^\infty, J, \ve^\infty} \ \big| \ x\in \TV_{JK}\cap W_\al,  (e,x;t^\p)\in \ov\st(\Oo^\infty_{K,\al})\bigr\},
\end{align}
and  $\ov\st(\Oo^\infty_{K,\al})$ is defined in \eqref{eq:stOoK}.
 
 Further we can restrict these collars to the corresponding  subsets $\ov\st(\Oo_{J,\al})$ 
 of $\p Y_{\Vv, J, \,\ve}$ for  all $I\subsetneq J \subset K$, obtaining a set of locally defined collars of width $w_\al$.
 Note that 
 this collar still has width $w_\al$ because we used the constant $\eps_I$ in \eqref{eq:coll00} rather than $\eps_I^\infty$. Hence
although $\ve < \ve^\infty$ in general, when we restrict the domain of $\phi$ in \eqref{eq:coll00} to 
 the points in $\p Y_{\Vv, K,\, \ve}$ 
 the image of $\phi$ lies in $Y_{\Vv, K,\, \ve}$ by condition (d$'$) in Definition~\ref{def:compatV}.

 We claim that  these collars satisfy all the conditions in Proposition~\ref{prop:col}.
In particular,  if $I\subsetneq H\subsetneq K$ the domain of $c^Y_{K,\al} $ contains the image of the collar $c^Y_{H,\al} $ by 
 \eqref{eq:OKal}.  They are compatible with projections and invariant under rescaling   by construction.

The domains  $\ov\st(\Oo_{J,\al})$ of these collars are not open in  $\p Y_{\Vv, J, \,\ve}$ because of the restriction 
  $t\in \ov\st^\De_J(|x|)$, and because the condition that $(e,x;t^\p)\in \ov\st(\Oo_{K,\al})$  places certain extra (but unimportant) restrictions
  on $\|\pr_{E_{K\less I}} e\|$ when $t^\p$ has been rescaled far from $b_I$.
 However, modulo these provisos, for each such $J$ they consist of the full inverse image in $\p Y_{\Vv, J, \,\ve}$
 of the  following open subset $|\Oo_\al|: = |\Oo_{K,\al}^\infty|$  of the \lq boundary' $\p |V^\infty_K|$ of $|V^\infty_K|$:
\begin{align}\label{eq:NNW}
|\Oo_\al|: = |\Oo_{K,\al}^\infty| \;\subset \; \p |V^\infty_K| : = {\textstyle \bigcup_{H\subsetneq K} }
 |V^\infty_{HK}| \subset \ol(|\Vv^\infty|),
%&  |\Nn(W_\al)|  = \bigl\{ |x| \ \big| \ \exists (e,x;t)\in \im \phi, \mbox{ where $\phi$ is as in \eqref{eq:coll00}}\bigr\}.
\end{align}
 where $\Oo_{K,\al}^\infty$ is defined in \eqref{eq:OKal}.\MS

\NI {\bf Step 2:}  {\it Construction of a global collar from a covering by local collars}
 
 We now explain a method from \cite[Prop.~3.42]{Hat} that 
combines  local collars  $$
\bigl(c_\al: \Uu_\al\times [0,w_\al) \to Y\bigr)_{1\le \al \le N}\vspace{.07in}
$$
 defined over 
open subsets $\Uu_\al\subset \p Y$ of the boundary of a manifold $Y$ into a global collar over 
$\p'\, Y$ of width $w$, where $\p'\, Y$  is any precompact subset of $ \bigcup_\al \Uu_\al$ and  $w<\min_\al w_\al/2$.  

To this end, choose a partition of unity $(\la_\al)_\al$ subordinate to the covering of $\p'\, Y$ by the sets $(U_\al)_\al$,
and define
 $$
Y': = Y \cup_\theta \bigl(\p'\, Y \times [-w, 0]\bigr),  
$$
where $\theta$ identifies  $\p'\, Y \times\{0\}$ with $\p' \, Y$ in the obvious way.
We claim that there is a homeomorphism
$$
\Psi: \bigl(Y',  \p'\, Y \times [-w, 0]\bigr) \longrightarrow \Bigl(Y, \; {\textstyle  \bigcup_{\al} } c_\al\bigl(\Uu_\al\times [0,2w)\bigr)\Bigr).
$$
Granted this, we  define the collar by
\begin{align*}%\label{eq:cY}
c^Y:  \p'\, Y \times [0, w) \longrightarrow Y,\quad \bigl(y,r\bigr)\mapsto \Psi_J\bigl(\,y,\, r-w\bigr).
\end{align*}

The homeomorphism  $\Psi$ is a composite
$$
\Psi = \Psi_{N}\circ \cdots \circ \Psi_{1},
$$
of homeomorphisms, 
$$
\Psi_{\ell }: Y'(-1+ \sum_{\al<\ell } \la_\al) \to Y'(-1+ \sum_{\al\le \ell } \la_\al),
$$
where %$(\rho^J_\al)$ is the partition of unity described above and, 
for any function $\si: \p'\, Y  \to [0,1]$
we define
$$
Y'(-1+\si): = Y \cup_\theta \big\{ (y,r)\ \big| \ y\in \p'\, Y , \  (-1+ \si (y))w \le r\le 0\bigr\}.
$$
To define $\Psi_\ell $, first extend the product structure of the external collar $\p Y \times [-w,0] $ via the local collar $c_\ell $ 
to 
%$ c_\al(\Uu_\al \times [0,2w))\subset Y$. %which by \eqref{eq:coll4$}  has a product structure given by
%%via the  local collar of width $w_\al > 2w$ defined in \eqref{eq:colY}.  
%Then  define $\Psi_{J,\ell }$ to have support in this 
obtain an extended collar neighborhood 
$$
\Hat c_{\ell }:  \Uu_\ell  \times [-w, w_\ell ) \to  Y'.
$$
Then define 
 \begin{align*}%\label{eq:collbe2}
\Psi_{\ell }( \Hat c_{\ell }(y,r)) & = \Hat c_{\ell }(y,  f_{y,\ell }(r))
 \end{align*}
 where 
 $$
{\textstyle  f_{y,\ell }:  \bigl[(-1 + \sum_{\al<\ell } \la_\al(y)) w,\ 2w\bigr]\to \bigl[(-1 + \sum_{\al\le \ell } \la_\al(y))w, 2w\bigr]}
 $$
is a homeomorphism that translates by $\la_\ell (y) $ if  $r\le \sum_{\al<\ell } \la_\al(y)) w$.
 This completes the construction.
 
 \begin{rmk}\rm \label{rmk:colllift}  Notice that if each local collar $c_\al$ lifts a map $\pr_\De: (Y,\p Y)\to \bigl([0,1),\{0\}\bigr)$,
 then the global collar does as well; i.e. we have
 $$
 \pr_\De\circ c(y,r) = r.
 $$
 This holds because each $f_{y,\ell}$ is a translation by $\la_\ell(y) w$ on the relevant part of its domain, where $\sum_\ell \la_\ell(y) = 1$. 
Further, if for some map $\pr_E:Y\to E$  we have $c_\al(y,r) = \pr_E(y)$, then the global collar also satisfies $c^Y(y,r) = \pr_E(y)$.
 \hfill$\er$
 \end{rmk}

 \NI {\bf Step 3:} {\it  Completion of the proof.}
 
 Once the cover and partition of unity are chosen, the construction in Step 2 depends only on the ordering of the sets in the cover.
Even though we saw in Step 1 that the local covers satisfy all the compatibility conditions required in  Proposition~\ref{prop:col}, we  will have to 
organize the construction  rather carefully in order to achieve this for the global collars.

 Recall from the discussion of \eqref{eq:piK} that because the atlas $\Kk$ is assumed tame and preshrunk and hence good,
 the subspace topology on $|\Vv^\infty|$ (considered as a subset of $|\Kk|$) is metrizable, and so we may  fix a metric on $|\Vv^\infty|$.
 Since the sets $|V_I|, |V_J|$ have disjoint closures unless $I\subset J$ or $J\subset I$, we may choose  
 \begin{align}\label{eq:de0}
\de_0>0 \mbox{ smaller than half the distance between any  two such sets.}
\end{align}
We next order the sets $|W_\al|_{1\le \al \le N}$ of the cover of $\ol(\Vv)$ so that as $\al$ increases the cardinality $|I_\al|$ of the minimal set 
$I$ in Lemma~\ref{le:W}~(i) increases.    Thus we assume that there are numbers $0 = n_0\le n_1\le  n_2 \le \dots \le  n_{\ka-1} = N$ 
so that
$$
N_{k-1} < \al \le N_k \Longrightarrow |I_\al|= k.
$$
By \eqref{eq:NNW}  the sets $\bigl(|\Oo_\al|\bigr)_{1\le\al \le N}$ cover a neighborhood of the compact subset $\ol(|\Vv|)$ in $|\Vv^\infty|$.   Further  by condition (ii) in Lemma~\ref{le:W} and our choice of $N_k$,  if $\al> N_k$ the set
$|\Oo_\al|$ does not meet any $|V_I|$ with $|I|\le k$.
Hence  we may choose $\de_0>\de_1>0$ so that
for each $k$, 
the sets $\bigl(|\Oo_\al|\bigr)_{1\le\al \le N_k}$ cover the closed  $\de_1$-neighborhood  
$$
\ov{\Nn}\!_{\de_1}(k): = \ov {\Nn}\!_{\de_1}\bigl({\textstyle \bigcup_{|I|\le k, L\in \Ii_\Kk}}  |\ov {V_{IL}}|\bigr) \subset \ol(|\Vv|)
$$
of the compact subset 
${\textstyle \bigcup_{|I|\le k, L\in \Ii_\Kk}} |\ov {V_{IK}}|$.
By shrinking the sets $\Oo_\al$ to $\Oo_\al'$,  we may then assume in addition that  for some $0<\de_2<\de_1$ we have
\begin{align}\label{eq:oo}
\bigl(\al> N_k\bigr)\, \Longrightarrow  |\Oo_\al'| \cap \ov{\Nn}\!_{\de_2}(k)
%\Bigl(\ov{\Nn}_\de\bigl({\textstyle \bigcup_{|I|\le k}} \ov{|V_{IK}|}\bigr)\Bigr) 
= \emptyset,\quad \forall k.
\end{align}

For each $k\le \ka$, choose a partition of unity  $(\la^k_\al)_{1\le \al \le N_k}$ for 
$
\ov{\Nn}\!_{\de_2}(k)
$
 with respect to the covering 
by $(|\Oo_\al'|)_{1\le \al \le N_k}$, such that 
%where, for each $k$,
%$(\la^k_\al)_{1\le \al \le N_k}$ extends $(\la^{k-1}_\al)_{1\le \al \le N_{k-1}}$ in the sense that:
\begin{align}\label{eq:rhokk}
1\le \al \le N_{k-1} \Longrightarrow \la^k_\al = \la^{k-1}_\al
\end{align}
Finally,  choose $w'>0$ such that
\begin{align}
&2w'< {\rm min}_\al w_\al.
%&\pr_{|V|}\bigl(c^Y_{\al,J} (\ov\st(\Oo_\al) \times [0,w')  )\subset \ov{\Nn}_\de\bigl({\textstyle \bigcup_{H\supset I_\al, |H|\le k}} \ov{|V_{HJ}|}\bigr)\quad \forall \al\le N_k, J\subset K_\al,
\end{align}

Now define
\begin{align}\label{eq:dkY}
\p^k Y_{\Vv,J, \,\ve}={\textstyle{\bigcup _{1\le \al\le N_k}} } \bigl\{(e,x;t)\ | \ (e,x;t)\in \ov\st(\Oo'_{J,\al})
%,\, |x|\in \ov{\Nn}_{\de_1}(k)\less  \ov{\Nn}_{\de_2}(k-1)
\bigr\},
\end{align}
where  $\ov\st(\Oo'_{J,\al})$ is defined just as in \eqref{eq:colY} but with $\Oo^\infty_{K,\al}$ replaced by $\Oo^\infty_{K,\al}\cap \pi_\Kk^{-1}(|\Oo_\al'|)$.

Then   for
 each $I\subsetneq J$ with $|I|=k$, we may
use the local collars $c^Y_{J,\al}$ together with the partition of unity  on $\p^k Y_{\Vv,J, \,\ve}$ obtained  by pulling back $(\la^k_\al)$  to construct a collar $$
c^{Y}_{J,k}:\p^k Y_{\Vv,J, \,\ve}\times [0,w_J') \to  Y_{\Vv,J, \,\ve}
$$
 as in Step 2.
Condition~\ref{eq:rhokk} implies that  $c^{Y}_{J,k}$ agrees with  $c^{Y}_{J,k-1}$ on their common domain of definition.  Hence  
the collars fit together to give a well defined collar
\begin{align}\label{eq:colYY}
c^Y_J:\; &  \p'\, Y_{\Vv,J, \,\ve}\times [0,w_J') \to  Y_{\Vv,J, \,\ve}, \;\;\mbox{ where }\\ \notag
 &  \p'\, Y_{\Vv,J, \,\ve} : = \textstyle{\bigcup_{k< |J|} \p^k Y_{\Vv,J, \,\ve}}.
\end{align}
Note that $c^Y_J$ lifts $c^\De_J$ by Remark~\ref{rmk:colllift}. Thus it does have the form required by \eqref{eq:coll1}.

It remains to check that that we can choose collar widths $w_J\le w_J'$ so that the resulting collars have all the required properties.

\begin{itemlist}\item The maps $c^Y_J$ are equivariant, because the local collars are, and the partition of unity is pulled back from $|\Vv^\infty|$.
\item  To see that the  $c^Y_J$ are compatible with projection to $E_{A\less \bullet}$,
suppose that $I\subsetneq J$ has $|I|= k< |J|$.   
Then $c^Y_J$ has the properties in \eqref{eq:coll20}  because all the local collars do.  
Further the points $\io_{EV}(e,x) = (b_I^{-1}\cdot e,x; b_I)$ mentioned in  \eqref{eq:coll2}  lie in $\p^k Y_{\Vv,J, \,\ve}$.  
Therefore $c^Y_J(\io_{EV}(e,x), r)$ is made by combining the local collars $(c^Y_{J,\al})_{\al\le N_k}$.  But we saw in Step 1
that  all these local collars satisfy \eqref{eq:coll2} for $E_{A\less I}$.  It follows that the combined collar formed in Step 2 must also 
satisfy \eqref{eq:coll2} for $E_{A\less I}$.
\item  Similarly, the fact that the relevant local collars that form $c^Y_J$ are invariant under rescaling as in \eqref{eq:rescal} implies
that $c^Y_J$ also satisfies \eqref{eq:rescal}.
\item  To prove that the pairs $(c^Y_J, w_J)$ are compatible with covering maps we need to check two things: \vspace{.01in}
\begin{itemize} \item[{\rm (a)}] that their domains are large enough (i.e. that \eqref{eq:collH} holds for all $I\subsetneq H\subsetneq J$) and
 \item[{\rm (b)}]  that when $H\subsetneq J$ the collar $c^Y_H$ has a natural lift to $Y_{\Vv, J,\,\ve}$. 
 \end{itemize}  \vspace{.03in}
Claim (b) again follows because the local collars used to form $c^Y_H$ (as well as the partition of unity) can be lifted in this way.
(This is just a consequence of equivariance.)
Claim (a) has two parts.  The first claims that 
if $(e,x;t)\in \p'\,  Y_{\Vv,J,\, \ve}$ has $x\in \TV_{IH}\cap \TV_{HJ}$ where $I\subsetneq H \subsetneq J$, then 
$(e,\rho_{HJ}(x);t)$ is in the domain $\p'\,  Y_{\Vv,H,\, \ve}$ of $c^Y_H$. To see this, note that  
$\p'\,  Y_{\Vv,J,\, \ve}$ is the union over $k$ of the sets 
$\p^k\,  Y_{\Vv,J,\, \ve}$ of
 \eqref{eq:dkY}.  But we have
\begin{align*}
 \p^k\,  Y_{\Vv,J,\, \ve} \cap\bigl\{(e,x;t) \ | \ x\in \TV_{IH}\cap \TV_{HJ}\bigr\} & = 
  \p^{|H|}\,  Y_{\Vv,J,\, \ve} \cap\bigl\{(e,x;t) \ | \ x\in \TV_{IH}\cap \TV_{HJ}\bigr\},\\
  &
= \bigl\{(e,x;t) \ | \ (e,\rho_{HK}(x); t) \in    \p^{|H|}\,  Y_{\Vv,H,\, \ve} \},
\end{align*}
where the first equality holds by  \eqref{eq:oo}, while the
second holds because the sets $\ov\st(\Oo^\infty_{J,\al})$ 
 are compatible with the covering maps $\rho_{HJ}$ 
by \eqref{eq:colY}.

The second part of (a) concerns the choice of suitable  widths $w_H\le w_H'$  for all $H\in \Ii_\Kk$.  Since the domains of the collars are 
by now fixed, we can choose each $w_H$ independently:  its choice depends only on the domains of the collars $c_J^Y$ for $J\supsetneq H$.  
Notice that by the definition of the set $\Oo^\infty_{K,\al}$ in \eqref{eq:OKal},  it holds (with $w_H = \frac 12 \de_\al$ for example) for the original domains $\Oo^\infty_{K,\al}$  of the local collars.    Moreover,  because $\de_2<  \de_0$ (where $\de_0$ is the separation distance in
 \eqref{eq:de0}), this property is not affected by
the shrinking from $|\Oo^\infty_{K,\al}|$ to $|\Oo_\al'|$ in \eqref{eq:oo}.    Hence it is easy to see that one can choose suitable $w_H$ for the global collars.

\item  We must check that this collar restricts to any compatible shrinking $(\Vv', \ve')\sqsubset (\Vv,\ve)$.  
But this is immediate since the above construction depends only on the choice of coordinate charts in \eqref{eq:coll00} which restrict to $(\Vv',\ve')$ by the definition of compatibility, and the choice of an appropriate  partition of unity that we can also restrict to $\Vv'$. 

\item Finally we must check that if $\Kk$ is oriented, the collar map $c^Y_J$ preserves the natural induced orientation on its domain and range.   But this is clear from its construction.
 \end{itemlist}
This completes the proof of Lemma~\ref{le:col}.
\end{proof}

\begin{cor}\label{cor:col}  Any reduction $\Vv'$ has a collar compatible shrinking $(\Vv,\ve)$. 
%Moreover, for any given constants $\de_J>0$ we may arrange that the collar width $r_J <\de_J$ for all $J$.
\end{cor}
\begin{proof}  By Definition~\ref{def:collcompat}, it suffices to construct a compatible $(V_J,\eps_J)$ such that 
\begin{itemlist}\item[{\rm (e)}] for all pairs $I\subsetneq J$ we have $\eps_I \le w_J^2$, where $w_J$ is 
the collar width for $V_J$.
\end{itemlist}  
Without loss of generality, let us suppose that $(\Vv',\ve')$ is compatible,  with collars $c^Y_J$ of widths $w_J'$.   As in the proof of 
 Lemma~\ref{le:Veps}  we work  by downwards induction on $|J|$.  Hence at the $k$th stage, we  assume that we have
 compatible $(\Vv^{k+1},\ve^{k+1})$ such that
 condition (e) holds for all $I\subsetneq J$ with $|I| \ge k+1$, and aim to construct  compatible $(\Vv^{k},\ve^{k}, w_J^k)$  so that 
 (e) holds whenever $|I| \ge k.$  As before we take $(V_J^k, \eps_J^k, w_J^k)  = (V_J^{k+1}, \eps_J^{k+1}, w_J^{k+1})$ if $|J|\ge k+1$.  The key point is this:  if 
 we shrink the set  $(V_I^{k+1}, \eps_I^{k+1})$ where $|I|\le k$ by decreasing $\eps_I^{k+1}$ and hence $V_I^{k+1}$ (because of condition (c) in Definition~\ref{def:compat}), then this does not decrease
 the collar width $c^Y_{J,k+1}$ of any $V_J^{k+1}$ with $I\subsetneq J$, since this change only affects points that either lie in  the boundary of 
 $Y_{\Vv^{k+1},J,\, \ve^{k+1}}$ or are interior points  with $I(x)  = \{i | s_i(x)\ne 0\}\subset I$ that do not occur in $\im (c^Y_{J,k+1})$ because of its construction.   Hence it makes sense to choose $\eps_I^k\le \eps_I^{k+1}$  for the elements $|I|=k$ so that  condition (e) holds at level $k$, and then shrink $V_I^{k+1}$ to a set $V_I^{k}$ that satisfies (a,b,c). As usual, this can be done independently for each $I$ at level $k$.   To complete the inductive step, we then  
make appropriate choices for lower level $I$  as in Lemma~\ref{le:Veps}  to obtain a compatible shrinking $(\Vv^k,\ve^k)$ that satisfies (e) at levels $\ge k$.  This completes the proof.   
\end{proof}

\appendix
\numberwithin{equation}{section}

\section
{Rational \c{C}ech cohomology and homology}\label{ss:app}

We briefly describe the properties 
%The fundamental class of $M$ will lie in one 
of the (co)homology theories in  \cite{Ma} that are based on the properties of Alexander--Spanier cochains.   We do not need the full generality of this theory because the space $M = |\bM|_\Hh$ is locally compact and Hausdorff.  
%We shall need the following properties of these theories.
Throughout we assume that $Y$ is  locally compact and Hausdorff, with $A\subset Y$ closed and $U\subset Y$ open, and take coefficients in $\Q$.
Further, we denote these theories by $\check{H}$ to distinguish them from singular (co)cohomology.
\footnote
{\, In \cite[Ch~10]{Ma}  the theory we call $ \check{H}^*$ below is  denoted 
 by  $H^*_c$ to distinguish it from another theory that does not concern us here.}

We  need the following properties of the cohomology theory.

\begin{itemlist}\item [(a)]  (\cite[Thm~3.21]{Ma}) If $Y $ is a connected orientable $n$-dimensional manifold then $\check{H}^i(Y) = 0$ unless $i=n$ in which case $\check{H}^n(Y) = \Q$, i.e. $\check{H}^*$ is like rational singular cohomology with compact supports;
\item[(b)]  (\cite[\S1.2]{Ma})  If $f: A\to Y$ is proper, there is an induced map $f^*: \check{H}^i(Y)\to \check{H}^i(A)$;
\item[(c)] (\cite[\S1.3]{Ma}) if $U\subset Y$ is open, there is an induced map $f_*: \check{H}^i(U)\to \check{H}^i(Y)$.  Further, if $Y$ is as in (a), and $U$ is an open $n$-disc, then $f_*$ is an isomorphism. 
\item[(d)] (\cite[Thm~1.6]{Ma}) if $A\subset Y$ is closed  then there is an exact sequence
\begin{align}\label{eq:Ales}
\cdots   \to  \check{H}^i(Y\less A) \to \check{H}^i(Y) \to \check{H}^i(A)\stackrel{\de} \to \check{H}^{i+1}(Y\less A) \to \cdots,
\end{align}
i.e.  the group 
 $\check{H}^i(A)$ plays the role of the  relative group ${H}^i(Y, Y\less A)$.
\end{itemlist}
\MS

The dual homology theory  developed in  \cite[Ch~4]{Ma} 
is denoted $H_*^\infty$ in \cite[Ch~10]{Ma} to emphasize that it is analogous to  locally finite singular homology; we shall call it  
$\check{H}_*^\infty$.
It follows from the universal coefficient theorem \cite[Thm.~4.17]{Ma} that 
\begin{align}
\check{H}_k^\infty(X) = \Hom\bigl(\check{H}^k(X);\Q \bigr).
\end{align}
Further, because $Y$ is locally compact and Hausdorff, it follows from the uniqueness property for $\check{H}^*_c$ stated in \cite[\S6.7]{Ma}, that the dual theory $\check{H}_k^\infty$ is isomorphic to rational Borel--Moore homology.

As shown by the following, the functorial properties of $\check{H}_*^\infty$ are different from the usual singular theory. 

\begin{itemlist}\item[(a$'$)] If $Y$ is a connected orientable $n$-manifold, then $\check{H}^\infty_i(Y) = 0$ unless $i=n$ in which case $\check{H}_n^\infty(Y) = \Q$; more generally, any orientable $n$-manifold has a fundamental class 
\begin{align}\label{eq:ffcl} \mu_Y \in \check{H}^\infty_n(Y).
\end{align}
\item[(b$'$)] (\cite[\S4.6]{Ma}) if $U\subset Y$ is open, there is an induced restriction 
\begin{align}\label{eq:rhoYU}
\rho_{Y,U}: \check{H}^\infty_i(Y)\to \check{H}^\infty_i(U);
\end{align}
 moreover for $U_1\subset U_2\subset Y$ we have $\rho_{Y,U_1} = \rho_{U_2,U_1}\circ \rho_{Y,U_2}$.
 \item[(c$'$)] (\cite[\S4.6]{Ma})  If  $f: A\to Y$ is continuous and proper, then there is an induced pushforward  $f_*: \check{H}^\infty_i(A)\to \check{H}^\infty_i(Y)$;  moreover, given a   proper inclusion $\io: A\to Y$,  there is a functorial long exact sequence
 \begin{align}\label{eq:homles}
 \cdots\to \check{H}^\infty_i(A)\stackrel{\io_*}\to \check{H}^\infty_i(Y) \stackrel{\rho_{Y,Y\less A}}\longrightarrow 
 \check{H}^\infty_i(Y\less A) \stackrel{\p }\to \check{H}^\infty_{i-1}(A) \to \cdots
 \end{align} 

 \item[(d$'$)]  (\cite[\S4.3 (3c)]{Ma}) If $f:A\to Y$ is proper and $U$ is open in $Y$, then the following diagram commutes:
 \[
 \xymatrix
 {
 \check{H}^\infty_{i}(A)\ar@{->}[r]^{f_*} \ar@{->}[d]^{\rho_{A,A\cap f^{-1}(U)}} &  \check{H}^\infty_{i}(Y)\ar@{->}[d]^{\rho_{Y,U}}\\
  \check{H}^\infty_{i}(A\cap f^{-1}(U))\ar@{->}[r]^{\;\;\;\;\;\;f_*}  &   \check{H}^\infty_{i}(U).
 }
 \]
 \item[(e$'$)]   
(\cite[\S4.9\,(6)]{Ma}) if $Y=U\cup V$ where $U,V$ are open then there is an exact Mayer--Vietoris sequence of the form: 
 $$
 \cdots \to  \check{H}^\infty_{i+1}(U\cap V) \to  \check{H}^\infty_i(Y) \to  \check{H}^\infty_i(U)\oplus  \check{H}^\infty_i(V) \to  \check{H}^\infty_i(U\cap V) \to\cdots
 $$ 
 In particular, if $U$ is the disjoint union of a finite number of sets of $U_i$, then $$
 \check{H}^\infty_*(U) \cong \oplus_i  \check{H}^\infty_*(U_i).
 $$
 \item[(f$'$)] (\cite[p.334]{Ma}) if $U\subset Y$ is open while $A\subset Y$ is closed, there is a cap product
\begin{align}\label{eq:cap}
\cap:  \check{H}^\infty_{p+q}(Y\less A) \otimes \check{H}^p(Y\less U) \to  \check{H}^c_{q}(Y, U\cup A).
\end{align}
 This  takes values in compactly supported \c{C}ech homology, a theory whose functorial properties are analogous to those of
 the usual singular homology.  In particular,  if the triple $(U\cup A; U, A)$ is {\it excisive} for $\check{H}^c$  (i.e. 
 $ \check{H}^c_{q}(A, U\cap A) \cong  \check{H}^c_{q}(U\cup A,U) $), then 
 there is a commutative 
 diagram
 \begin{align} \label{eq:commcap}
 \xymatrix
 {
 \check{H}^\infty_{p+q + 1}(Y\less A) 
\otimes \check{H}^p(Y\less U) \ar@{->}[d]_{\p\otimes (-1)^p\, \io^*} \ar@{->}[r]^{\qquad\quad \cap}   &  \check{H}^c_{q+1}(Y, U\cup A) \ar@{->}[d]_{\de}
 \\
 \check{H}^\infty_{p+q}(A)   \otimes \check{H}^p(A\less U)  \ar@{->}[r]^{\qquad \cap}  & \check{H}^c_{q}(A, U\cap A) .
}
\end{align}
Note that the above diagram exists when  $Y$ is locally compact, $A$ is closed  and $Y\less U$ is compact.  To see this, 
choose a 
nested sequence $\Nn_k$ of precompact open neighborhoods of $Y\less U$ in $Y$ with
$$
Y\less U = {\textstyle \bigcap_k} \Nn_k, \quad U = {\textstyle \bigcup_k}  (Y\less \Nn_k).
$$
Since  by definition
$$
\check{H}^c_{*}(Y, U\cup A) =   \underset{{\leftarrow}}\lim\; \check{H}^c_{*}(Y, (Y\less \Nn_k)\cup A),  \qquad 
\check{H}^c_{*}(A, U\cap A) =   \underset{{\leftarrow}}\lim\; \check{H}^c_{*}(A, A\less \Nn_k),
$$
and the triple of closed sets 
$(Y,Y\less \Nn_k, A)$ is excisive by \cite[Cor.9.5]{Ma}, it follows that  $(Y,Y\less U, A)$ is excisive as required.
\item[(g$'$)] (Exercise 5 on p 272 of \cite{Ma}) If $X$ is  Hausdorff and $X\less A$ is a precompact open subset of $X$, then $\check{H}^\infty_{*}(X\less A) = \check{H}^c_{*}(X,A)$.
\item[(h$'$)]  This homology is {\it taut}; i.e.
 if  $X\subset Y$ is closed, where  $Y$  is locally compact and  Hausdorff, and $N_{k+1}\subset N_k$ is a nested sequence of 
closed  neighborhoods of $X$ in $ Y$, then
 (by \cite[Thm~6.4]{Ma})
 $$
 \check{H}_d^\infty(X) =  \underset {\leftarrow}\lim ( \check{H}^\infty_d(N_k)).
  $$
\end{itemlist}
\MS

\NI {\bf Acknowledgements:}  I thank Mohammed Abouzaid, Helmut Hofer, and Katrin Wehr\-heim for enlightening conversations, and  Eleny Ionel and Brett Parker for very helpful comments on an earlier version of this paper.   I also warmly thank the referee for a very careful reading and many pertinent comments that helped to improve the exposition and the accuracy of some of the lemmas.
Finally, I thank the IAS Women and Mathematics program and the Newton Institute for  Mathematical Sciences for their hospitality during the writing of this paper.


\begin{thebibliography}{CCCCC}


\bibitem[C1]{Castell1} R. Castellano, Smoothness of Kuranishi atlases on Gromov--Witten moduli spaces, arXiv:1511.04350.

\bibitem[C2]{Castell2} R. Castellano, Genus zero Gromov-Witten axioms via Kuranishi atlases,  arXiv:1601.04048.

\bibitem[FF]{TF}  M. Farajzadeh-Tehrani and K. Fukaya, Gromov--Witten theory via Kuranishi Structures, arXiv:1701.07821.

\bibitem[FO]{FO}  K. Fukaya and K. Ono, Arnold conjecture and
Gromov--Witten invariants, {\it Topology}  {\bf 38} (1999), 933--1048.
%

\bibitem[Hat]{Hat}  A. Hatcher, Algebraic Topology, Wiley

\bibitem[H]{H:sur} H.~Hofer,  Polyfolds and Fredholm theory, arXiv:1412.4255, to appear in  {\it Lectures on Geometry}, Clay Lecture Note series, OUP, 2017.

\bibitem[HWZ1]{HWZ} H.~Hofer, K.~Wysocki, and E.~Zehnder, 
       Applications of  Polyfold theory I: The Polyfolds of Gromov--Witten theory,
      arXiv:1107.2097.

\bibitem[HWZ2]{Hbigbook} H.~Hofer, K.~Wysocki, and E.~Zehnder, Polyfold and Fredholm theory,       arXiv:1707.08941.


\bibitem[Ma]{Ma} W. Massey, {\it Homology and Cohomology theory}, Dekker, 1978.

\bibitem[M1]{Mbr} D. McDuff,
 Branched manifolds, groupoids, and multisections, 
 {\it Journal of Symplectic Geometry} {\bf 4} (2007) 259--315.
 
\bibitem[M2]{Mnot} D. McDuff, Notes on Kuranishi atlases, arXiv:1411.4306

\bibitem[M3]{Morb} D. McDuff, Strict orbifold atlases and weighted branched manifolds,  to appear in {\it J. Symp. Geom.}, arXiv:1506.05350


\bibitem[MW1]{MW1}  D. McDuff and K. Wehrheim, The topology of Kuranishi atlases, {\it Proceedings of the London Mathematical Society} (3) {\bf 115} (2017), no 2,  221-292.

\bibitem[MW2]{MW2} D.\ McDuff and K.\ Wehrheim,   The fundamental class of smooth Kuranishi atlases with trivial isotropy, 
 {\it Journal of Topology and Analysis}  Vol 10 (2018).
 
\bibitem[MW3]{MWiso}  D. McDuff and K. Wehrheim, Smooth Kuranishi atlases
        Geometry \& Topology 21-5 (2017), 2725--2809. 

\bibitem[MW4]{MWgw} D.\ McDuff and K. Wehrheim,  work in progress.


\bibitem[P]{P} John Pardon,  An algebraic approach to virtual fundamental cycles on 
moduli spaces of $J$-holomorphic curves,
{\it Geom. Topol.} {\bf 20} (2016), 779-1034.




\bibitem[Y]{Yang}  Dingyu Yang, The polyfold--Kuranishi correspondence I: a choice-independent theory of Kuranishi structures, 
arXiv:1402.7008


\end{thebibliography}
\end{document}